\documentclass[reqno]{amsart}
\usepackage{hyperref}
\DeclareMathOperator{\rand}{rand}
\DeclareMathOperator{\diag}{diag}
\usepackage{amsmath}
\usepackage{mathabx}
\usepackage{subcaption}
\usepackage{graphicx}

\begin{document}

\title[\hfil  PINT for Fourth Order Time Dependent PDEs]
{Diagonalization Based Parallel-in-Time Method for a Class of Fourth Order Time Dependent PDEs}
 
\author[G. Garai. B. C. Mandal \hfilneg]
{Gobinda Garai, Bankim C. Mandal}

\address{Gobinda Garai \newline
School of Basic Sciences,
Indian Institute of Technology Bhubaneswar, India}
\email{gg14@iitbbs.ac.in}

\address{Bankim C. Mandal \newline
School of Basic Sciences,
Indian Institute of Technology Bhubaneswar, India}
\email{bmandal@iitbbs.ac.in} 

\thanks{Submitted.}

\begin{abstract}
In this paper, we design, analyze and implement efficient time parallel method for a class of fourth order time-dependent partial differential equations (PDEs), namely biharmonic heat equation, linearized  Cahn-Hilliard (CH) equation and the nonlinear CH equation. We use diagonalization technique on all-at-once system to develop efficient iterative time parallel methods for investigating the solution behaviour of said equations. We present the convergence analysis of Parallel-in-Time (PinT) algorithms.  We verify our findings by presenting numerical results.
\end{abstract}

\maketitle
\numberwithin{equation}{section}
\newtheorem{theorem}{Theorem}[section]
\newtheorem{lemma}[theorem]{Lemma}
\newtheorem{definition}[theorem]{Definition}
\newtheorem{proposition}[theorem]{Proposition}
\newtheorem{remark}[theorem]{Remark}
\allowdisplaybreaks

\section{Introduction} \label{intro}
In this work we are interested in designing time parallel algorithm for linear and non-linear fourth order time-dependent PDEs based on Waveform relaxation (WR) and the diagonalization technique. The simplest linear fourth order time-dependent PDE is given by
\begin{equation}\label{modelproblem1} 
    u_t =-\Delta^2u,\;  (x,t)\in\Omega\times(0,T],
\end{equation}
and the linearized Cahn-Hilliard equation is given by
\begin{equation}\label{modelproblem2} 
    u_t =-\beta^2\Delta u-\epsilon^2\Delta^2u,\;  (x,t)\in\Omega\times(0,T],
\end{equation}
where $\Omega\subset\mathbb{R}^d, d=1, 2$. The equation  \eqref{modelproblem1} is related to modelling of clamping plate, thin beam etc. The existence of solution and numerical solution of \eqref{modelproblem1} can be seen in \cite{gustafsson1995time} for different choices of boundary conditions. The equation \eqref{modelproblem2} is the linearization of the Cahn-Hilliard (CH) equation about a solution in spinodal region, see \cite{Elliott, grant1993spinodal}, where the parameter $\beta$ takes value in $[-1/\sqrt{3}, 1/\sqrt{3}]$.
We also formulate time parallel algorithm for the nonlinear CH equation.
The CH equation has been suggested as a prototype to represent the evolution of a binary melted alloy below the critical temperature in \cite{Cahn, Hilliard}, it also arises from the Ginzburg-Landau free energy functional: 
$$\mathcal{E}(u):=\int_{\Omega} F(u)+\frac{\epsilon^2}{2}\vert\nabla u\vert^2 , $$
where $F(u) = 0.25(u^2 - 1)^2$ and $\frac{\epsilon^2}{2}\vert\nabla u\vert^2$ is the gradient energy and $\epsilon (0<\epsilon\ll 1)$ is the thickness of the interface. The CH equation with the natural homogeneous Neumann boundary condition takes the form 
\begin{equation}\label{CH}
\begin{cases}
u_t = \Delta v, & (x,t)\in\Omega\times(0,T],\\
v=f(u) - \epsilon^2\Delta u, & (x,t)\in\Omega\times(0,T],\\
\partial_{\nu}u=0=\partial_{\nu}v, & (x,t)\in\partial\Omega\times(0,T], 
\end{cases} 
\end{equation}
where $f(u)=F'(u)$ and $\nu$ is outward unit normal to $\partial\Omega$. The solution $u(x, t)$ of \eqref{CH} is always bounded, thus we assume the following 
\begin{equation}\label{Lipchitz_CH}
\max_u \vert f'(u)\vert \leq M,
\end{equation}
where $M$ is a non-negative constant. The solution of the CH equation involves two different dynamics, one is phase separation  which is quick in time, and another is phase coarsening which is slow in time. The fine-scale phase regions are formed during the early stage and they are separated by the interface of width  $\epsilon.$ Whereas during phase coarsening, the solution tends to an equilibrium state which minimizes
the system energy.
By differentiating the energy functional $\mathcal{E}(u)$ and total mass $\int_{\Omega}u$ with respect to time $t$, we get 
\begin{equation}\label{energy minimization}
\frac{d}{dt}\mathcal{E}(u)\leq 0, \mbox{\hspace{1cm}} \frac{d}{dt}\int_{\Omega}u = 0.
\end{equation}
The relations in \eqref{energy minimization} says that the CH equation describes energy minimization and the total mass conservation while the system evolves. Since the non-increasing property \eqref{energy minimization} of the total energy is an essential feature of the CH equation, it is a key issue for long time simulation that is expected to be preserved by numerical methods. 
The CH equation is a nonlinear equation and is impossible to find its analytical solution but there exists solution of the CH equation as shown in \cite{ElliottZheng, liuzhao}. Various research have been done on finding numerical scheme for the CH equation to approximate the solution with either Dirichlet \cite{DuNicolaides, David} or Neumann boundary conditions\cite{elliott1987numerical, furihata2001stable, shin2011conservative, ElliottZheng, Elliott, StuartHumphries, Christlieb, ChengFeng}. A review on different numerical treatments to the CH equation can be found in \cite{lee2014physical}. The possible application of the CH equation as a model are: image inpainting \cite{EsedoAndrea}, tumour growth simulation \cite{Wise}, population dynamics \cite{cohen1981generalized}, dendritic growth \cite{kim1999universal}, planet formation \cite{tremaine2003origin}, etc.

The above described works are all in time stepping fashion for advancement of evolution of the CH equation. To get a good accuracy of the numerical solution of CH equation one needs the spatial mesh size $h$ is of $O(\epsilon)$ or even finer. Then the discretized equation will result in a very large scale algebraic system that should be solved sequentially for capturing the long term behaviour of the CH equation.   
Consequently, it is of great importance to accelerate the simulation using parallel computation, which can be fulfilled by time parallel technique. In this work, we lay our efforts on designing PinT method for above mentioned evolution equations.
In the last few decades there are lot of efforts on formulating various type of time parallel technique, for an overview see \cite{gander50year}. In this work we use the WR \cite{lelarasmee1982waveform} and diagonalization technique \cite{gander2016direct, gander2017halpern, maday2008parallelization} to get a PinT algorithm. A diagonalization based method was first proposed in \cite{maday2008parallelization}, which rely on reformulating space-time discretization as an \textit{all-at-once} system using tensor product and use diagonalization of temporal discretization matrix. For temporal discretization matrix to be diagonalizable one needs to use variable time step size as given in \cite{gander2016direct}; for equidistant time step the temporal discretization matrix is a Jordan block. The described method in \cite{gander2016direct} for variable time step size suffers form roundoff error and the accuracy of temporal variable. The works in \cite{mcdonald2018preconditioning, goddard2019note} use Strang type block $\alpha$-circulant matrix as a preconditioner for \textit{all-at-once} system and the resulting preconditioned system can be solved using diagonalization technique.  Recently, in \cite{gander2019shulin} Gander \& Wu proposed a time parallel iterative algorithm which uses uniform time step to solve initial value problem. A \textit{periodic-like} initial condition is imposed to get a Strang type $\alpha$-circulant temporal matrix, which is diagonalizable. We use similar approach in this paper.

The rest of this paper is arrange as follows. We introduce in Section \ref{Section2} the time parallel algorithm in one and two spatial dimension for the equation \eqref{modelproblem1}, and study the convergence behaviour. In Section \ref{Section3} we formulate the time parallel algorithm for the equation \eqref{modelproblem2}, and present the convergence result. In Section \ref{Section4} we design the PinT algorithm for the equation \eqref{CH}, and present the convergence result. To illustrate our analysis, the accuracy and robustness of the proposed formulation, we show numerical results in Section \ref{Section5}.

\section{Linear Time-Dependent Biharmonic Equation}\label{Section2}
In this section we consider the linear time-dependent biharmonic equation as given below
\begin{equation}\label{modelproblem1a} 
\begin{cases}
    u_t =-\Delta^2u, & (x,t)\in\Omega\times(0,T],\\
   \partial_{\nu}u=0=\partial_{\nu}(\Delta u), & (x,t)\in\partial\Omega\times(0,T],\\
   u(x,0) = u_0(x), & x\in\Omega,
\end{cases}
\end{equation}
where $\nu$ is outward unit normal to $\partial\Omega$ and $\Omega\subset\mathbb{R}$. The PinT algorithm for \eqref{modelproblem1a}, based on WR iteration and \textit{periodic-like} initial condition, is the following 
\begin{equation}\label{modelproblem1a_wr} 
\begin{cases}
    u_t^k =-\Delta^2u^k, & (x,t)\in\Omega\times(0,T],\\
   u^k(x,0)=\alpha u^k(x,T)-\alpha u^{k-1}(x,T) +u_0(x) , & x\in\Omega,\\
   \partial_{\nu}u^k=0=\partial_{\nu}(\Delta u^k), & (x,t)\in\partial\Omega\times(0,T],
\end{cases}
\end{equation}
where $k\geq 1$ is the WR iteration index and $\alpha$ is the PinT free parameter. Evidently upon convergence of \eqref{modelproblem1a_wr} the term $\alpha (u^k(x,T)- u^{k-1}(x,T))$ is cancelled and therefore the converged solution is the solution of \eqref{modelproblem1a}. We now introduce the diagonalization based implementation of \eqref{modelproblem1a_wr}.
\subsection{Discretization and PinT Formulation}
Using the centred finite difference scheme for spatial derivative in an interval with equidistant grid points and homogeneous Neumann boundary condition as described in \eqref{modelproblem1a}, we have the following semi-discrete system corresponding to \eqref{modelproblem1a_wr}
\begin{equation}\label{semi-discrete1}
  \begin{cases}
    \dot{\textbf{u}}^k(t) + A\textbf{u}^k(t) =0, \\
    \textbf{u}^k(0)=\alpha \textbf{u}^k(T) -\alpha \textbf{u}^{k-1}(T) + \textbf{u}_{0},
  \end{cases}
\end{equation}
where $\textbf{u} = (u^1, u^2,\cdots, u^{N_x})^T\in \mathbb{R}^{N_x}$, and $A=\Delta^2_h \in\mathbb{R}^{N_x \times N_x}$ with 
\begin{equation}\label{deltah}
\Delta_h= \frac{1}{h ^2}
\begin{bmatrix}
    -2  & 2  \\
    1  & -2& 1   \\
    &\ddots&\ddots &\ddots \\
     & & 1  & -2& 1   \\
    & & & 2 & -2    
\end{bmatrix}.
\end{equation}
Employing the linear $\theta$-method to \eqref{semi-discrete1} yields
\begin{equation}\label{linear-theta1}
  \begin{cases}
    \frac{\textbf{u}_n^k - \textbf{u}_{n-1}^k}{\Delta t} + A\left( \theta \textbf{u}_n^k + (1-\theta)\textbf{u}_{n-1}^k\right) =0, & n=1,2,\cdots, N_t \\
    \textbf{u}_{0}^k=\alpha \textbf{u}_{N_t}^k -\alpha \textbf{u}_{N_t}^{k-1} + \textbf{u}_{0},
  \end{cases}
\end{equation}
where $N_t={T}/{\Delta t}$ and $\theta\in[0, 1]$. We consider $\theta=1$ and $\theta=1/2$, which corresponds to Backward-Euler and Trapezoidal  scheme. After gathering all the space-time points in \eqref{linear-theta1}, we have the following \textit{all-at-once} or  \textit{space-time} system
\begin{equation}\label{all-at-once1}
\left( C_1^{(\alpha)}\otimes I_x + C_2^{(\alpha)}\otimes A\right) \textbf{U}^k = \textbf{b}^{k-1}, 
\end{equation}
where $\textbf{U}^k = (\textbf{u}_1^k, \textbf{u}_2^k,\cdots, \textbf{u}_{N_t}^k)^T\in \mathbb{R}^{N_x N_t}, C_1^{\alpha}, C_2^{\alpha} \in\mathbb{R}^{N_t\times N_t}$  and  $\textbf{b}^{k-1}\in\mathbb{R}^{N_t N_x}$ are given by
\begin{equation}\label{ciculantmat}
 C_1^{(\alpha)}= \frac{1}{\Delta t}
\begin{bmatrix}
    1  & & & -\alpha \\
    -1  &  1   \\
    &\ddots&\ddots  \\
    & & -1 & 1    
\end{bmatrix}_{N_t\times N_t},
 C_2^{(\alpha)}= 
\begin{bmatrix}
     \theta  & & & (1-\theta)\alpha \\
      1-\theta &  \theta   \\
    & \ddots &\ddots  \\
    & & 1-\theta & \theta       
\end{bmatrix}_{N_t\times N_t}, 
\end{equation} 
\[
\textbf{b}^{k-1}=\left( (\textbf{u}_{0}-\alpha \textbf{u}_{N_t}^{k-1})\left(  \frac{1}{\Delta t}I_x - (1-\theta)A \right),\textbf{0},\cdots,\textbf{0} \right)^T.
\]
The matrices $C_{1,2}^{(\alpha)}$ are Strang type $\alpha-$circulant matrices and can be diagonalizable as stated in the following Lemma.
\begin{lemma}[Diagonalization of $\alpha$-circulant matrix, see \cite{bini2005numerical}]
The matrices $C_{1,2}^{(\alpha)}$ defined in \eqref{ciculantmat} can be diagonalized as 
\[
C_{j}^{(\alpha)} = VD_jV^{-1}, D_j=\diag(\sqrt{N_t}\mathbb{F}\Gamma_{\alpha}C_{j}^{(\alpha)}(:,1)),\; j=1,2,
\]
where $V=\Gamma_{\alpha}^{-1}\mathbb{F}^*, \mathbb{F}= \frac{1}{\sqrt{N_t}}\left[ \omega_0^{(l_1-1)(l_2-1)} \right]_{l_1,l_2=1}^{N_t}$ (with $i=\sqrt{-1}$ and $\omega_0=e^{\frac{2\pi i}{N_t}}$), $\Gamma_{\alpha}=\diag(1, \alpha^{\frac{1}{N_t}},\cdots,\alpha^{\frac{N_t-1}{N_t}})$  and $C_{j}^{(\alpha)}(:,1)$ being the first column of $C_{j}^{(\alpha)}, j=1, 2$.  
\end{lemma}
\noindent Using the property of tensor product, we can factor the \textit{all-at-once} system $C_1^{(\alpha)}\otimes I_x + C_2^{(\alpha)}\otimes A $ as $(V\otimes I_x)(D_1\otimes I_x + D_2\otimes A)(V^{-1}\otimes I_x )$, and we can solve \eqref{all-at-once1} by performing the following three steps
\begin{equation}\label{pintiter1}
\begin{aligned}
\text{Step}-(1)\; & S_1=(\mathbb{F}\otimes I_x)(\Gamma_{\alpha}\otimes I_x)\textbf{b}^{k-1},\\
\text{Step}-(2)\; & S_{2,n}=(\lambda_{1,n}I_x + \lambda_{2,n}A)^{-1}S_{1,n},\; n=1,2,\cdots N_t,\\
\text{Step}-(3)\; & \textbf{U}^{k}=(\Gamma_{\alpha}^{-1}\otimes I_x)(\mathbb{F}^*\otimes I_x)S_2,\\
\end{aligned}
\end{equation}
where $D_j=\diag(\lambda_{j,1},\cdots, \lambda_{j,N_t}), S_j=(S_{j,1}^T,\cdots,S_{j,N_t}^T)^T$  for $j=1,2$. The eigenvalues $\lambda_{1,n}$ and $\lambda_{2,n}$ are defined by 
\[
\lambda_{1,n}:=1-\alpha^{\frac{1}{N_t}}e^{-i\frac{2n\pi}{N_t}},\; \lambda_{2,n}:=\theta + (1-\theta)\alpha^{\frac{1}{N_t}}e^{-i\frac{2n\pi}{N_t}}
\] In \eqref{pintiter1}, Step-(1) and Step-(3) can be computed using FFT with $O(N_xN_t\log N_t)$ operations. Step-(2) is fully parallel. The convergence behaviour of the above described algorithm are given in the following section.
\subsection{Convergence Analysis}\label{subsection1}
In this section we present the error estimate for \eqref{modelproblem1a_wr} at different level.
\subsubsection{Error estimate at continuous level}
Let $e^k(x,t) = u^k(x,t)-u(x,t)$ be the error function at each PinT iteration. Then from \eqref{modelproblem1a} and \eqref{modelproblem1a_wr} we have the following error equation at continuous level
\begin{equation}\label{modelproblem1a_err} 
\begin{cases}
    e_t^k =-\Delta^2e^k, & (x,t)\in\Omega\times(0,T],\\
   e^k(x,0)=\alpha( e^k(x,T)- e^{k-1}(x,T)) , & x\in\Omega.
\end{cases}
\end{equation}
Taking Fourier transform in space for the equation \eqref{modelproblem1a_err}, we obtain the system of differential equations for each Fourier mode $\omega$ as follows 
\begin{equation}\label{modelproblem1a_err_fourier} 
\begin{cases}
    \hat e_t^k =-\omega^4 \hat e^k, & t\in(0,T],\\
   \hat e^k(0)=\alpha( \hat e^k(T)- \hat e^{k-1}(T)).
\end{cases}
\end{equation}
Solving \eqref{modelproblem1a_err_fourier} for each Fourier mode we have 
\begin{equation}\label{erreq1}
 \hat e^k(t) = e^{-\omega^4 t} \hat e^k(0).
\end{equation}
From \eqref{erreq1} we have $\parallel \hat e^k(T)\parallel_{l^2} \leq e^{-\omega^4 T}\parallel \hat e^k(0)\parallel_{l^2}$ and from $\hat e^k(0)=\alpha( \hat e^k(T)- \hat e^{k-1}(T))$ we have $\parallel\hat e^k(0)\parallel_{l^2} \leq \vert\alpha\vert \parallel\hat e^k(T)\parallel_{l^2}+ \vert\alpha\vert\parallel\hat e^{k-1}(T)\parallel_{l^2}$. Thus by using the first inequality twice in the second inequality we have 
\begin{equation}
\parallel\hat e^k(0)\parallel_{l^2} \leq \vert\alpha\vert e^{-\omega^4 T} \parallel\hat e^k(0)\parallel_{l^2}+ \vert\alpha\vert e^{-\omega^4 T} \parallel\hat e^{k-1}(0)\parallel_{l^2},
\end{equation}
and this imply 
\begin{equation}\label{erreq2}
\parallel\hat e^k(0)\parallel_{l^2}\leq \frac{\vert\alpha\vert e^{-\omega^4 T}}{1-\vert\alpha\vert e^{-\omega^4 T}} \parallel\hat e^{k-1}(0)\parallel_{l^2},
\end{equation}
where $\parallel \bullet\parallel_{l^2}$ is the sequential $l^2$ norm for the Fourier frequency variable.
\begin{theorem}[Error estimate at continuous level]\label{thm1}
Let $\{u^k\}_{k\geq 1}$ be the functions generated by PinT method \eqref{modelproblem1a_wr}, with $\vert\alpha\vert <1$. Then, the error function $e^k(x,t) $ satisfies the following linear convergence estimate
\[
\parallel e^k \parallel_{L^{\infty}(0, T;L^2(\Omega))} \leq \left(\frac{\vert\alpha\vert e^{-\omega_{\min}^4 T}}{1-\vert\alpha\vert e^{-\omega_{\min}^4 T}} \right) ^k  \parallel e^0 \parallel_{L^{\infty}(0, T;L^2(\Omega))},
\]
where $\omega_{\min}$ is the lowest Fourier frequency. 
\end{theorem}
\begin{proof}
From \eqref{erreq1} we have $\max\limits_{t\in[0, T]}\parallel\hat e^k(t) \parallel_{l^2} \leq \parallel\hat e^k(0) \parallel_{l^2}$, then using \eqref{erreq2}
and Parseval-Plancherel identity  we get
\[
\max\limits_{t\in[0, T]}\parallel e^k(t) \parallel_{L^2} \leq \left( \frac{\vert\alpha\vert e^{-\omega^4 T}}{1-\vert\alpha\vert e^{-\omega^4 T}} \right)^k  \parallel e^0(0) \parallel_{L^2}.
\]
Now by observing that the function $g_1(y) = \frac{\vert\alpha\vert e^{-y}}{1-\vert\alpha\vert e^{-y}}$ is monotonic decreasing, result follows immediately.
\end{proof}
\subsubsection{Error estimate at Semi-discrete level}
We perform the error analysis at semi-discrete level, to see the dependency on spatial mesh size $h$. Let $\textbf{e}^k(t)=\textbf{u}^k(t) - \textbf{u}(t)$ be the error function at each PinT iteration. Then from semi-discrete version of equation \eqref{modelproblem1a} and \eqref{semi-discrete1} we have the following error equation at semi-continuous level 
\begin{equation}\label{semi-discrete1_err}
  \begin{cases}
    \dot{\textbf{e}}^k(t) + A\textbf{e}^k(t) =0, \\
    \textbf{e}^k(0)=\alpha (\textbf{e}^k(T) - \textbf{e}^{k-1}(T)).
  \end{cases}
\end{equation}
On solving \eqref{semi-discrete1_err} we have 
\begin{equation}\label{erreq1_semidis}
 \textbf{e}^k(t) = \textbf{e}^{-A t}  \textbf{e}^k(0).
\end{equation}
By letting $t=T$ in \eqref{erreq1_semidis} and substitute in $\textbf{e}^k(0)=\alpha (\textbf{e}^k(T) - \textbf{e}^{k-1}(T))$, we obtain
\begin{equation}\label{erreq2_semidis}
 \textbf{e}^k(0)= \frac{-\alpha \textbf{e}^{-A T}}{1-\alpha \textbf{e}^{-A T}} \textbf{e}^{k-1}(0).
\end{equation}
Before analyzing further we first discuss spectral distribution of $A$. Eigenvalues of $\Delta_h$ in \eqref{deltah} are 
\[
\lambda_p=\frac{2}{h^2}\left\lbrace \cos\left( \frac{(p-1)\pi}{N_x -1}\right)-1 \right\rbrace, p=1,\cdots , N_x. 
\]
These $\lambda_p$'s are distinct and satisfy $\lambda_p\leq 0, \forall p$. Hence the eigenvalues of $A$ are $\lambda_p^2$, distinct and the spectrum of $A$, $\sigma(A)\subset[0, \infty)$. Thus the matrix $A$ is diagonalizable as 
\[
A=PDP^{-1}, D=\diag(\lambda_1^2, \lambda_2^2,\cdots, \lambda_{N_x}^2),
\]
where $P$ is made of eigenvectors corresponding to $\lambda_i^2$. Then for any vector norm $\parallel \bullet\parallel$ we have from \eqref{erreq2_semidis}
\[
\parallel P\textbf{e}^k(0)\parallel \leq \max\limits_{z\in\sigma (AT)}W(z)\parallel P\textbf{e}^{k-1}(0)\parallel,\; W(z)=\frac{\vert\alpha e^{-z}\vert}{\vert1-\alpha e^{-z}\vert}.
\]
Now from \eqref{erreq1_semidis} we have $\max\limits_{t\in[0, T]}\parallel \textbf{e}^k(t)\parallel\leq \parallel \textbf{e}^k(0)\parallel$, hence it is sufficient to analyze the behaviour of the function $W(z)$ to capture the error contraction. As $\sigma(A)\subset[0, \infty)$ we have $z\in [0, \infty)$ and the function $W(z)$ is monotonic decreasing for $z\in [0, \infty)$. The convergence result given in the following theorem.
\begin{theorem}[Error estimate at Semi-discrete level]\label{thm2}
Let $\textbf{u}^k$ be the $k-$th iteration of the PinT algorithm \eqref{semi-discrete1} with $\vert\alpha\vert<1/2$. Then we have the following linear convergence estimate \[
\parallel P\textbf{e}^k \parallel_{L^{\infty}(0, T;L^{\infty}(\Omega))}\leq \left(\frac{\vert\alpha\vert}{1-\vert\alpha\vert} \right) ^k \parallel P\textbf{e}^0\parallel_{L^{\infty}(0, T;L^{\infty}(\Omega))}.\]  
\end{theorem}
\begin{proof}
Proof of the theorem can be realized from the above discussion.
\end{proof}
\subsubsection{Error estimate at fully discrete level}
We now present the error analysis at fully discrete level, to see the dependency on temporal mesh size $\Delta t$ and the time integrators. Let $\textbf{e}_n^k=\textbf{u}_n^k - \textbf{u}_n$ be the error vector at each PinT iteration. Then for linear $\theta-$method we have 
\begin{equation}\label{linear-theta1_err}
  \begin{cases}
    (I_x + \theta\Delta t A)\textbf{e}_n^k = (I_x - (1-\theta)\Delta t A)\textbf{e}_{n-1}^k, & n=1,2,\cdots, N_t \\
    \textbf{e}_{0}^k=\alpha (\textbf{e}_{N_t}^k -\textbf{e}_{N_t}^{k-1}).
  \end{cases}
\end{equation}
Using recurrence relation in \eqref{linear-theta1_err} we have 
\begin{equation}\label{linear-theta1_err2}
  \begin{cases}
    \textbf{e}_{N_t}^k = \left(R_{\theta}(\Delta t A)\right)^{N_t}\textbf{e}_{0}^k, & R_{\theta}(\Delta t A)=(I_x + \theta\Delta t A)^{-1}(I_x - (1-\theta)\Delta t A) \\
    \textbf{e}_{0}^k=\alpha (\textbf{e}_{N_t}^k -\textbf{e}_{N_t}^{k-1}).
  \end{cases}
\end{equation}
Then similar to the analysis of semi-discrete case we have 
\[
\parallel P\textbf{e}^k(0)\parallel \leq \max\limits_{z\in\sigma (\Delta t A)}\widetilde{W}(z)\parallel P\textbf{e}^{k-1}(0)\parallel,\; \widetilde{W}(z)=\frac{\vert\alpha R_{\theta}^{N_t}(z)\vert}{\vert1-\alpha R_{\theta}^{N_t}(z)\vert},
\]
where the matrix $P$ as described earlier. So analyzing $\widetilde{W}(z)$ is sufficient to see the error contraction. For $\theta=1$, we have $R_1(z)=\frac{1}{1+z}$ for $z\in[0, \infty)$, which is a decreasing function in $z$. The function $g_2(y)=\frac{\vert\alpha\vert y}{1-\vert\alpha\vert y}$ is monotonic increasing. Hence $\max\limits_{z\in\sigma (\Delta t A)}\widetilde{W}(z) \leq g_2(R_1(0))$ for backward-Euler method. For $\theta=1/2$, we have $R_{1/2}(z)=\frac{1-\frac{z}{2}}{1+\frac{z}{2}}$ for $z\in[0, \infty)$, which is a decreasing function in $z$. Similarly we have $\max\limits_{z\in\sigma (\Delta t A)}\widetilde{W}(z) \leq g_2(R_{1/2}(0))$ for the Trapezoidal rule.
\begin{theorem}[Error estimate at fully discrete level]\label{thm3}
For $\theta=1$ and $\theta=1/2$ the discrete error $\textbf{e}_n^k$ at $k-$th iterate of the PinT algorithm \eqref{linear-theta1} with $\vert\alpha\vert<1/2$ satisfies the following linear convergence estimate 
\[
\max\limits_{n=1,2,\cdots,N_t}\parallel P\textbf{e}_n^k \parallel\leq \left(\frac{\vert\alpha\vert}{1-\vert\alpha\vert} \right) ^k \parallel P\textbf{e}_0^0 \parallel.\]   
\end{theorem}
\begin{proof}
Proof of the theorem can be realized from the above discussion.
\end{proof}
\remark
We can extend the above analysis for 2D and 3D at each level.
\begin{enumerate}
\item In 2D, Theorem \ref{thm1} will also hold with the convergence factor $ \frac{\vert\alpha\vert e^{-(\omega_{\min}^4 +\xi_{\min}^4) T}}{1-\vert\alpha\vert e^{-(\omega_{\min}^4+\xi_{\min}^4) T}} $, where $\omega_{\min}$ and $\xi_{\min}$ are the lowest Fourier frequency for spatial variables. Similarly one can write the convergence factor in 3D.  
\item
The spectrum of two dimensional discrete Laplacian $I\otimes\Delta_h + \Delta_h\otimes I$ with homogeneous  Neumann boundary condition and equidistant grid points in both spatial direction 
is given by 
\[
\lambda_{p,q}=\frac{2}{h^2}\left\lbrace \cos\left( \frac{(p-1)\pi}{N_x -1}\right)-1 \right\rbrace + \frac{2}{h^2}\left\lbrace \cos\left( \frac{(q-1)\pi}{N_x -1}\right)-1 \right\rbrace, p,q=1,\cdots , N_x, 
\]
where $\Delta_h$ is one dimensional discrete Laplacian given in \eqref{deltah}. Hence one can observe that the spectral property remains same for $A$ in 2D, so Theorem \ref{thm2} holds true in 2D.
\item Theorem \ref{thm3} also depends on spectral distribution of $A$, hence it also holds in 2D.
\end{enumerate}
\section{Linearised CH Equation}\label{Section3}
In this section we consider the linearised CH equation in 1D as given below
\begin{equation}\label{modelproblem2a} 
\begin{cases}
    u_t =-\beta^2\Delta u - \epsilon^2\Delta^2u, & (x,t)\in\Omega\times(0,T]\\
    \partial_{\nu}u=0= \partial_{\nu}(\Delta u), & (x,t)\in\partial\Omega\times(0,T]\\
    u(x,0)=u_0, & x\in\Omega.
\end{cases}
\end{equation}
The PinT algorithm for \eqref{modelproblem2a}, is based on WR iteration and \textit{periodic-like} initial condition, is the following 
\begin{equation}\label{modelproblem2a_wr} 
\begin{cases}
    u_t^k =-\beta^2\Delta u^k-\epsilon^2\Delta^2u^k, & (x,t)\in\Omega\times(0,T],\\
   u^k(x,0)=\alpha u^k(x,T)-\alpha u^{k-1}(x,T) +u_0(x) , & x\in\Omega,\\
\end{cases}
\end{equation}
where $k\geq 1$ is the WR iteration index and $\alpha\in(0, 1)$ is the PinT free parameter. The problem \eqref{modelproblem2a} is linear, so the discretization and  PinT formulation at semi-discrete level and fully-discrete level will follow from \eqref{semi-discrete1} and \eqref{linear-theta1} respectively by taking $A=\beta^2\Delta_h + \epsilon^2\Delta^2_h \in\mathbb{R}^{N_x \times N_x}$, where $\Delta_h$ as defined in \eqref{deltah}. Thus the PinT algorithm for \eqref{modelproblem2a} is given by \eqref{pintiter1} for above prescribed $A$. Thus we will discuss only the convergence behaviour in the next section.
\subsection{Convergence analysis}
We present the convergence result for \eqref{modelproblem2a} at different level in 1D, procedure of getting error estimate is same as process given in subsection \ref{subsection1}.
\begin{theorem}[Error estimate at continuous level]\label{thm4}
Let $\{u^k\}_{k\geq 1}$ be the functions generated by PinT method \eqref{modelproblem2a_wr}, with $\vert\alpha\vert<1/2$. Then, the error function $e^k(x,t)=u^k(x,t)-u(x,t)$ satisfies the following linear convergence estimate
\[
\parallel e^k \parallel_{L^{\infty}(0, T;L^2(\Omega))} \leq \rho^k  \parallel e^0 \parallel_{L^{\infty}(0, T;L^2(\Omega))},
\]
where the convergence factor $\rho$ is given by
\begin{equation*}
\rho^k=
\begin{cases}
    e^{\widetilde{g}(z^*)} \left( \frac{\vert\alpha\vert e^{\widetilde{g}(z^*)}}{1-\vert\alpha\vert e^{\widetilde{g}(z^*)}}\right)^k , & \text{for}\; z<\frac{\beta^2 \sqrt{T}}{\epsilon^2},\\
   \left( \frac{\vert\alpha\vert}{1-\vert\alpha\vert}\right)^k  , & \text{for}\; z\geq\frac{\beta^2 \sqrt{T}}{\epsilon^2},\\
\end{cases}
\end{equation*}
with $\widetilde{g}(z)=-\epsilon^2 z^2 + \beta^2 \sqrt{T}z,\; z^*=\frac{\beta^2 \sqrt{T}}{2\epsilon^2}$. 
\end{theorem}
\begin{proof}
Following the continuous error analysis in Subsection \ref{subsection1}, we have for \eqref{modelproblem2a_wr}
\begin{equation}\label{erreq3}
\hat e^k(t) = e^{\left( \beta^2\omega^2-\epsilon^2\omega^4 \right) t} \hat e^k(0),
\end{equation}
\begin{equation}\label{erreq4}
\parallel\hat e^k(0)\parallel_{l^2}\leq \frac{\vert\alpha\vert e^{\left( \beta^2\omega^2-\epsilon^2\omega^4 \right) T}}{1-\vert\alpha\vert e^{\left( \beta^2\omega^2-\epsilon^2\omega^4 \right) T}} \parallel\hat e^{k-1}(0)\parallel_{l^2},
\end{equation}
where $\omega$ is a Fourier variable. Observe that the term $\beta^2\omega^2-\epsilon^2\omega^4$ is not always less than zero. By that we mean for small frequency mode $\omega$ the term $\beta^2\omega^2-\epsilon^2\omega^4$ is grater than zero, given the barrier on the parameter $\beta, \epsilon$. Considering above fact, we define $\widetilde{g}(z)=-\epsilon^2 z^2 + \beta^2 \sqrt{T}z$, where $z=\omega^2 \sqrt{T}>0$. The function $\widetilde{g}(z)>0$ for $z<\frac{\beta^2 \sqrt{T}}{\epsilon^2}$, and having maxima at $z=z^*=\frac{\beta^2 \sqrt{T}}{2\epsilon^2}$. The function $\widetilde{g}(z)\leq 0$ for $z\geq\frac{\beta^2 \sqrt{T}}{\epsilon^2}$. Finally we get the error estimate by  
using Parseval-Plancherel identity on the error solution \eqref{erreq3} and recurrence relation \eqref{erreq4}, and 
noticing that the function $g_3(y)=\frac{\vert\alpha\vert e^{y}}{1-\vert\alpha\vert e^{y}}$ is monotonic increasing.
\end{proof}

\begin{theorem}[Error estimate at semi-discrete level]\label{thm5}
Let $\textbf{u}^k$ be the $k-$th semi-discrete iteration of the PinT algorithm \eqref{modelproblem2a_wr} with $\vert\alpha\vert<1/2$, and set $\textbf{e}^k(t)=\textbf{u}^k(t)- \textbf{u}(t)$ be the error at $k-$th iterate. Then we have the following linear convergence estimate \[
\parallel P\textbf{e}^k \parallel_{L^{\infty}(0, T;L^{\infty}(\Omega))}\leq \rho^k \parallel P\textbf{e}^0\parallel_{L^{\infty}(0, T;L^{\infty}(\Omega))},\]
where the convergence factor $\rho$ is given by
\begin{equation*}
\rho^k=
\begin{cases}
    e^{-\widecheck{g}(z^*)} \left( \frac{\vert\alpha\vert e^{-\widecheck{g}(z^*)}}{1-\vert\alpha\vert e^{-\widecheck{g}(z^*)}}\right)^k , & \text{for}\; z>-\frac{\beta^2 T}{\epsilon^2},\\
   \left( \frac{\vert\alpha\vert}{1-\vert\alpha\vert}\right)^k  , & \text{for}\; z\leq -\frac{\beta^2 T}{\epsilon^2},\\
\end{cases}
\end{equation*}
with $\widecheck{g}(z)=\frac{\epsilon^2}{T} z^2 + \beta^2 z,\; z^*=-\frac{\beta^2 T}{2\epsilon^2}$.  
\end{theorem}
\begin{proof}
Following the semi-discrete error analysis in Subsection \ref{subsection1}, we have for the semi-discrete version of \eqref{modelproblem2a_wr}
\begin{equation}\label{erreq3_semidis}
 \textbf{e}^k(t) = \textbf{e}^{-A t}  \textbf{e}^k(0),
\end{equation}
\begin{equation}\label{erreq4_semidis}
 \textbf{e}^k(0)= \frac{-\alpha \textbf{e}^{-A T}}{1-\alpha \textbf{e}^{-A T}} \textbf{e}^{k-1}(0),
\end{equation}
where $A=\beta^2\Delta_h + \epsilon^2\Delta^2_h$. The eigenvalues of $A$ are $\mu_p=\beta^2\lambda_p + \epsilon^2\lambda_p^2$ for $p=1,2,\cdots N_x$, where $\lambda_p$ is as defined earlier. Clearly $\mu_p$'s are distinct, so $A$ is diagonalisable, and can be written as $A=PDP^{-1}$. Thus as earlier it is sufficient to analyze $\max\limits_{r\in\sigma (AT)}W(r)$, where $W(r)=\frac{\vert\alpha e^{-r}\vert}{\vert1-\alpha e^{-r}\vert}$. Now observe that $\mu_p$ is not always grater than or equal to zero for all $p$'s, given the condition on $\beta, \epsilon$. So we define $\widecheck{g}(z)=\frac{\epsilon^2}{T} z^2 + \beta^2 z$, where $z=\lambda_p T\leq 0$. The function $\widecheck{g}(z)<0$ for $z>-\frac{\beta^2 T}{\epsilon^2}$, having minima at $z=z^*=-\frac{\beta^2 T}{2\epsilon^2}$. The function $\widecheck{g}(z)\geq 0$ for $z\leq -\frac{\beta^2 T}{\epsilon^2}$. Now the result will follow using \eqref{erreq4_semidis} on \eqref{erreq3_semidis} and observing that the function $g_4(y)=\frac{\vert\alpha\vert e^{-y}}{1-\vert\alpha\vert e^{-y}}$ is monotonic decreasing.
\end{proof}
\begin{theorem}[Error estimate at fully discrete level]\label{thm6}
For $\theta=1$ and $\theta=1/2$ the fully discrete error $\textbf{e}_n^k=\textbf{u}_n^k -\textbf{u}_n$ at $k$-th iteration of the PinT algorithm \eqref{modelproblem2a_wr} with $\vert\alpha\vert<1/2$ satisfies the following linear convergence estimate 
\[
\max\limits_{n=1,2,\cdots,N_t}\parallel P\textbf{e}_n^k \parallel\leq \rho^k \parallel P\textbf{e}_0^0 \parallel,\; \rho=\frac{\vert\alpha\vert R_{\theta}^{N_t}(z^*)}{1-\vert\alpha\vert R_{\theta}^{N_t}(z^*)}.\]   
\end{theorem}
\begin{proof}
At this point it is clear that analyzing the behaviour of the quantity $\max\limits_{z\in\sigma (\Delta t A)}\widetilde{W}(z)$ is enough, where $\widetilde{W}(z)=\frac{\vert\alpha R_{\theta}^{N_t}(z)\vert}{\vert1-\alpha R_{\theta}^{N_t}(z)\vert}$. Now the stability function corresponding to $\theta=1$ is $R_1(z)=\frac{1}{1+\phi(z)}$, where $\phi(z)=-\beta^2 z + \frac{\epsilon^2}{\Delta t}z^2$ with $z=-\lambda_p \Delta t \geq 0$. The function $R_1(z)$ is monotonic increasing for $z\leq\frac{\beta^2 \Delta t}{2\epsilon^2}$ and monotonic decreasing for $z>\frac{\beta^2 \Delta t}{2\epsilon^2}$. Observing that the function $g_5(y)=\frac{\vert\alpha\vert y^{N_t}}{1-\vert\alpha\vert y^{N_t}}$ is monotonic increasing. So we have $\max\limits_{z\in\sigma (\Delta t A)}\widetilde{W}(z) \leq \frac{\vert\alpha\vert R_{1}^{N_t}(z^*)}{1-\vert\alpha\vert R_{1}^{N_t}(z^*)}$, where $z^*=\frac{\beta^2 \Delta t}{2\epsilon^2}$. For $\theta=1/2$ the stability function is $R_{1/2}(z)=\frac{1-\frac{\phi (z)}{2}}{1+\frac{\phi (z)}{2}}$, and a similar analysis holds for $R_{1/2}(z)$ and we get the estimate for $\theta=1/2$.
\end{proof}
\remark
We can extend the above analysis to 2D and 3D at each level.
\begin{enumerate}
\item In 2D, Theorem \ref{thm4} also holds for the convergence factor $\rho$ having the form
\begin{equation*}
\rho^k=
\begin{cases}
    e^{\widetilde{g}(z_1^*, z_2^*)} \left( \frac{\vert\alpha\vert e^{\widetilde{g}(z_1^*, z_2^*)}}{1-\vert\alpha\vert e^{\widetilde{g}(z_1^*, z_2^*)}}\right)^k , & \text{for}\; z_1, z_2 <\frac{\beta^2 \sqrt{T}}{\epsilon^2},\\
   \left( \frac{\vert\alpha\vert}{1-\vert\alpha\vert}\right)^k  , & \text{for}\; z_1, z_2\geq\frac{\beta^2 \sqrt{T}}{\epsilon^2},\\
\end{cases}
\end{equation*}
where $\widetilde{g}(z_1, z_2)=-\epsilon^2 \left( z_1^2 +z_2^2\right)  + \beta^2 \sqrt{T}\left( z_1+z_2\right) ,\; z_1^*=z_2^*=\frac{\beta^2 \sqrt{T}}{2\epsilon^2}$. The variables are $z_1=\omega^2 \sqrt{T}>0, z_2=\xi^2 \sqrt{T}>0$, where $\omega, \xi$ are Fourier variables in 2D. Similarly one can write the convergence factor $\rho$ in 3D.  
\item
The spectrum of two dimensional discrete operator $A=\beta^2\Delta_h + \epsilon^2\Delta^2_h$ with homogeneous  Neumann boundary condition and equidistant grid points in both spatial direction 
are $\beta^2\lambda_{p,q} + \epsilon^2\lambda_{p,q}^2$ for $p,q=1,\cdots , N_x$, where $\lambda_{p,q}$ is as defined earlier. Hence one can observe that the spectral property remains same for $A$ in 2D and we get the error estimate similarly by considering $z=-\lambda_{p,q}T$, so the Theorem \ref{thm5} also holds true in 2D.
\item Theorem \ref{thm6} also depends on spectral distribution of $A$, hence we get the error estimate in 2D for $\theta=1, 1/2$ in the same way by considering $z=-\lambda_{p,q}\Delta t$.
\item For the general fourth order operator of the form 
\begin{equation}\label{general_4thorder}
u_t=-\Delta^2 u + \Delta u, 
\end{equation}
having the following PinT algorithm 
\begin{equation*} 
\begin{cases}
    u_t^k =-\Delta^2 u^k +\Delta u^k, & (x,t)\in\Omega\times(0,T],\\
   u^k(x,0)=\alpha u^k(x,T)-\alpha u^{k-1}(x,T) +u_0(x) , & x\in\Omega,\\
\end{cases}
\end{equation*}
a similar analysis can be done to get the error estimate at continuous level where the convergence factor $\rho$ is given by
\begin{equation*}
\rho=
\begin{cases}
     \frac{\vert\alpha\vert e^{-\left( \omega_{\min}^2+\omega_{\min}^4\right)T}}{1-\vert\alpha\vert e^{-\left( \omega_{\min}^2+\omega_{\min}^4\right)T}  }, & \text{in}\;  1D,\\
    \frac{\vert\alpha\vert e^{-\left( \omega_{\min}^2 +\xi_{\min}^2+\omega_{\min}^4+\xi_{\min}^4\right) T}}{1-\vert\alpha\vert e^{-\left( \omega_{\min}^2 +\xi_{\min}^2+\omega_{\min}^4+\xi_{\min}^4\right) T}} , & \text{in}\;  2D.\\
\end{cases}
\end{equation*}
Using the same reasoning as above one can get the error estimate for \eqref{general_4thorder} at semi-discrete and fully discrete level.
\end{enumerate}
\section{The CH Equation}\label{Section4}
In this section we formulate the iterative PinT algorithm for the CH equation 
\begin{equation}\label{modelproblem3} 
\begin{cases}
    u_t =\Delta f(u) - \epsilon^2\Delta^2u, & (x,t)\in\Omega\times(0,T]\\
    \partial_{\nu}u=0= \partial_{\nu}(\Delta u), & (x,t)\in\partial\Omega\times(0,T]\\
    u(x,0)=u_0, & x\in\Omega.
\end{cases}
\end{equation}
The PinT algorithm at continuous level for \eqref{modelproblem3} is based on WR iteration and \textit{periodic-like} initial condition, as the following: 
\begin{equation}\label{modelproblem3_wr} 
\begin{cases}
    u_t^k =\Delta f(u^k)-\epsilon^2\Delta^2u^k, & (x,t)\in\Omega\times(0,T],\\
   u^k(x,0)=\alpha u^k(x,T)-\alpha u^{k-1}(x,T) +u_0(x) , & x\in\Omega,\\
\end{cases}
\end{equation}
where $k\geq 1$ is the WR iteration index and $\alpha$ is the PinT free parameter. Before formulating the PinT algorithm at fully discrete level, we can observe that non-increasing property of the energy also holds for the algorithm \eqref{modelproblem3_wr}. This means that $\frac{d}{dt}\mathcal{E}(u^k)\leq 0$ for all $k\geq 1$. 
\subsection{Formulation of PinT algorithm}
We construct two PinT algorithm for the CH equation \eqref{modelproblem3} depending on the discretization of the equation \eqref{modelproblem3}. The first one is fully implicit and conditionally gradient stable, with the condition $\Delta t\leq O(\epsilon^2)$, see \cite{Elliott}. And the other is an unconditional gradient stable nonlinear splitting scheme of Eyre in \cite{ David, Eyre}. Firstly we formulate the PinT algorithm for the fully implicit case.
We discretize the spatial derivatives using central finite difference scheme in an interval with equidistant grid points and homogeneous Neumann boundary condition as shown in \eqref{modelproblem3}. We obtain the following semi-discrete CH equation in \eqref{semi-dis_CH} and the corresponding semi-discrete PinT iterative algorithm in \eqref{semi-dis_CH_WR}
\begin{equation}\label{semi-dis_CH}
  \begin{cases}
    \dot{\textbf{u}}(t) =H(t, \textbf{u}), \\
    \textbf{u}(0)=\textbf{u}_{0},
  \end{cases}
\end{equation}
\begin{equation}\label{semi-dis_CH_WR}
  \begin{cases}
    \dot{\textbf{u}}^k(t) =H(t, \textbf{u}^k), \\
    \textbf{u}^k(0)=\alpha \textbf{u}^k(T) -\alpha \textbf{u}^{k-1}(T) + \textbf{u}_{0},
  \end{cases}
\end{equation}
where $H(t,\textbf{u}):\mathbb{R}^{+}\times\mathbb{R}^{N_x}\rightarrow\mathbb{R}^{N_x}$ is given by $H(t,\textbf{u})=\Delta_h f(\textbf{u})-\epsilon^2 \Delta_h^2\textbf{u}$. Applying backward-Euler in time in \eqref{semi-dis_CH}, we have the following fully discrete scheme in \eqref{fully-dis_CH} and the corresponding fully discrete PinT iterative algorithm in \eqref{fully-dis_CH_WR}
\begin{equation}\label{fully-dis_CH}
  \begin{cases}
    \frac{\textbf{u}_n-\textbf{u}_{n-1}}{\Delta t} =H(t_n, \textbf{u}_n), \\
    \textbf{u}(t_0)=\textbf{u}_{0},
  \end{cases}
\end{equation}
\begin{equation}\label{fully-dis_CH_WR}
  \begin{cases}
    \frac{\textbf{u}_n^k-\textbf{u}_{n-1}^k}{\Delta t} =H(t_n, \textbf{u}_n^k), \\
    \textbf{u}^k_0=\alpha \textbf{u}_{N_t}^k -\alpha \textbf{u}_{N_t}^{k-1} + \textbf{u}_{0}.
  \end{cases}
\end{equation}
By gathering all the space-time points in \eqref{fully-dis_CH_WR}, we have the following nonlinear \textit{all-at-once} or  nonlinear \textit{space-time} system of $\textbf{U}^k$ as
\begin{equation}\label{all-at-once_CH}
\left( C_1^{(\alpha)}\otimes I_x\right)\textbf{U}^k = \left(I_{t}\otimes \Delta_h\right)\widetilde{f}(\textbf{U}^k)- \left(I_{t}\otimes \Delta_h\right)\textbf{U}^k -\epsilon^2\left(I_{t}\otimes \Delta_h^2\right)\textbf{U}^k +\textbf{b}^{k-1}, 
\end{equation}
where $I_t\in\mathbb{R}^{N_t \times N_t}$ is the identity matrix, $\widetilde{f}(\textbf{U}^k):=\left( (\textbf{u}_1^k)^3, (\textbf{u}_2^k)^3, \cdots, (\textbf{u}_{N_t}^k)^3 \right)^T$ and  
$
\textbf{b}^{k-1}=\left( \frac{I_x}{\Delta t}\left(\textbf{u}_{0}-\alpha \textbf{u}_{N_t}^{k-1}\right),0,\cdots,0 \right)^T.
$
To solve the nonlinear system \eqref{all-at-once_CH}, we apply the quasi-Newton method on $G(\textbf{U}^k)$, where 
\begin{equation}\label{nonlinear_G}
G(\textbf{U}^k):=\left( C_1^{(\alpha)}\otimes I_x\right)\textbf{U}^k - \left(I_{t}\otimes \Delta_h\right)\widetilde{f}(\textbf{U}^k)+ \left(I_{t}\otimes \Delta_h\right)\textbf{U}^k +\epsilon^2\left(I_{t}\otimes \Delta_h^2\right)\textbf{U}^k -\textbf{b}^{k-1}.
\end{equation}
We define the following matrix 
\begin{equation}
J(\textbf{U}^k):=
\begin{bmatrix}
J_s(\textbf{u}_1^k) & & & &\\
& J_s(\textbf{u}_2^k) & & &\\
& &  \ddots & &\\
& & & & J_s(\textbf{u}_{N_t}^k)
\end{bmatrix},
\end{equation}
where $J_s(\textbf{u}_{n}^k):=\diag((3u_n^{1,k})^2, \cdots, (3u_n^{N_x, k})^2)$, for $n=1,\cdots, N_t$. Now we write the Jacobian of $G$ as the following 
\begin{equation}\label{jacobian1}
G'(\textbf{U}^k)=\left( C_1^{(\alpha)}\otimes I_x\right) - \left(I_{t}\otimes \Delta_h\right)J(\textbf{U}^k)+ \left(I_{t}\otimes \Delta_h\right) +\epsilon^2\left(I_{t}\otimes \Delta_h^2\right).
\end{equation}
Now the quasi-Newton method to solve \eqref{nonlinear_G} becomes computing the following iteration for $m=1,2,\cdots$
\begin{equation}\label{quasi_newton}
\textbf{U}^k_m = \textbf{U}^k_{m-1} - (\widetilde{G}'(\textbf{U}^k_{m-1}))^{-1}G(\textbf{U}^k_{m-1}), 
\end{equation}
where $\widetilde{G}'$ is an approximation to the Jacobian  $G'$. The approximated Jacobian $\widetilde{G}'$ has the following form 
\begin{equation}\label{approx_jacobian}
\widetilde{G}'(\textbf{U}^k)=\left( C_1^{(\alpha)}\otimes I_x\right) - \left(I_{t}\otimes \Delta_h\widetilde{J}(\textbf{U}^k)\right)+ \left(I_{t}\otimes \Delta_h\right) +\epsilon^2\left(I_{t}\otimes \Delta_h^2\right), 
\end{equation}
where $\widetilde{J}(\textbf{U}^k):=I_t \otimes \left( \frac{1}{N_t}\sum_{n=1}^{N_t}J_s(\textbf{u}_n^k)\right) $.
This process of approximating the Jacobian has been successfully applied in \cite{gander2017halpern, gander2019shulin, gu2020parallel}. We rewrite \eqref{quasi_newton} using \eqref{nonlinear_G} and \eqref{approx_jacobian} to get the following
\begin{equation}\label{quasi_newton_newform}
\widetilde{G}'(\textbf{U}^k_{m-1})\textbf{U}^k_m = \left(I_{t}\otimes \Delta_h\right)\widetilde{f}(\textbf{U}^k_{m-1})-\left(I_{t}\otimes \Delta_h\widetilde{J}(\textbf{U}^k_{m-1})\right)\textbf{U}^k_{m-1} +\textbf{b}^{k-1}.
\end{equation}
We use the diagonalizability of $C_1^{(\alpha)}$ and the tensor product to factor the system $\widetilde{G}'(\textbf{U}^k_{m-1})$ as $(V\otimes I_x)(D_1\otimes I_x - I_{t}\otimes \Delta_h\widetilde{J}(\textbf{U}^k_{m-1})+ I_{t}\otimes \Delta_h +\epsilon^2 I_{t}\otimes \Delta_h^2)(V^{-1}\otimes I_x )$, and then we  solve \eqref{quasi_newton_newform} by performing the following three steps
\begin{equation}\label{pintiter_CH}
\begin{aligned}
\text{Step}-(1)\; & S_1=\left(\mathbb{F}\otimes I_x \right)\left(\Gamma_{\alpha}\otimes I_x \right)\textbf{r}^{k}_{m-1},\\
\text{Step}-(2)\; & S_{2,n}=\left(\lambda_{1,n}I_x  -  \Delta_h\widetilde{J}(\textbf{U}^k_{m-1})+  \Delta_h +\epsilon^2 \Delta_h^2 \right)^{-1}S_{1,n},\; n=1,2,\cdots N_t,\\
\text{Step}-(3)\; & \textbf{U}_m^{k}=\left(\Gamma_{\alpha}^{-1}\otimes I_x \right) \left(\mathbb{F}^*\otimes I_x \right)S_2,\\
\end{aligned}
\end{equation}
where $\textbf{r}^{k}_{m-1}:=\left(I_{t}\otimes \Delta_h\right)\widetilde{f}(\textbf{U}^k_{m-1})-\left(I_{t}\otimes \Delta_h\widetilde{J}(\textbf{U}^k_{m-1})\right)\textbf{U}^k_{m-1} +\textbf{b}^{k-1}$, and  $D_1, S_1, S_2, \lambda_{1,n}$ are defined earlier. We call this algorithm PinT-I.

\vspace{1cm}
Next we formulate the PinT algorithm for the CH equation \eqref{modelproblem3} using the nonlinear splitting scheme of Eyre \cite{David, Eyre}. The idea here is to split the homogeneous free
energy $F(u)$ into the sum of a convex term and a concave term, and then treating the convex term implicitly and the concave term explicitly to obtain a nonlinear unconditionally stable first order in time and second order in space accurate approximation of the CH equation. The iterative PinT algorithm at fully discrete level corresponding to nonlinear splitting is the following
\begin{equation}\label{fully-dis_CH_WR_Eyre}
  \begin{cases}
    \frac{\textbf{u}_n^k-\textbf{u}_{n-1}^k}{\Delta t} =\Delta_h (\textbf{u}_n^k)^3 - \Delta_h \textbf{u}_{n-1}^k - \epsilon^2\Delta_h^2 \textbf{u}_n^k, \\
    \textbf{u}^k_0=\alpha \textbf{u}_{N_t}^k -\alpha \textbf{u}_{N_t}^{k-1} + \textbf{u}_{0}.
  \end{cases}
\end{equation}
By gathering all the space-time points in \eqref{fully-dis_CH_WR_Eyre}, we have the following nonlinear \textit{all-at-once} system of $\textbf{U}^k$ as
\begin{equation}\label{all-at-once_CH_Eyre}
\left( C_1^{(\alpha)}\otimes I_x\right)\textbf{U}^k = \left(I_{t}\otimes \Delta_h\right)\widetilde{f}(\textbf{U}^k)- \left(C_3^{(\alpha)}\otimes \Delta_h\right)\textbf{U}^k -\epsilon^2\left(I_{t}\otimes \Delta_h^2\right)\textbf{U}^k +\textbf{b}_0^{k-1}, 
\end{equation}
where $C_3^{(\alpha)}$ is the matrix $C_2^{(\alpha)}$ for $\theta=0$ in \eqref{ciculantmat}, and $
\textbf{b}_0^{k-1}=\left( (\frac{I_x}{\Delta t}-\Delta_h)\left(\textbf{u}_{0}-\alpha \textbf{u}_{N_t}^{k-1}\right),0,\cdots,0 \right)^T.
$
To solve the nonlinear system \eqref{all-at-once_CH_Eyre}, we apply the quasi-Newton method on $Q(\textbf{U}^k)$, where 
\begin{equation}\label{nonlinear_Q}
Q(\textbf{U}^k):=\left( C_1^{(\alpha)}\otimes I_x\right)\textbf{U}^k - \left(I_{t}\otimes \Delta_h\right)\widetilde{f}(\textbf{U}^k)+ \left(C_3^{(\alpha)}\otimes \Delta_h\right)\textbf{U}^k +\epsilon^2\left(I_{t}\otimes \Delta_h^2\right)\textbf{U}^k -\textbf{b}_0^{k-1}.
\end{equation}
The Jacobian of $Q$ is the following 
\begin{equation}\label{jacobian_Q}
Q'(\textbf{U}^k)=\left( C_1^{(\alpha)}\otimes I_x\right) - \left(I_{t}\otimes \Delta_h\right)J(\textbf{U}^k)+ \left(C_3^{(\alpha)}\otimes \Delta_h\right) +\epsilon^2\left(I_{t}\otimes \Delta_h^2\right).
\end{equation}
The quasi-Newton method to solve \eqref{nonlinear_Q} consists of the following iteration for $m=1,2,\cdots$
\begin{equation}\label{quasi_newton_eyre}
\textbf{U}^k_m = \textbf{U}^k_{m-1} - (\widetilde{Q}'(\textbf{U}^k_{m-1}))^{-1}Q(\textbf{U}^k_{m-1}), 
\end{equation}
where $\widetilde{Q}'$ is an approximation to the Jacobian  $Q'$ in \eqref{jacobian_Q}. The approximated Jacobian $\widetilde{Q}'$ has the following form 
\begin{equation}\label{approx_jacoQ}
\widetilde{Q}'(\textbf{U}^k)=\left( C_1^{(\alpha)}\otimes I_x\right) - \left(I_{t}\otimes \Delta_h\widetilde{J}(\textbf{U}^k)\right)+ \left(C_3^{(\alpha)}\otimes \Delta_h\right) +\epsilon^2\left(I_{t}\otimes \Delta_h^2\right). 
\end{equation}
Rewriting \eqref{quasi_newton_eyre} using \eqref{nonlinear_Q} and \eqref{approx_jacoQ} yields 
\begin{equation}\label{quasi_newton_newform_eyre}
\widetilde{Q}'(\textbf{U}^k_{m-1})\textbf{U}^k_m = \left(I_{t}\otimes \Delta_h\right)\widetilde{f}(\textbf{U}^k_{m-1})-\left(I_{t}\otimes \Delta_h\widetilde{J}(\textbf{U}^k_{m-1})\right)\textbf{U}^k_{m-1} +\textbf{b}_0^{k-1}.
\end{equation}
We use the diagonalizability of $C_1^{(\alpha)}, C_3^{(\alpha)}$ and the tensor product to factor the system $\widetilde{Q}'(\textbf{U}^k_{m-1})$ as $(V\otimes I_x)(D_1\otimes I_x - I_{t}\otimes \Delta_h\widetilde{J}(\textbf{U}^k_{m-1})+ D_3\otimes \Delta_h +\epsilon^2 I_{t}\otimes \Delta_h^2)(V^{-1}\otimes I_x )$, and then we  solve \eqref{quasi_newton_newform_eyre} by performing usual three steps
\begin{equation}\label{pintiter_CH_eyre}
\begin{aligned}
\text{Step}-(1)\; & S_1=\left(\mathbb{F}\otimes I_x \right)\left(\Gamma_{\alpha}\otimes I_x \right)\textbf{r}^{k}_{m-1},\\
\text{Step}-(2)\; & S_{2,n}=\left(\lambda_{1,n}I_x  -  \Delta_h\widetilde{J}(\textbf{U}^k_{m-1})+  \lambda_{3,n}\Delta_h +\epsilon^2 \Delta_h^2 \right)^{-1}S_{1,n},\; n=1,2,\cdots N_t,\\
\text{Step}-(3)\; & \textbf{U}_m^{k}=\left(\Gamma_{\alpha}^{-1}\otimes I_x \right) \left(\mathbb{F}^*\otimes I_x \right)S_2,\\
\end{aligned}
\end{equation}
where $\textbf{r}^{k}_{m-1}:=\left(I_{t}\otimes \Delta_h\right)\widetilde{f}(\textbf{U}^k_{m-1})-\left(I_{t}\otimes \Delta_h\widetilde{J}(\textbf{U}^k_{m-1})\right)\textbf{U}^k_{m-1} +\textbf{b}_0^{k-1}$, $D_3=\diag(\lambda_{3,1}, \cdots\\,  \lambda_{3,N_t})$, $\lambda_{3,n}= \alpha^{\frac{1}{N_t}}e^{-i\frac{2n\pi}{N_t}}$ and  $D_1, S_1, S_2, \lambda_{1,n}$ are defined earlier. We call this algorithm PinT-II.
\remark
One can extend the algorithm \eqref{pintiter_CH} and \eqref{pintiter_CH_eyre} to 2D and 3D in a natural way.
\subsection{Convergence analysis for the nonlinear CH equation}
In this section we present the convergence result for the PinT algorithm \eqref{modelproblem3_wr}. First we discuss the convergence at continuous setting.
\begin{theorem}[Error estimate at continuous level]\label{thm0_nch}
Let ${u}^k$ be the $k$-th iterate of the PinT algorithm \eqref{modelproblem3_wr}, and $u_0, u_0^0\in L^2$. Then ${u}^k$ converges to $u$ of \eqref{modelproblem3} in $L^{\infty}(0,T,L^2)$.
\end{theorem}
\begin{proof}
Let $e^k(x,t)=u^k(x,t)-u(x,t)$ be the error at $k$-th iteration, then we have the following error equation
\begin{equation}\label{err_nch1}
\begin{cases}
    \frac{\partial e^k}{\partial t} = \Delta\left( f(u^k)-f(u)\right) - \epsilon^2 \Delta^2 e^k, & (x,t)\in\Omega\times(0,T],\\
   e^k(x,0)=\alpha( e^k(x,T)- e^{k-1}(x,T)) , & x\in\Omega.
\end{cases}
\end{equation}
Multiplying the 1st equation \eqref{err_nch1} with $e^k$ and integrate over the domain $\Omega$ we have 
\begin{equation}\label{err_nch2}
\begin{aligned}[b]
\frac{1}{2}\frac{d}{dt} \parallel e^k\parallel_2^2  & =\int_{\Omega} \nabla\left( f(u^k)-f(u)\right)\nabla e^k - \epsilon^2 \parallel \Delta e^k\parallel_2^2\\
& \leq M \parallel e^k\parallel_2 \parallel \nabla e^k\parallel_2- \epsilon^2 \parallel \Delta e^k\parallel_2^2\\
& \leq \frac{M}{2}\left(\parallel e^k\parallel_2^2 + \parallel \nabla e^k\parallel_2^2 \right)-\epsilon^2 \parallel \Delta e^k\parallel_2^2\\
& \leq \frac{M}{2}\left(\parallel e^k\parallel_2^2 + \parallel \Delta e^k\parallel_2 \parallel e^k\parallel_2 \right)-\epsilon^2 \parallel \Delta e^k\parallel_2^2\\
& \leq C^* \parallel e^k\parallel_2^2- \frac{\epsilon^2}{2} \parallel \Delta e^k\parallel_2^2 \leq C^* \parallel e^k\parallel_2^2,
\end{aligned}
\end{equation}
where in the first inequality we use Lipschitz condition of $f$, in the second inequality we use the Cauchy-Schwarz inequality, in the third inequality we use the inequality $ \parallel \nabla e^k\parallel_2^2\leq \parallel \Delta e^k\parallel_2 \parallel e^k\parallel_2$, in the fourth inequality we use the fact that the term $\frac{M}{2}\parallel \Delta e^k\parallel_2 \parallel e^k\parallel_2 \leq \frac{M^2}{8\epsilon^2}\parallel e^k\parallel_2^2 + \frac{\epsilon^2}{2}\parallel \Delta e^k\parallel_2^2$, and $C^*= \left(\frac{M}{2}+\frac{M^2}{8\epsilon^2} \right)$.
From \eqref{err_nch2} we get the following 
\begin{equation}\label{err_nch3}
\parallel e^k(., t)\parallel_2 \leq \parallel e^k_0\parallel_2 e^{C^* t}.
\end{equation}
Now taking norm on the 2nd equation of \eqref{err_nch1} and using the inequality \eqref{err_nch3} at $t=T$ we obtain the following recurrence relation
\begin{equation}\label{err_nch4}
\parallel e^k_0\parallel_2 \leq \gamma \parallel e^{k-1}_0\parallel_2, \text{where}\; \gamma= \frac{\vert \alpha\vert e^{C^* T}}{1-\vert \alpha\vert e^{C^* T}}.
\end{equation}
We consider the sum $\sum_{k=1}^{N} \parallel e^k(.,T)\parallel_2$ and show that sum is independent of $k$ for $\vert \alpha\vert < \frac{1}{2e^{C^* T}}$. Using \eqref{err_nch4} in \eqref{err_nch3} at $t=T$ we get the following
\begin{equation}\label{err_nch5}
\sum_{k=1}^{N} \parallel e^k(.,T)\parallel_2  \leq  e^{C^* T}\sum_{k=1}^{N} \parallel e^k_0\parallel_2
 \leq e^{C^* T} \parallel e^0_0\parallel_2 \sum_{k=1}^{N} \gamma^k .
\end{equation}
For $\vert \alpha\vert < \frac{1}{2e^{C^* T}}$ we have $\gamma<1$, so the series $\sum_{k=1}^{\infty} \gamma^k$ is convergent. Hence 
from \eqref{err_nch5} we have $\parallel e^k(.,T)\parallel_2 \rightarrow 0$ as $k\rightarrow\infty$. Hence we have the result.
\end{proof}

Next we show the convergence behaviour of the PinT algorithm \eqref{semi-dis_CH_WR} for the CH equation in the semi-discrete setting. We assume  $f$ to be Lipschitz with Lipschitz constant $M$ in \eqref{Lipchitz_CH} and we know Lipschitz implies \textit{one-sided} Lipschitz. So in the semi-discrete context, we get $f$ to be  \textit{one-sided} Lipschitz satisfying 
\begin{equation}\label{onesidedlip_f}
\langle f(\textbf{u}_1) - f(\textbf{u}_2), \textbf{u}_1 - \textbf{u}_2 \rangle \leq M \parallel\textbf{u}_1 -\textbf{u}_2\parallel_2^2, \forall t\in [0, T], \textbf{u}_1, \textbf{u}_2 \in\mathbb{R}^{N_x},
\end{equation}
where $\langle \cdot\rangle$ is the Euclidean inner product.
Now using \eqref{onesidedlip_f} and the fact that $\frac{\langle \Delta_h^2 y, y\rangle}{\parallel y \parallel_2^2}=\lambda_p^2\geq\frac{1}{\sqrt{N_x}}\parallel \Delta_h^2\parallel_{\infty}$, $\forall y\in\mathbb{R}^{N_x}$, it is easy to see that $H$ is also \textit{one-sided} Lipschitz, satisfying the following \textit{one-sided} Lipschitz condition
\begin{equation}\label{one-sided_Lip}
\langle H(t, \textbf{u}_1) - H(t, \textbf{u}_2), \textbf{u}_1 - \textbf{u}_2 \rangle \leq L^* \parallel\textbf{u}_1 -\textbf{u}_2\parallel_2^2, \forall t\in [0, T], \textbf{u}_1, \textbf{u}_2 \in\mathbb{R}^{N_x},
\end{equation}
where $L^*=M\parallel \Delta_h\parallel_{\infty} -\frac{\epsilon^2}{\sqrt{N_x}}\parallel \Delta_h^2\parallel_{\infty} = \frac{4M}{h^2}-\frac{16 \epsilon^2}{{\sqrt{N_x}}h^4}<0$.  Observe that $L^*<0$ imply $ h^{3/2}<\frac{4 \epsilon^2}{M}$ by taking $N_x\sim O(\frac{1}{h})$. As $L^*<0$, we take $L^*=-L$, where $L=-M\parallel \Delta_h\parallel_{\infty} +\frac{\epsilon^2}{\sqrt{N_x}}\parallel \Delta_h^2\parallel_{\infty}>0$. Then the revised \textit{one-sided} Lipschitz condition for $H$ is the following
\begin{equation}\label{revised_one-sided_Lip}
\langle H(t, \textbf{u}_1) - H(t, \textbf{u}_2), \textbf{u}_1 - \textbf{u}_2 \rangle \leq -L \parallel\textbf{u}_1 -\textbf{u}_2\parallel_2^2, \forall t\in [0, T], \textbf{u}_1, \textbf{u}_2 \in\mathbb{R}^{N_x}.
\end{equation}
\begin{theorem}[Error estimate at semi-discrete level]\label{thm1_nch}
Let $\textbf{u}^k$ be the $k$-th iteration of the PinT algorithm \eqref{semi-dis_CH_WR} with $\vert\alpha\vert<1$, and $H(t,\textbf{u})$ satisfies \textit{one-sided} Lipschitz condition \eqref{revised_one-sided_Lip} with the constant $L> 0$ defined above. Then the semi-discrete error $\textbf{e}^k(t)=\textbf{u}^k(t)-\textbf{u}(t)$ at $k$-th iteration satisfies the following convergence estimate
\begin{equation} 
\max_{t\in[0,T]} \parallel \textbf{e}^k(t)\parallel_2 \leq \left(\frac{\vert\alpha\vert e^{-L T}}{1-\vert\alpha\vert e^{-L T}}\right)^k\parallel \textbf{e}^0(0)\parallel_2.
\end{equation}
\end{theorem}
\begin{proof}
The proof follows from the Theorem 5.1 in \cite{gander2019shulin}.
\end{proof}
We now show the convergence behaviour of the PinT algorithm \eqref{fully-dis_CH_WR} for the CH equation at fully discrete level. Let $\textbf{e}_n^k=\textbf{u}_n^k - \textbf{u}_n$ be the error vector at each PinT iteration. 
\begin{theorem}[Error estimate at fully discrete level]\label{thm2_nch}
The fully discrete error $\textbf{e}_n^k$ at $k-$th iterate of the PinT algorithm \eqref{fully-dis_CH_WR} with $\vert\alpha\vert<1$ satisfies the following convergence estimate 
\begin{equation*}
\max\limits_{n=1,2,\cdots,N_t}\parallel \textbf{e}_n^k \parallel_2\leq \rho^k \parallel \textbf{e}_0^0 \parallel_2, \rho=\frac{\vert\alpha\vert C_1}{1-\vert\alpha\vert C_1},
\end{equation*}
where the explicit expression of $C_1$ is given in \eqref{inequality_3}. 
\end{theorem}
\begin{proof}
From \eqref{fully-dis_CH} and \eqref{fully-dis_CH_WR}, the error $\textbf{e}_n^k$ satisfies
\begin{equation}\label{err_ch}
  \begin{cases}
    \frac{\textbf{e}_n^k - \textbf{e}_{n-1}^k}{\Delta t} = H(t_n, \textbf{u}_{n}^k)-H(t_n, \textbf{u}_{n}), & n=1,2,\cdots, N_t \\
    \textbf{e}_{0}^k=\alpha (\textbf{e}_{N_t}^k -\textbf{e}_{N_t}^{k-1}).
  \end{cases}
\end{equation}
For the Euclidean inner product, we have 
\begin{equation}\label{inequality_1}
\left\langle \frac{\textbf{e}_n - \textbf{e}_{n-1}}{\Delta t} , \textbf{e}_n\right\rangle \geq \frac{\parallel\textbf{e}_n\parallel_2^2 - \parallel\textbf{e}_{n-1}\parallel_2^2}{2\Delta t},
\end{equation}
for the error vector $\textbf{e}_n\in\mathbb{R}^{N_x}$, for all $n=1,2,\cdots, N_t$.
Taking inner product in \eqref{err_ch} with $\textbf{e}_n^k$, we get
\begin{equation}\label{discrete_error_CH}
\left\langle \frac{\textbf{e}_n^k - \textbf{e}_{n-1}^k}{\Delta t} , \textbf{e}_n^k\right\rangle = \left\langle H(t_n, \textbf{u}_{n}^k)-H(t_n, \textbf{u}_{n}), \textbf{u}_n^k - \textbf{u}_n\right\rangle.
\end{equation}
Then using the inequality \eqref{inequality_1} and \textit{one-sided} Lipschitz condition \eqref{revised_one-sided_Lip} for discrete time points in the equation \eqref{discrete_error_CH} we obtain
\begin{equation}\label{inequality_2}
\frac{\parallel\textbf{e}_n^k\parallel_2^2 - \parallel\textbf{e}_{n-1}^k\parallel_2^2}{2\Delta t} \leq -L\parallel\textbf{e}_n^k\parallel_2^2  
\end{equation}
Inequality \eqref{inequality_2} yields 
\begin{equation}\label{inequality_3}
\parallel\textbf{e}_{N_t}^k\parallel_2 \leq \sqrt{\left( \frac{1}{1+2L\Delta t}\right)^{N_t} } \parallel\textbf{e}_0^k\parallel_2  =  C_1\parallel\textbf{e}_0^k\parallel_2.
\end{equation}
Now from the second equation of \eqref{err_ch} we have $\parallel\textbf{e}_{0}^k\parallel_2=\vert\alpha\vert \parallel\textbf{e}_{N_t}^k\parallel_2 + \vert\alpha\vert \parallel\textbf{e}_{N_t}^{k-1}\parallel_2 \leq \vert\alpha\vert C_1 \parallel\textbf{e}_0^k\parallel_2 + \vert\alpha\vert C_1 \parallel\textbf{e}_0^{k-1}\parallel_2$, where on the second inequality we use \eqref{inequality_3} twice. Hence we have 
\begin{equation}\label{inequality_4}
\parallel\textbf{e}_0^k\parallel_2 \leq \frac{\vert\alpha\vert C_1}{1-\vert\alpha\vert C_1} \parallel\textbf{e}_0^{k-1}\parallel_2
\end{equation}
From \eqref{inequality_1},  we have $\max\limits_{n=1,2,\cdots,N_t}\parallel \textbf{e}_n^k \parallel_2\leq  \parallel \textbf{e}_0^0 \parallel_2$ and together with \eqref{inequality_4} gives the required convergence estimate.
\end{proof}

\section{Numerical Illustration}\label{Section5}
In this section we present the numerical experiments of the PinT algorithms for the linear and non-linear time dependent fourth order PDEs, which are mentioned earlier. For the PinT algorithm, the iterations start from a random initial guess and stop as the error measured in $\parallel u-u^k\parallel_{L^{\infty}(0,T;L^2(\Omega))}$ or $\parallel u-u^k\parallel_{L^{\infty}(0,T;L^{\infty}(\Omega))}$ reaches a tolerance of $1e {-10}$, where $u$ is the discrete sequential solution and $u^k$ is the discrete PinT solution at $k-$th iteration. We  have taken the random initial solution to perform numerical experiments for the PDEs under consideration.

\subsection{Biharmonic problem}
For numerical experiments of \eqref{modelproblem1a} in 1D we consider the domain $\Omega=(0,1)$. In Figure \ref{fig1}-\ref{fig2} we have shown the error curves and theoretical estimates by fixing $T=1, \Delta t=0.001$ and the PinT parameter $\alpha=0.001$ and varying mesh size $h$ and time integrator. The error measured in Figure \ref{fig1} is in the norm $L^{\infty}(0,T;L^2(\Omega))$ as in Theorem \ref{thm1} and we say the corresponding theoretical error bound as 'continuous bound'. The error measured in Figure \ref{fig2} is in the norm $L^{\infty}(0,T;L^{\infty}(\Omega))$ as in Theorem \ref{thm2} and Theorem \ref{thm3} and we say the corresponding theoretical error bound as 'discrete bound'. From Figure \ref{fig1} and Figure \ref{fig2} we observe that continuous bound is more accurate in Backward-Euler whereas discrete bound is more precise in Trapezoidal rule. From Figure \ref{fig1} and Figure \ref{fig2} it is also clear that the proposed PinT method is mesh independent.
\begin{figure}[h!]
    \centering %
\begin{subfigure}{0.3\textwidth}
  \includegraphics[height=3.5cm, width=4cm]{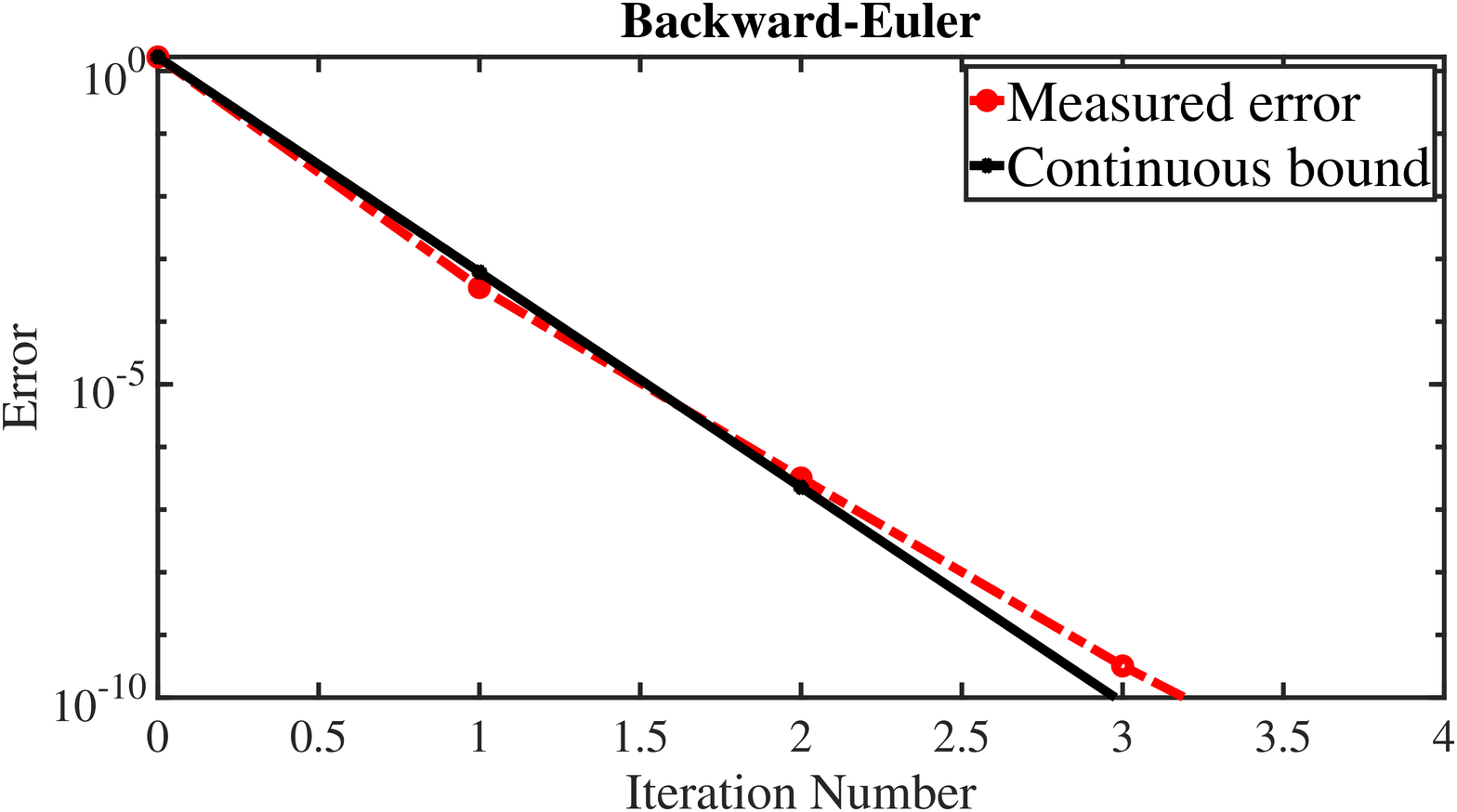}
\end{subfigure}\hfil 
\begin{subfigure}{0.3\textwidth}
  \includegraphics[height=3.5cm, width=4cm]{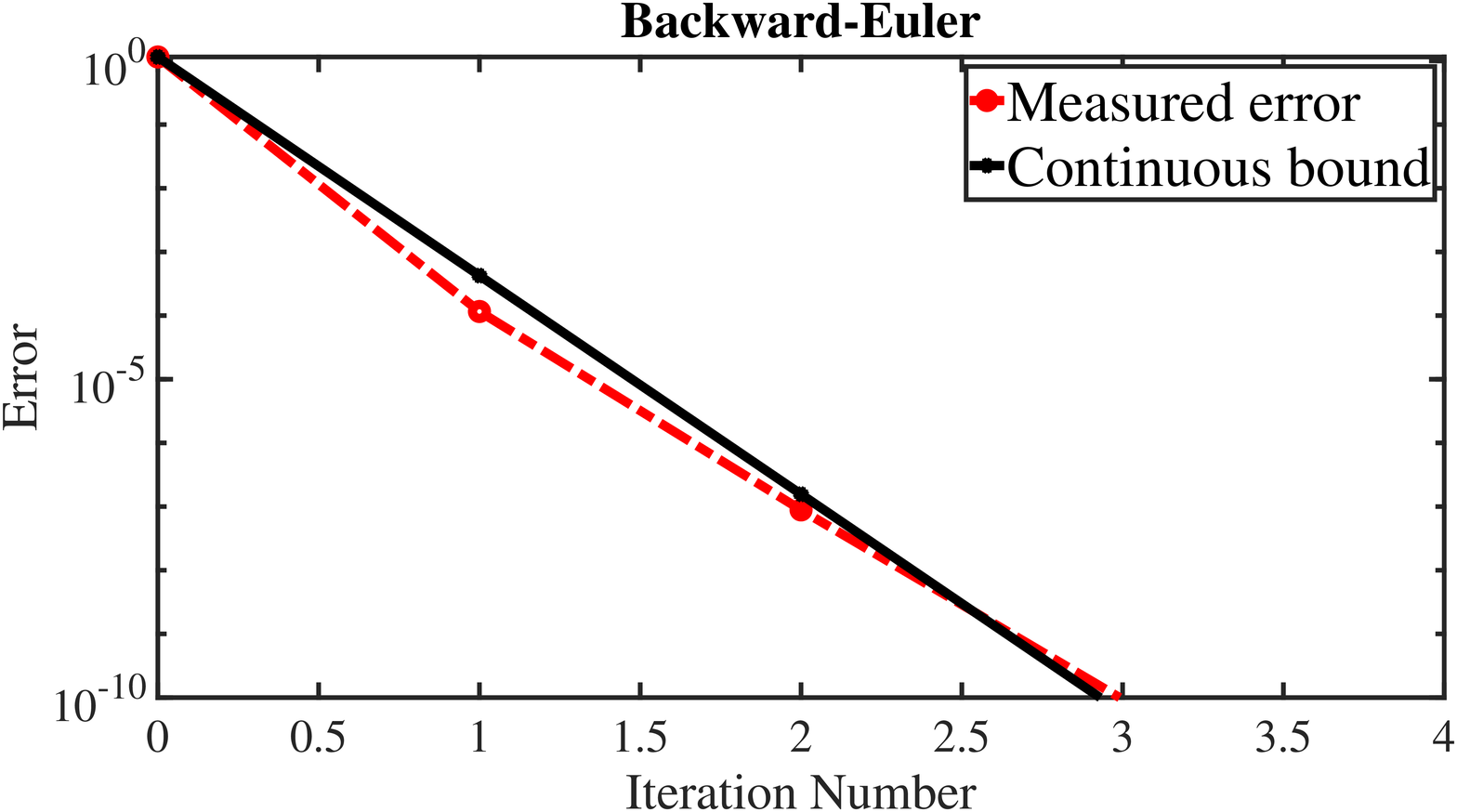}
\end{subfigure}\hfil 
\begin{subfigure}{0.3\textwidth}
  \includegraphics[height=3.5cm, width=4cm]{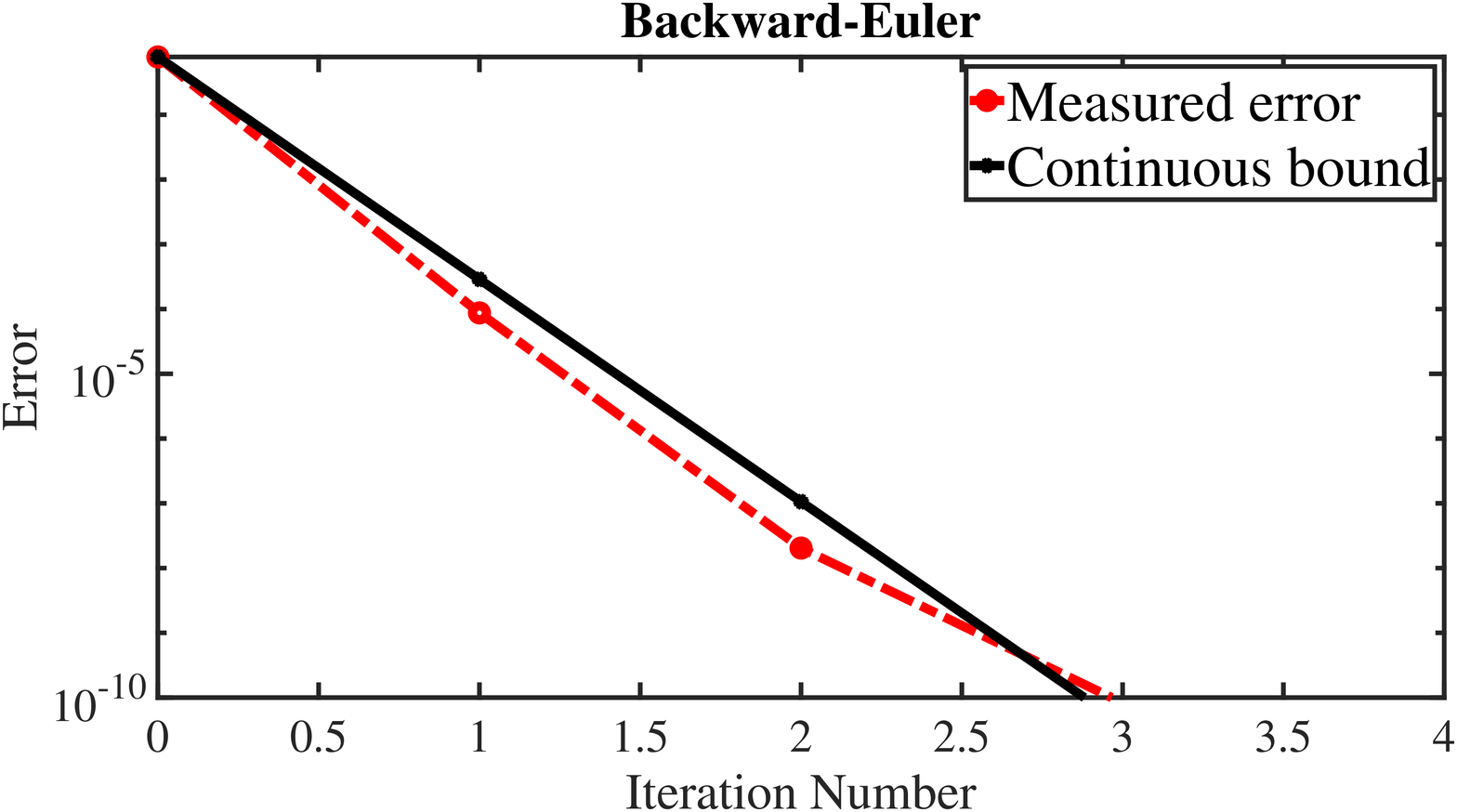}
\end{subfigure}

\medskip
\begin{subfigure}{0.3\textwidth}
  \includegraphics[height=3.5cm, width=4cm]{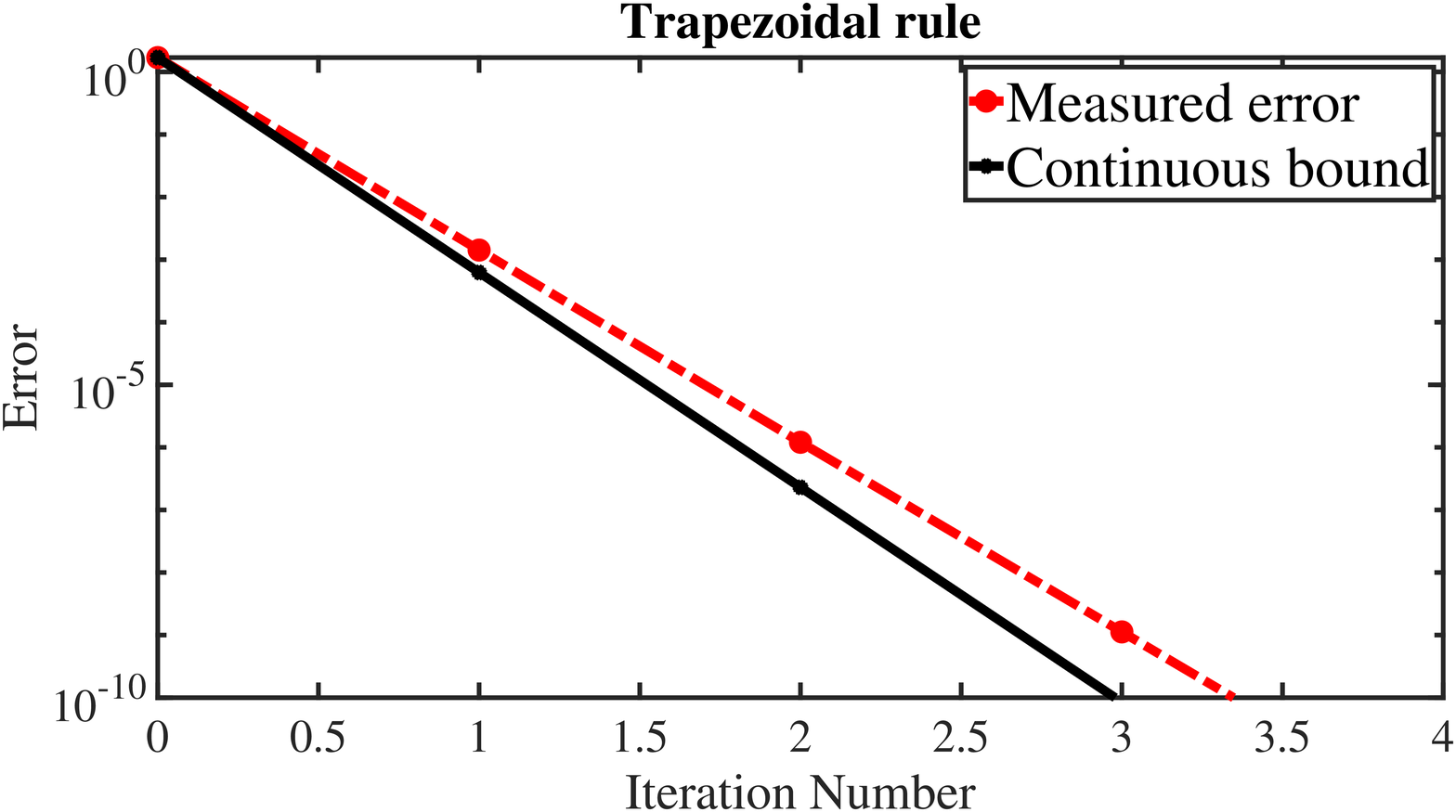}
\end{subfigure}\hfil 
\begin{subfigure}{0.3\textwidth}
  \includegraphics[height=3.5cm, width=4cm]{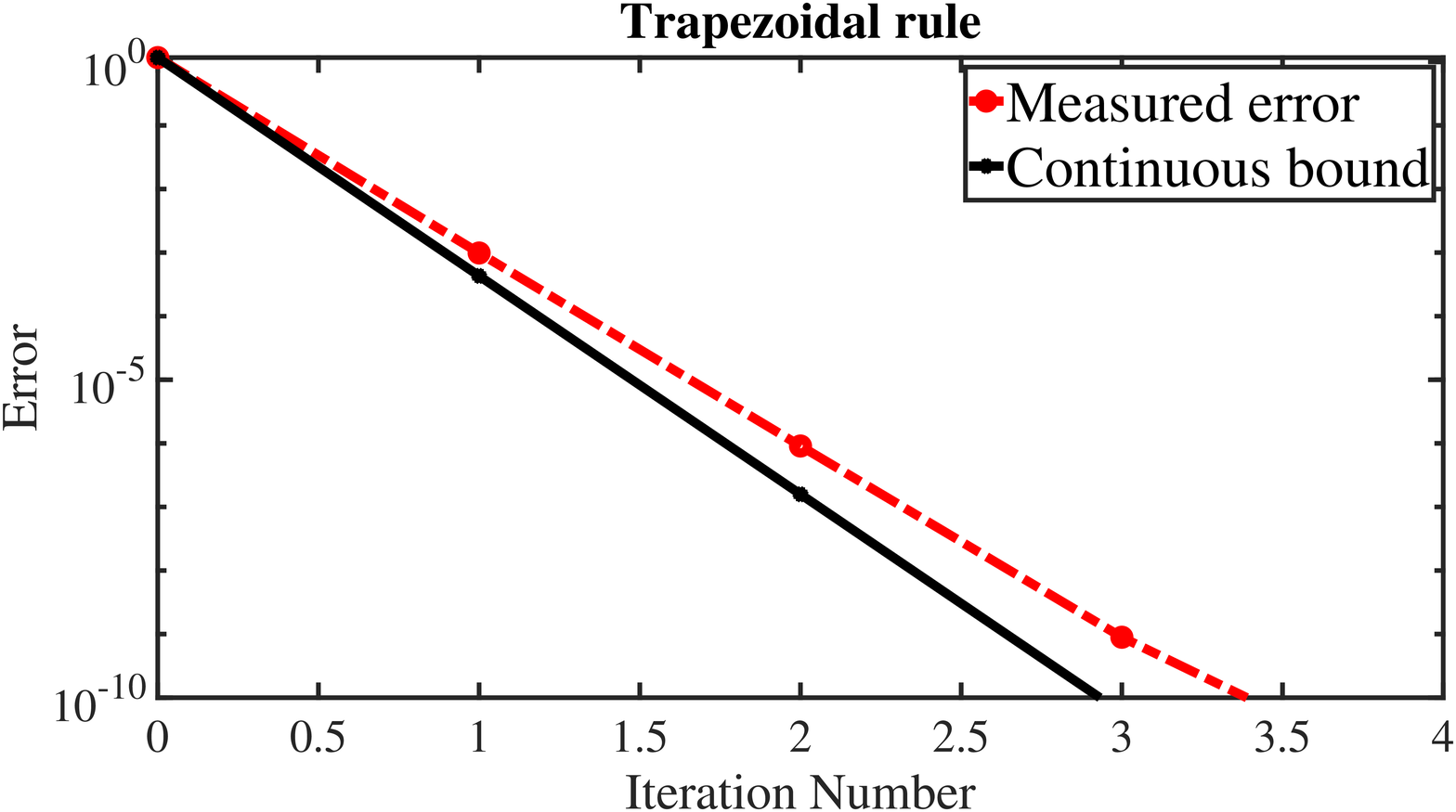}
\end{subfigure}\hfil 
\begin{subfigure}{0.3\textwidth}
  \includegraphics[height=3.5cm, width=4cm]{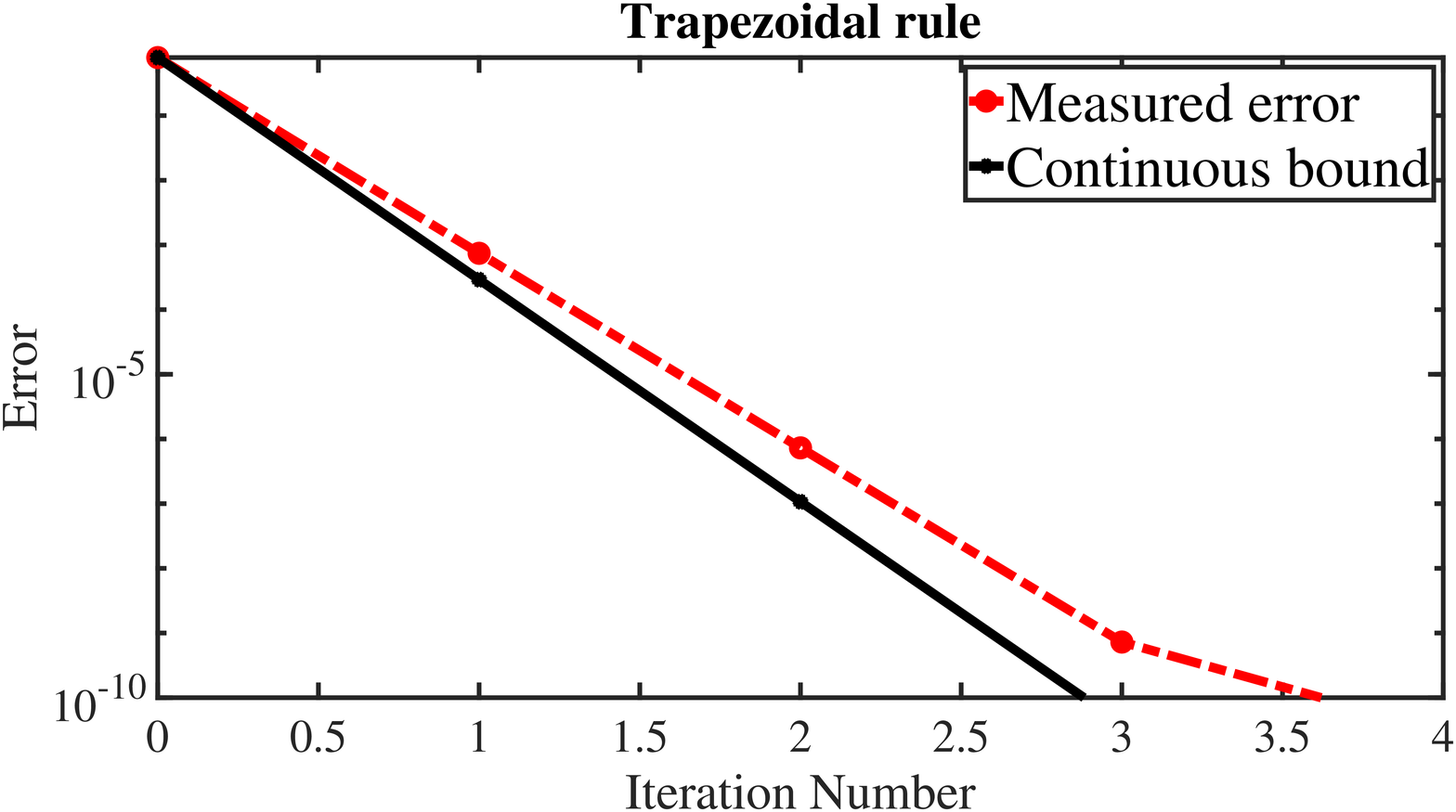}
\end{subfigure}
\caption{Comparison of numerical error and theoretical error estimates for the biharmonic problem with different mesh size $h$ and fixed PinT parameter $\alpha=0.001$. From left to right:  $h=1/64, h=1/128$ and $h=1/256$. Top row: $\theta=1$; Bottom row: $\theta=1/2$. }
\label{fig1}
\end{figure}
\begin{figure}[h!]
    \centering
    \subfloat{{\includegraphics[height=4cm,width=3.5cm]{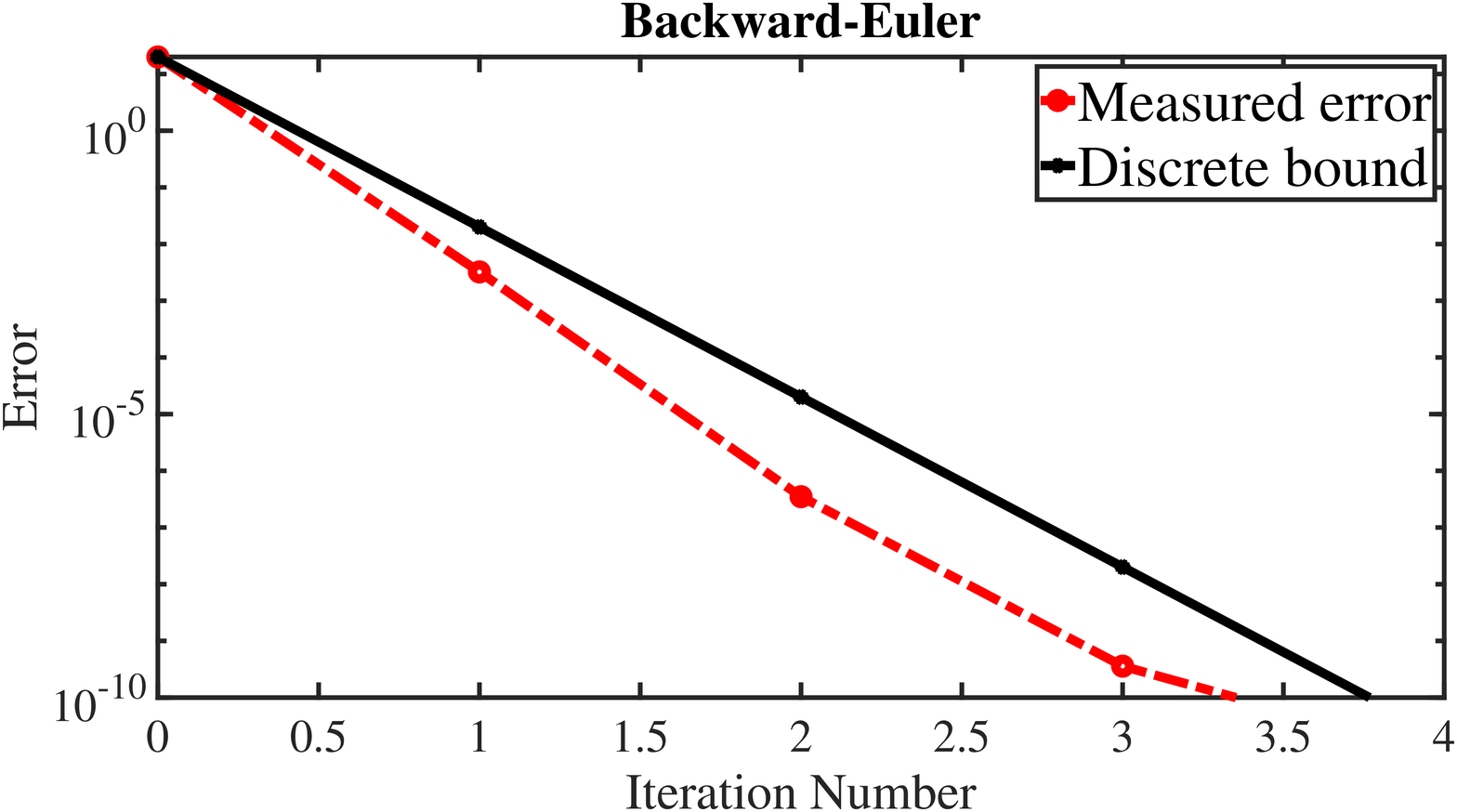} }}
     \subfloat{{\includegraphics[height=4cm,width=3.5cm]{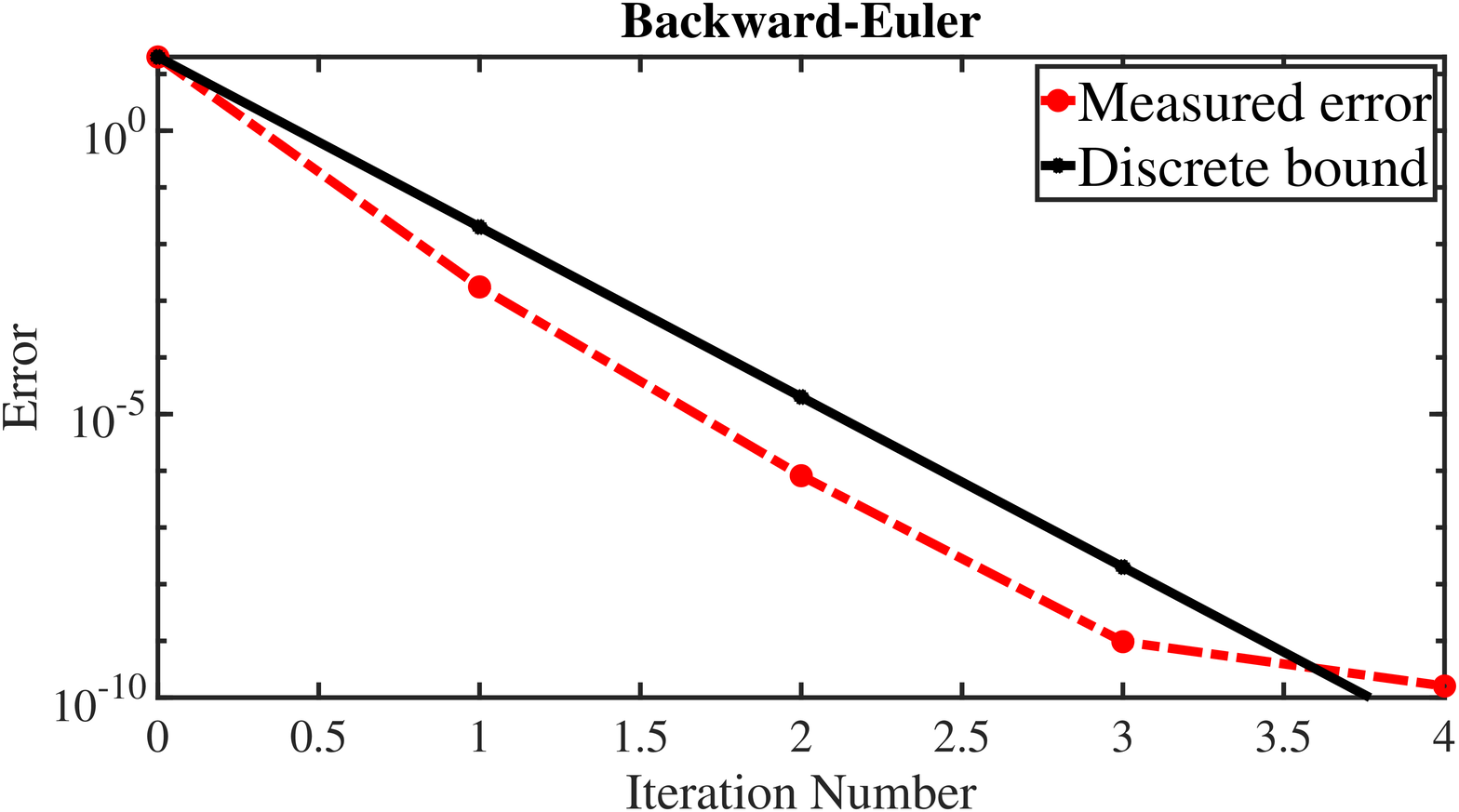} }}
     \subfloat{{\includegraphics[height=4cm,width=3.5cm]{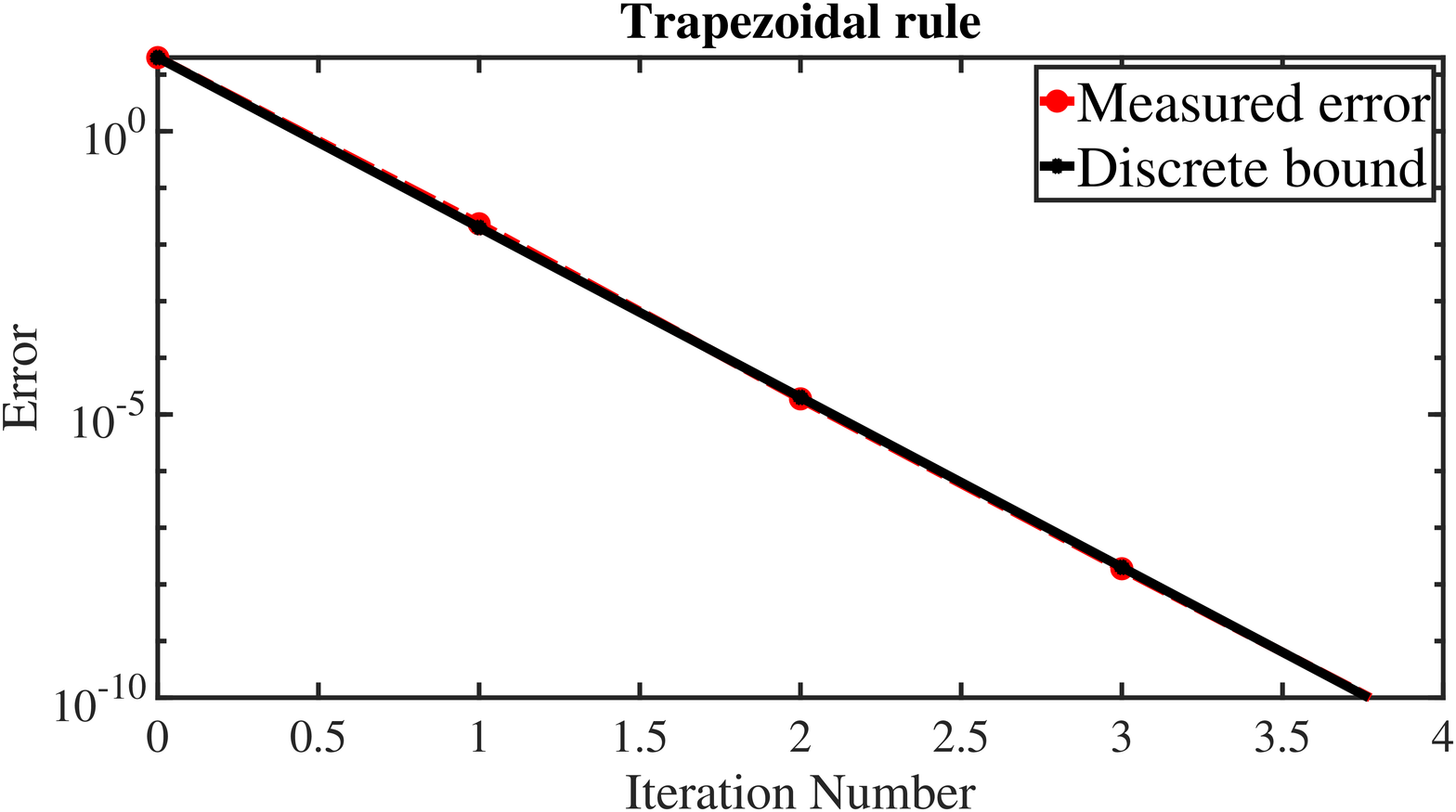} }}
     \subfloat{{\includegraphics[height=4cm,width=3.5cm]{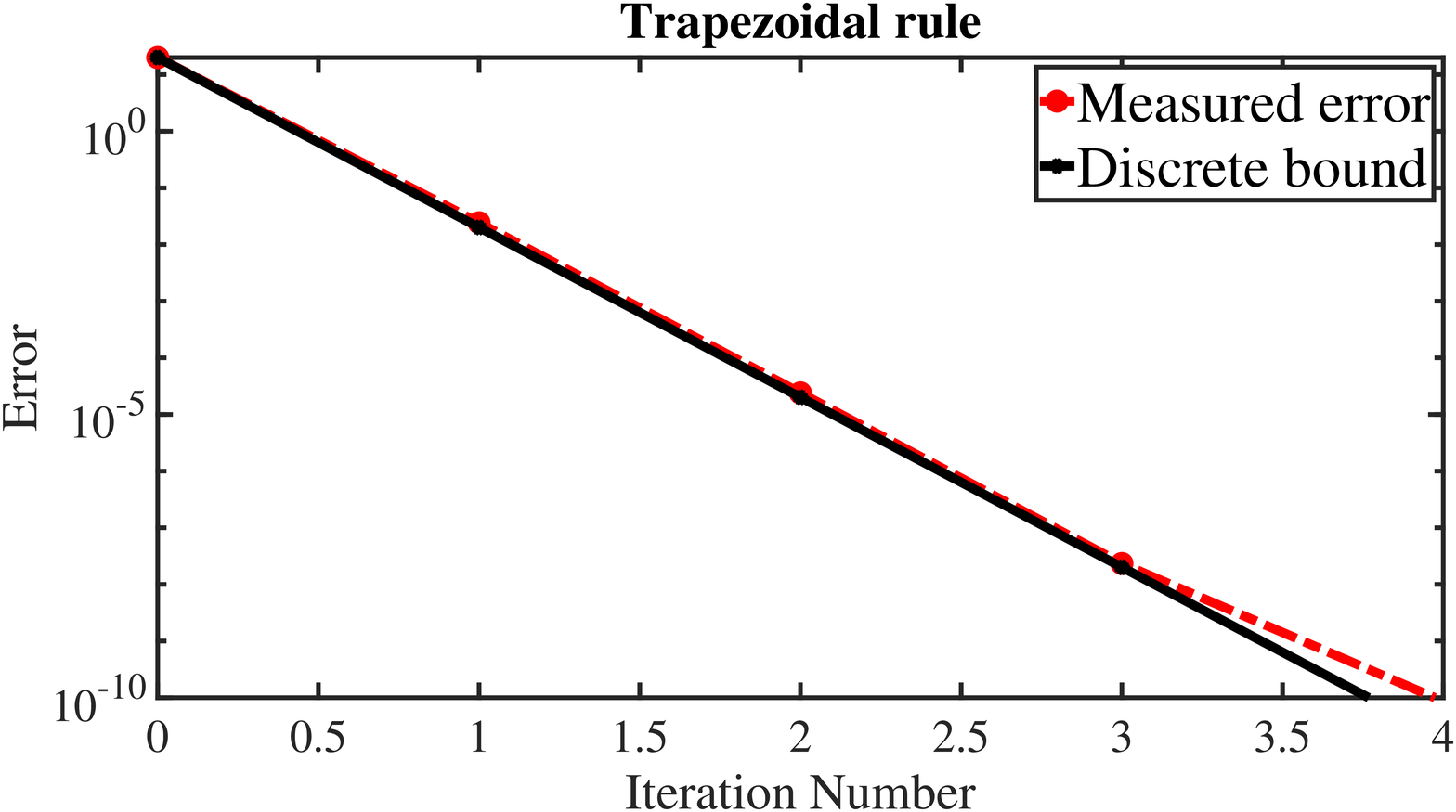} }}
    \caption{Comparison of theoretical error estimates and numerical error with different mesh size $h$. First (with $h=1/64$) and second (with $h=1/128$) figure: Backward-Euler; Third (with $h=1/64$) and fourth (with $h=1/128$) figure: Trapezoidal rule.}
    \label{fig2}
\end{figure}
Next we show how the convergence behaviour depends on $\alpha, \Delta t, T$ in terms of iteration count. The left plot of Figure \ref{fig3} is obtained by fixing $\Delta t=0.001, T=1, h=1/64$ and varying $\alpha$ for different time integrator. It is evident from the plot  that smaller $\alpha$ is good choice for the method to be robust. The middle plot of the Figure \ref{fig3} explains the dependence of the convergence on time step $\Delta t$. For example if we take $\Delta t=10^{-6}$, i.e., $N_t=10^{6}$, the method takes four iterations to converge, but at each iteration one can solve $N_t=10^{6}$ time steps in parallel. It also shows the independence of the temporal step size. The right plot of Figure \ref{fig3} describes the behaviour of PinT method for longer time window $T$ for fixed $\Delta t=0.001, \alpha=0.001, h=1/64$. In Figure \ref{fig3} both time integrators have almost same convergence behaviour.
\begin{figure}[h!]
    \centering
     \subfloat{{\includegraphics[height=4cm,width=4cm]{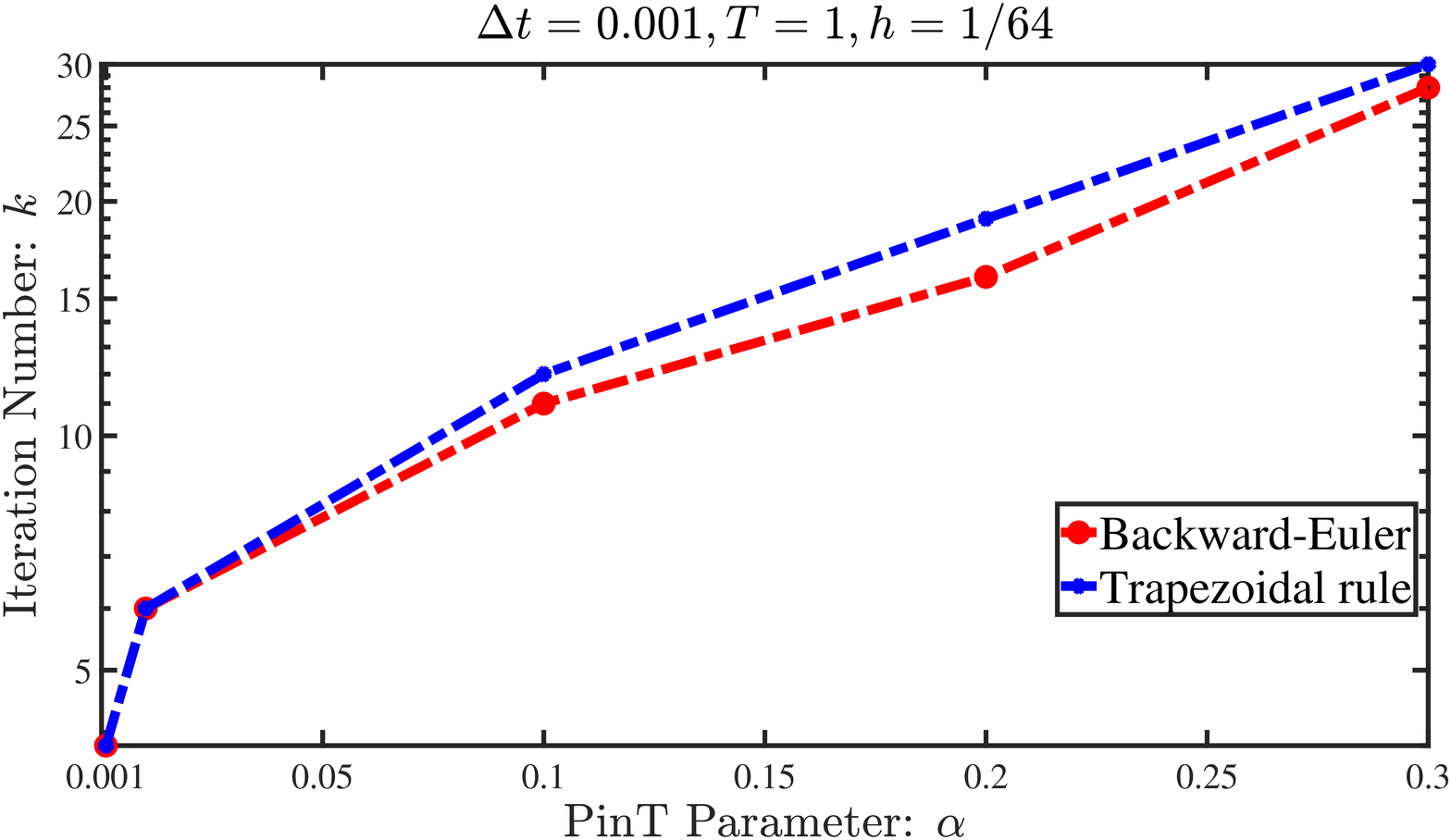} }}
     \subfloat{{\includegraphics[height=4cm,width=4cm]{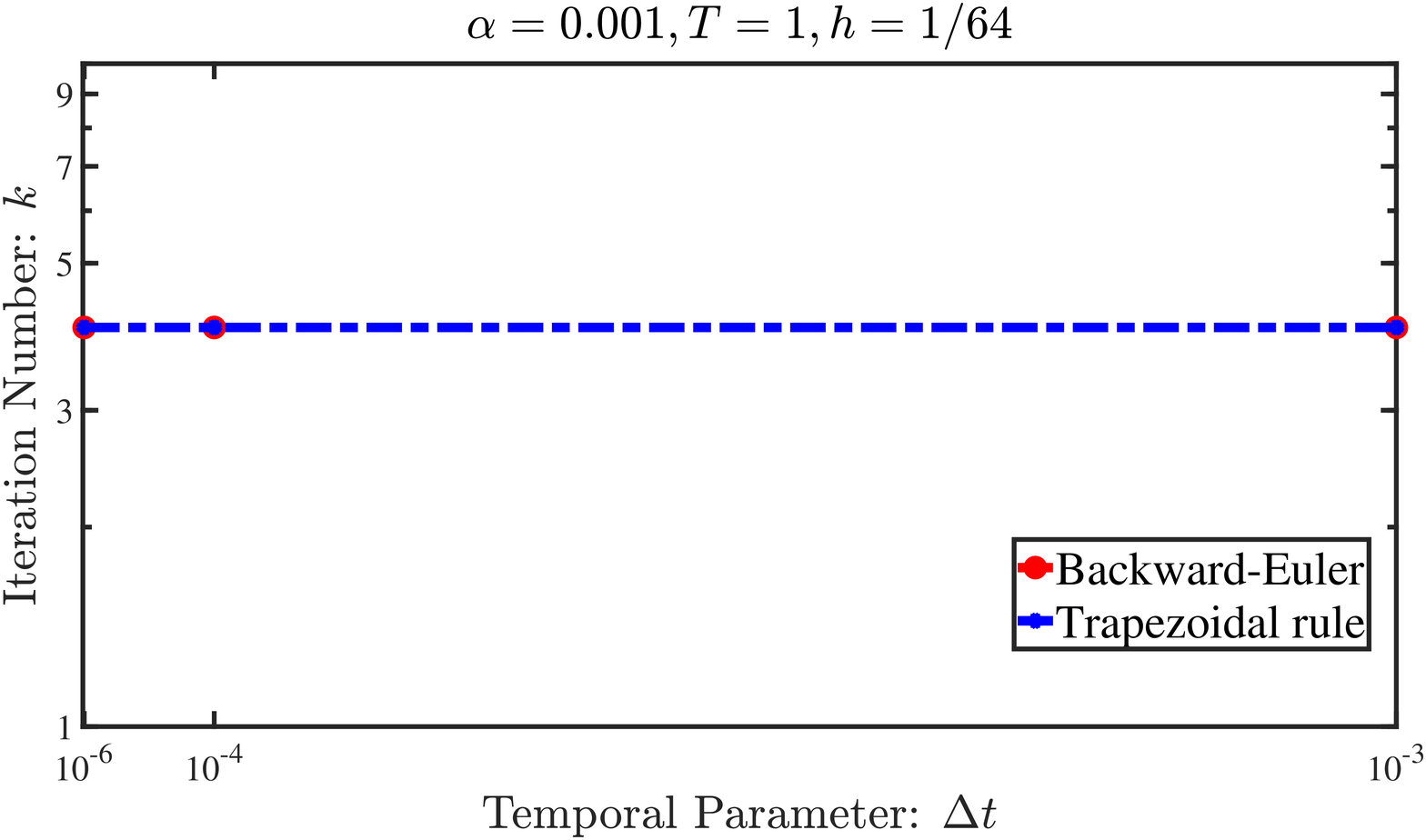} }}
     \subfloat{{\includegraphics[height=4cm,width=4cm]{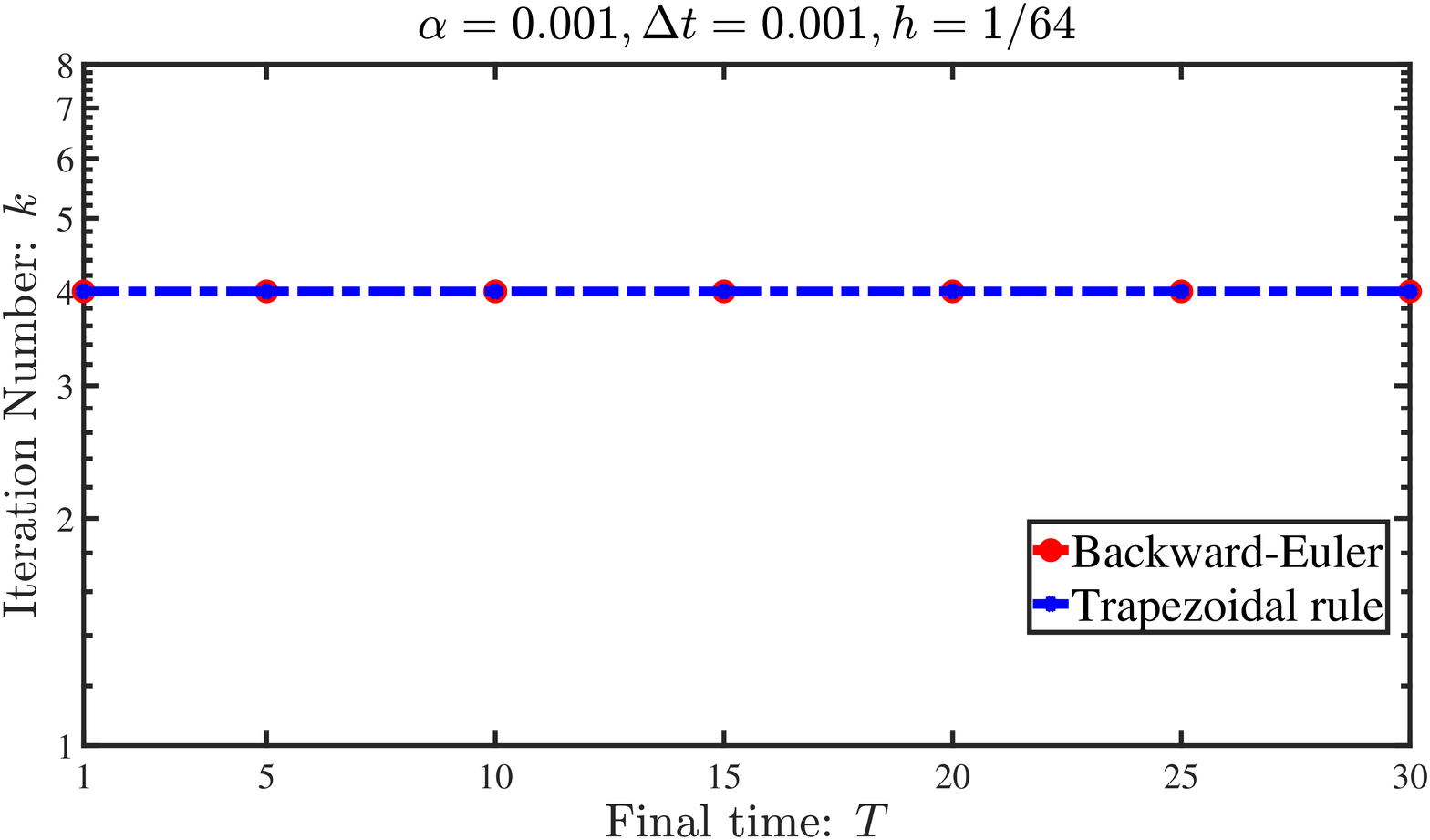} }}
    \caption{Dependency of convergence behaviour of PinT algorithm on $\alpha$ (left), $\Delta t$ (middle), $T$ (right).}
    \label{fig3}
\end{figure}
We also tested the PinT algorithm \eqref{pintiter1} with modified matrix $A$ (due to different boundary conditions) for the following boundary conditions
\[
(i)\; u=0, \partial_{\nu}u=0; (ii)\; u=0, \Delta u=0; (iii)\; \partial_{\nu}u=0, \Delta u=0,
\]
for which the problem \eqref{modelproblem1a} is also well-posed, and we observe similar numerical convergence behaviour.

In 2D for \eqref{modelproblem1a} we take the spatial domain $\Omega=(0, \pi)\times(0, \pi)$, with mesh size $h=1/10$ on both the direction. Figure \ref{fig4} shows the numerical error and theoretical error estimates given in Theorem \ref{thm3} by fixing $T=1, \Delta t=0.001$ and the PinT parameter $\alpha=0.001$. We observe analogous convergence behaviour in 2D for different choice of parameters ($h, \Delta t, \alpha$) as in the case of 1D experiments. Next we present numerical results for the linearized CH equation, which is more general to biharmonic operator. 
\begin{figure}[h!]
    \centering
     \subfloat{{\includegraphics[height=4cm,width=6cm]{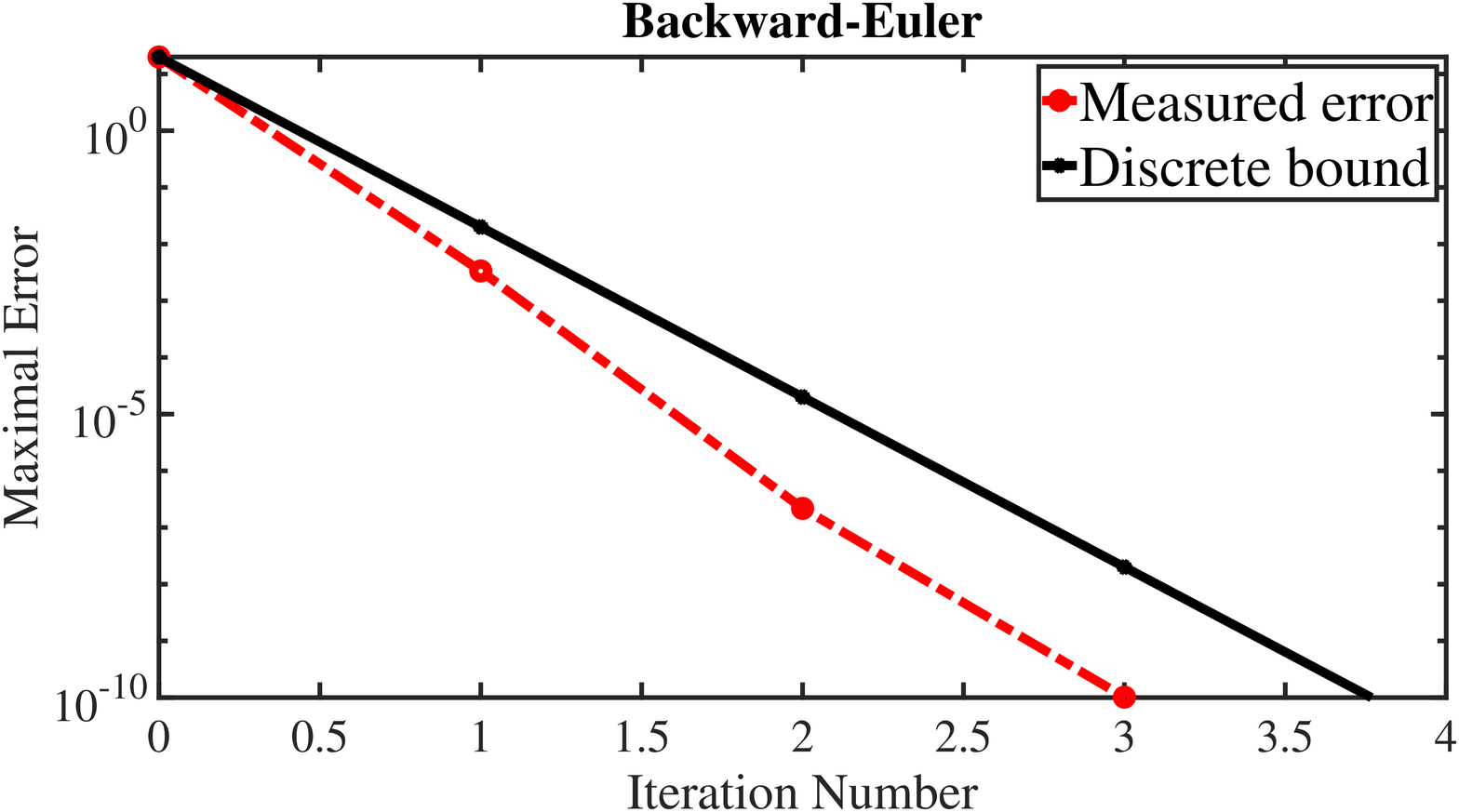} }}
     \subfloat{{\includegraphics[height=4cm,width=6cm]{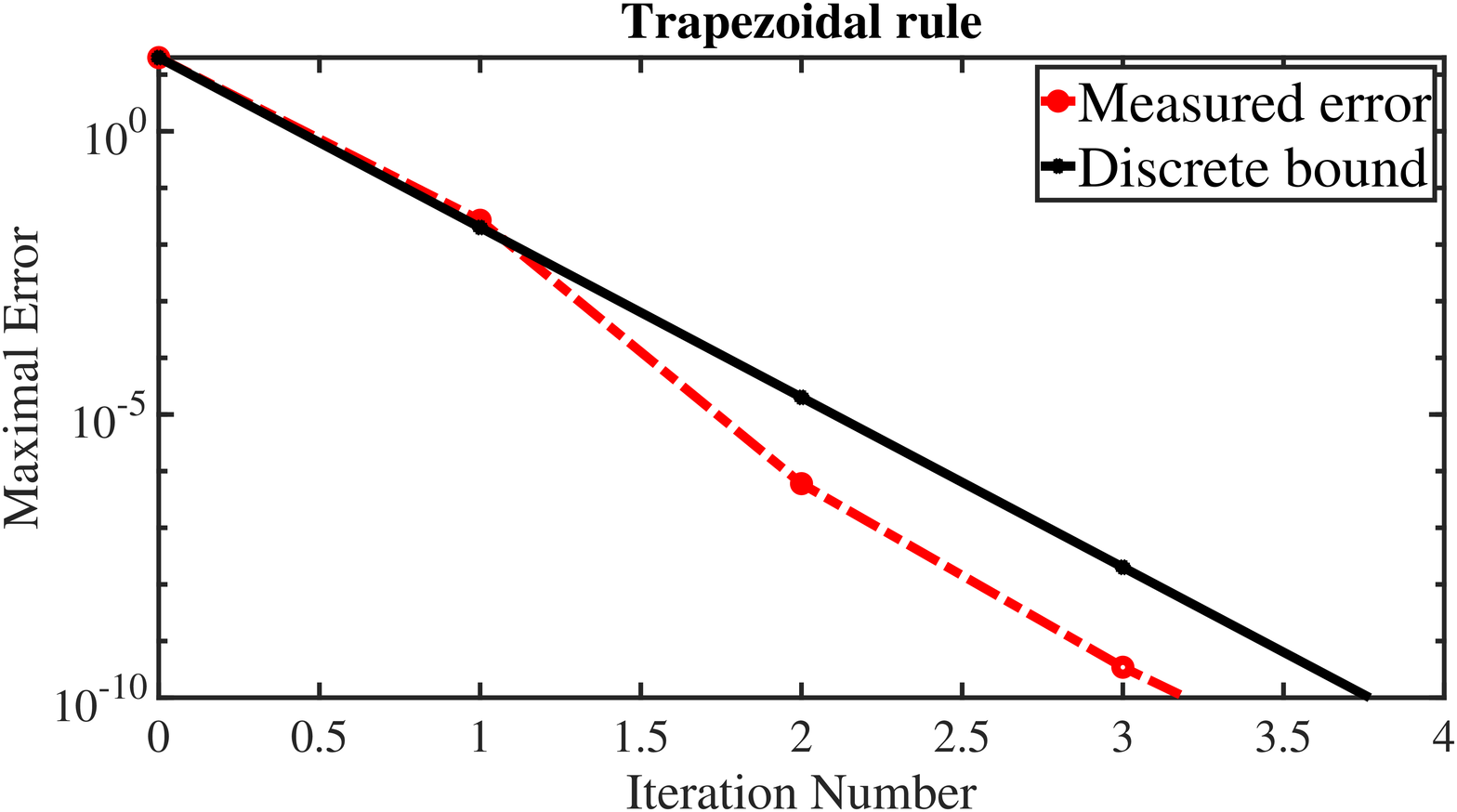} }}
    \caption{Comparison of numerical error and theoretical error bound for the biharmonic problem with $h=1/10, \Delta t=0.001, T=1$ and fixed PinT parameter $\alpha=0.001$. On the left: Backward-Euler; On the right: Trapezoidal rule.}
    \label{fig4}
\end{figure}

\subsection{Linearized CH equation}
We consider the Linearized CH equation \eqref{modelproblem2a} in 1D for the domain $\Omega=(0,1)$. Figure \ref{linch_fig1} shows the numerical error and theoretical error estimate by fixing $T=1, \Delta t=0.0001$ and the PinT parameter $\alpha=0.001$, and varying mesh size $h$ and time integrator. The error is measured in the norm $L^{\infty}(0,T;L^2(\Omega))$ as in Theorem \ref{thm4}. Dependency of the PinT parameter $\alpha$ for different time integrators can be seen in Figure \ref{linch_fig2}, and one can see that the algorithm yields two step convergence for very small choice of $\alpha$. For the experiment performed in Figure \ref{linch_fig1} and Figure \ref{linch_fig2} we take $N_t=10,000$, so that one computes $N_t$ steps at one go per iteration in parallel. In Figure \ref{linch_fig3} we see the dependency on  $\beta$ and final time $T$ of the time integrators. Both the time integrators are immune to choice of $\beta$, but Trapezoidal rule integrator is bit sensitive to long time simulation.
\begin{figure}[h!]
    \centering
    \subfloat{{\includegraphics[height=4cm,width=3.5cm]{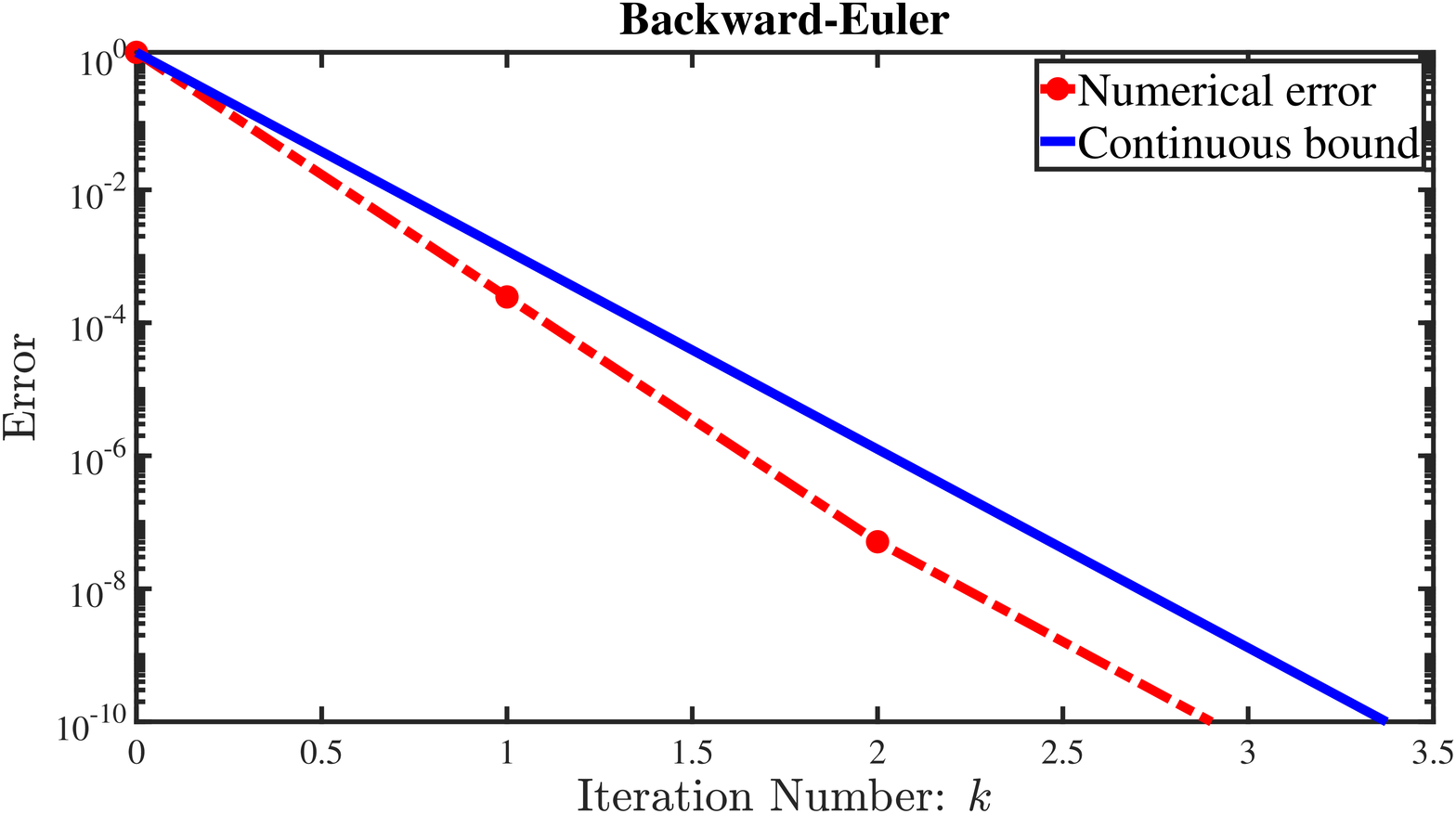} }}
     \subfloat{{\includegraphics[height=4cm,width=3.5cm]{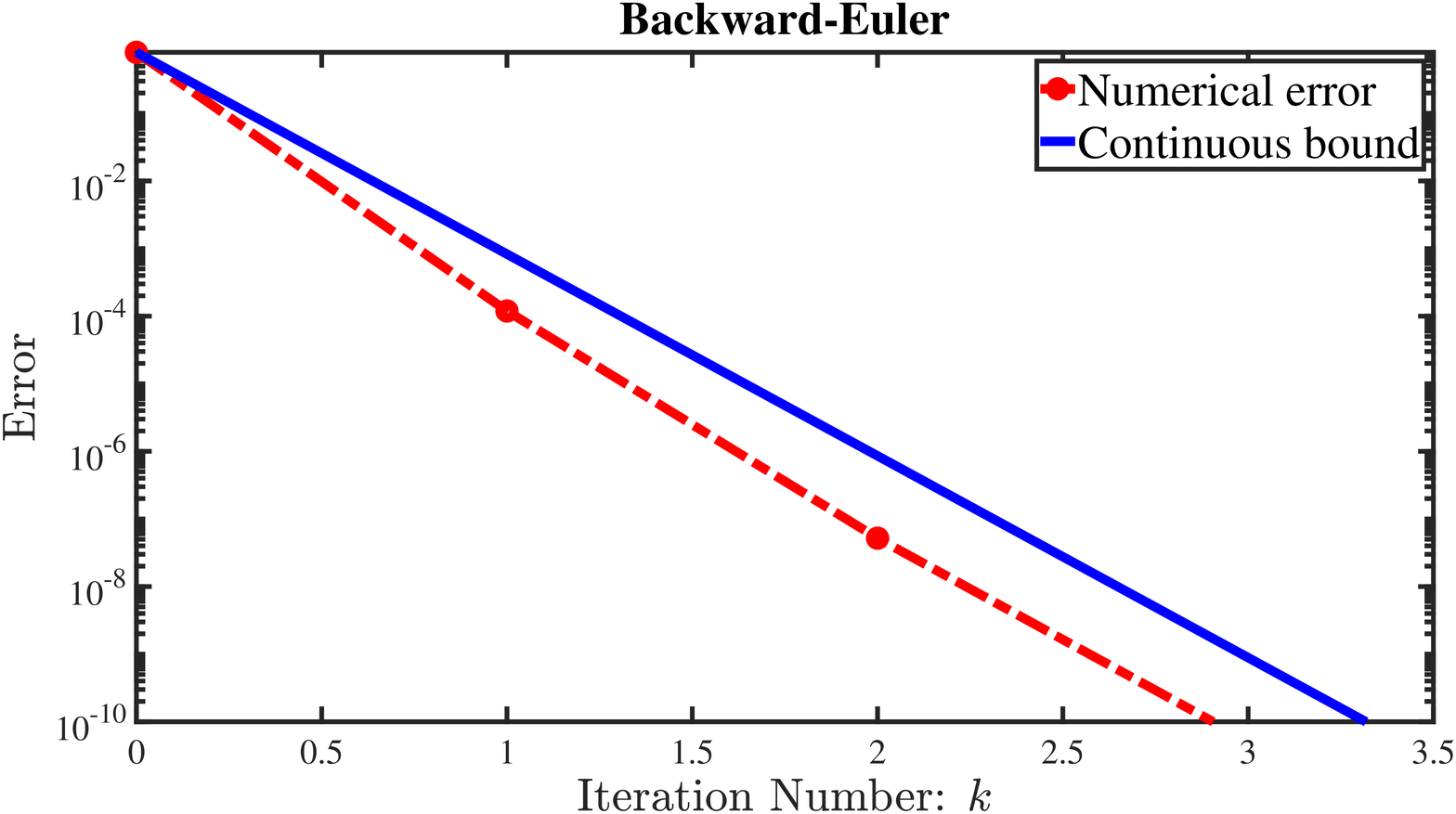} }}
     \subfloat{{\includegraphics[height=4cm,width=3.5cm]{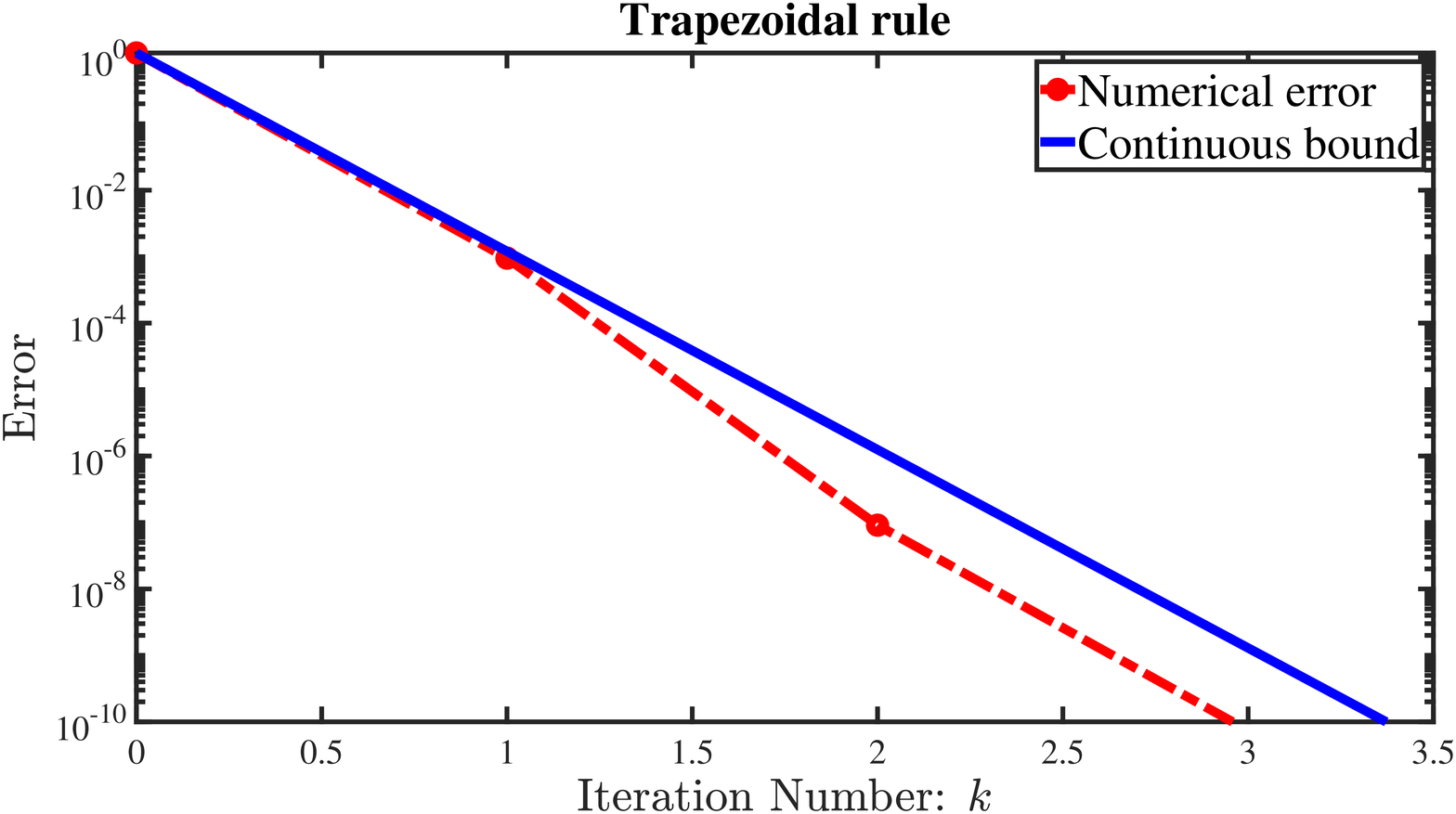} }}
     \subfloat{{\includegraphics[height=4cm,width=3.5cm]{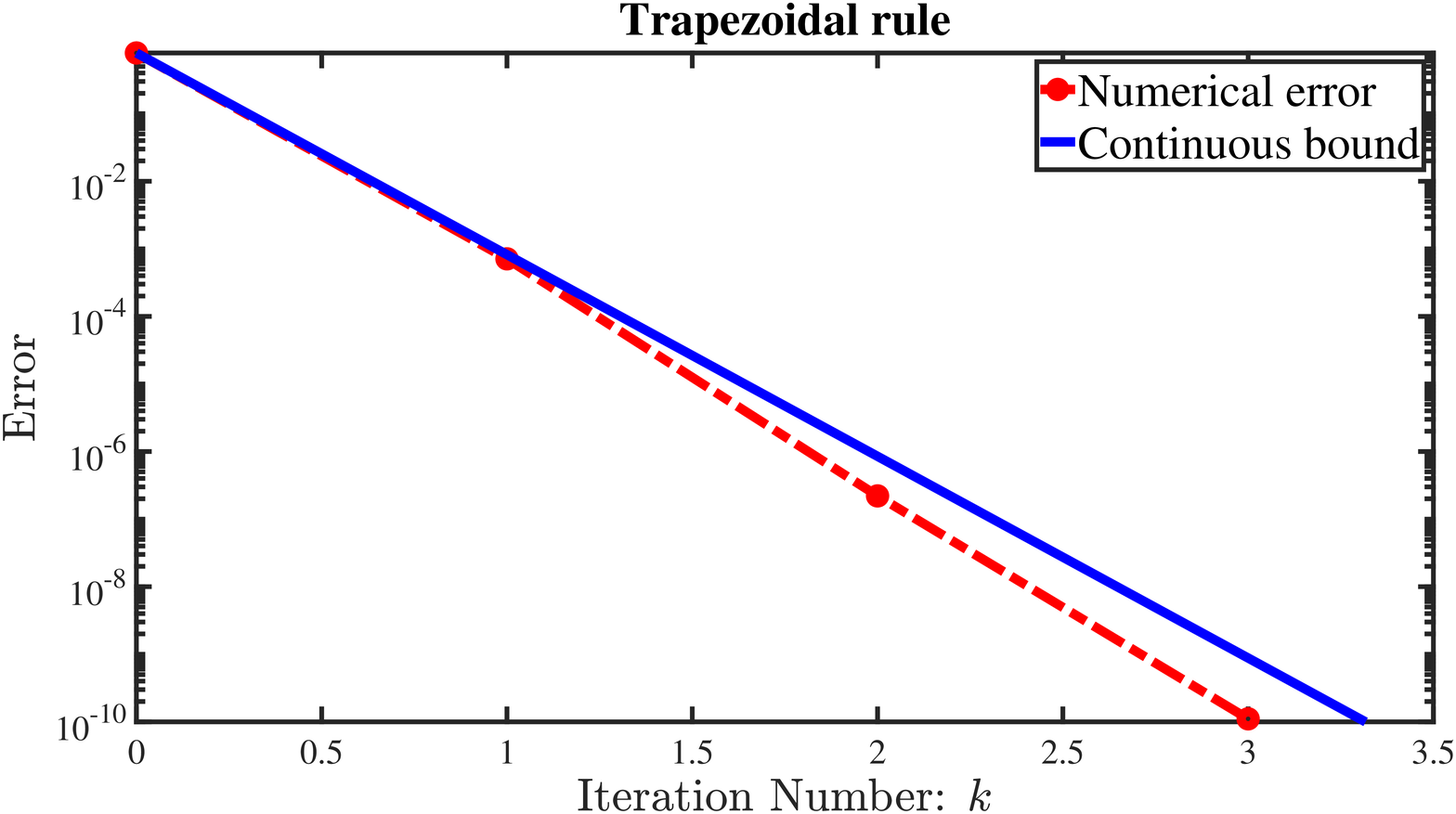} }}
    \caption{Comparison of theoretical estimates and numerical error with different mesh size $h$. First (with $h=1/128$) and second (with $h=1/256$) figure: Backward-Euler; Third (with $h=1/128$) and fourth (with $h=1/256$) figure: Trapezoidal rule.}
    \label{linch_fig1}
\end{figure}
\begin{figure}[h!]
    \centering
     \subfloat{{\includegraphics[height=4cm,width=6cm]{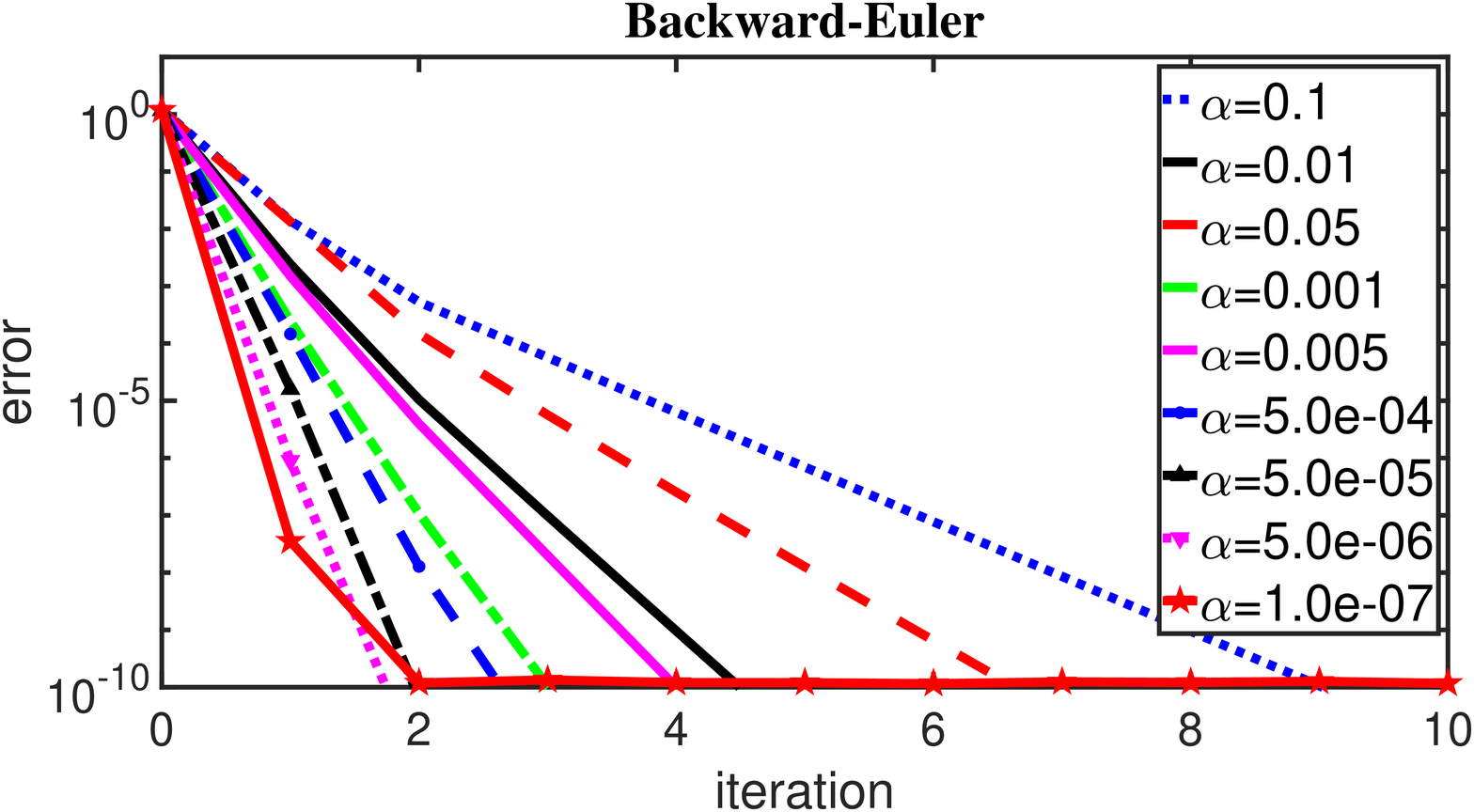} }}
     \subfloat{{\includegraphics[height=4cm,width=6cm]{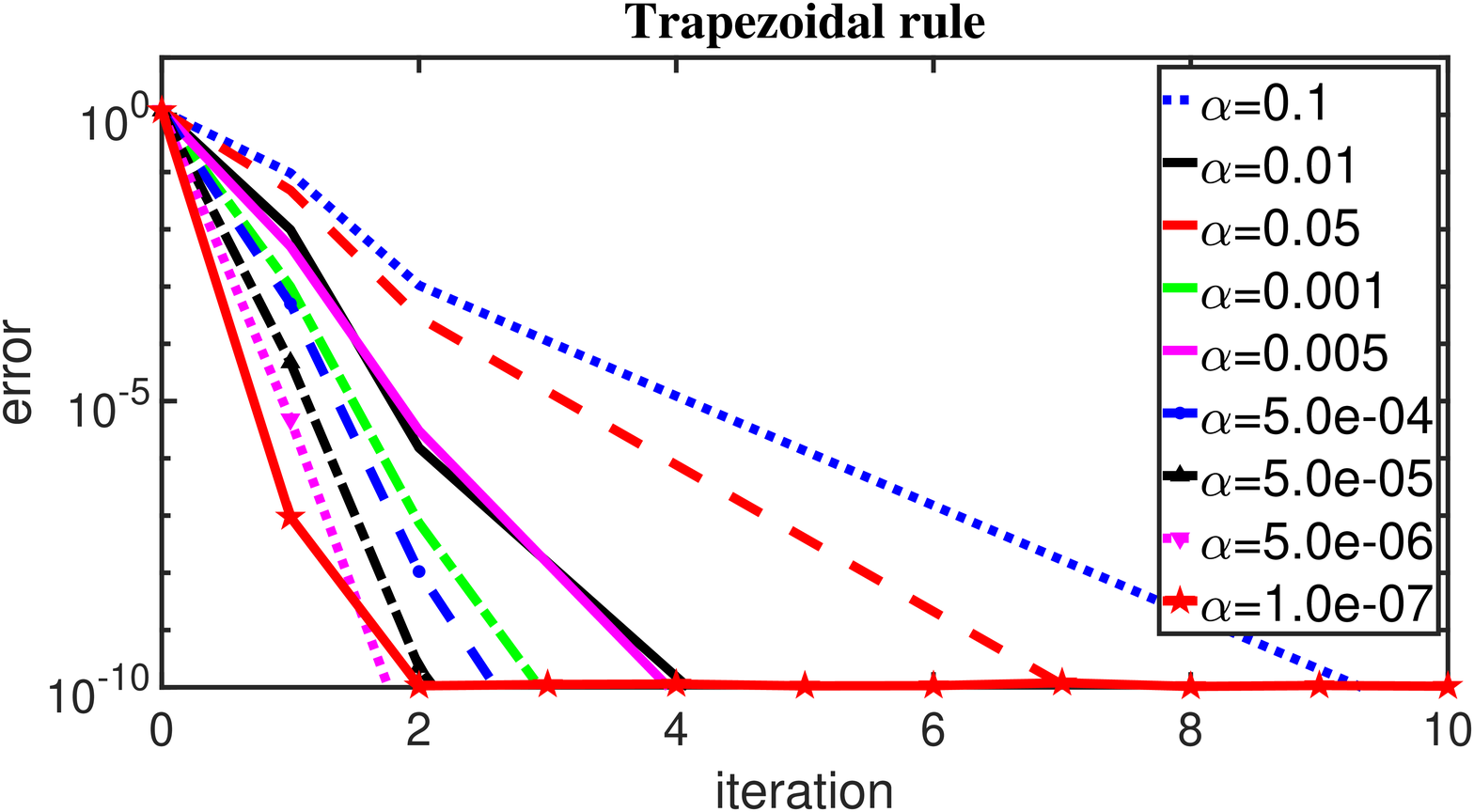} }}
    \caption{Convergence for various $\alpha$ with fixed parameters $h=1/128, \Delta t=0.0001, \beta=0.2, \epsilon^2=0.01, T=1$. On the left: Backward-Euler; On the right: Trapezoidal rule.}
    \label{linch_fig2}
\end{figure}
\begin{figure}[h!]
    \centering
     \subfloat{{\includegraphics[height=4cm,width=4cm]{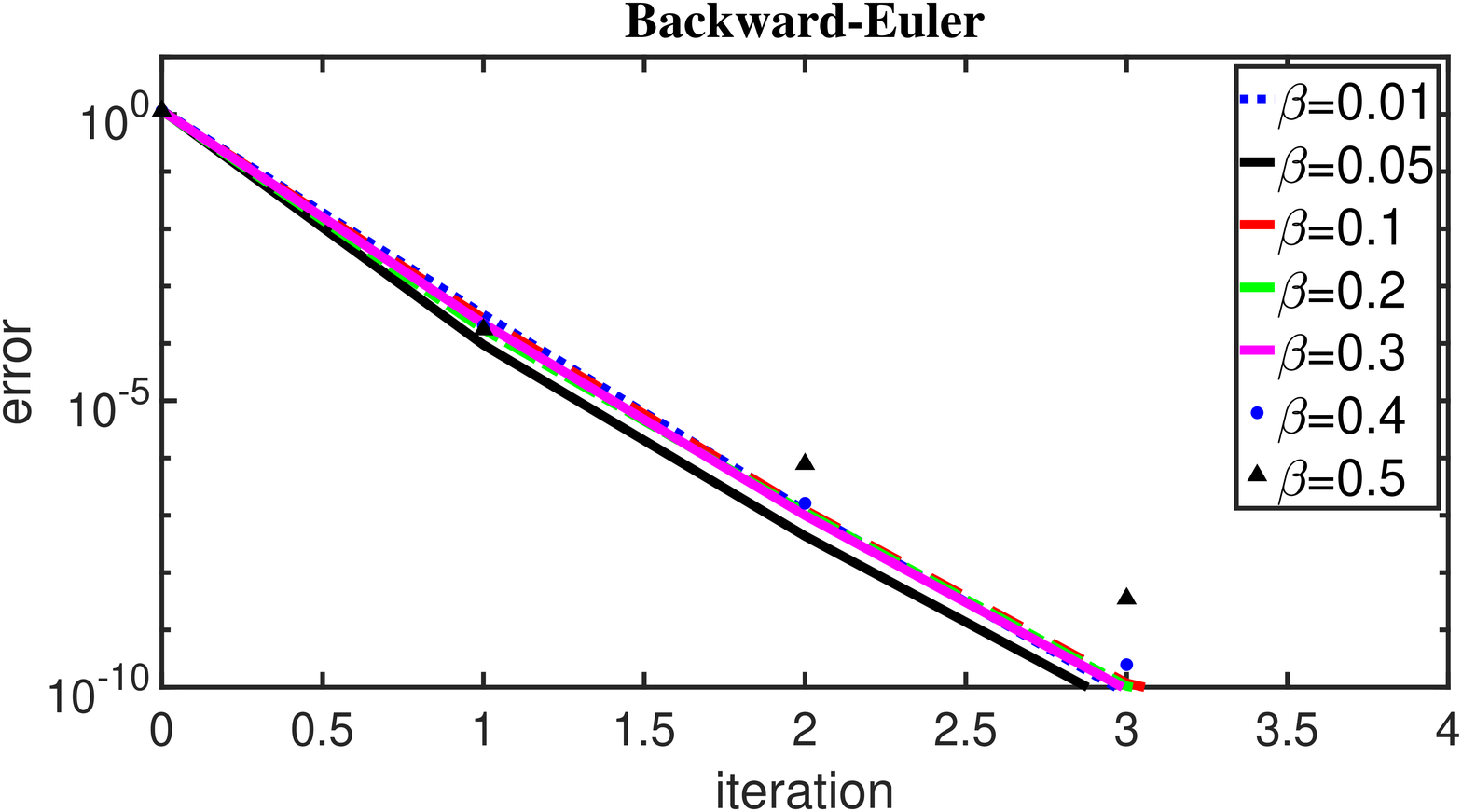} }}
     \subfloat{{\includegraphics[height=4cm,width=4cm]{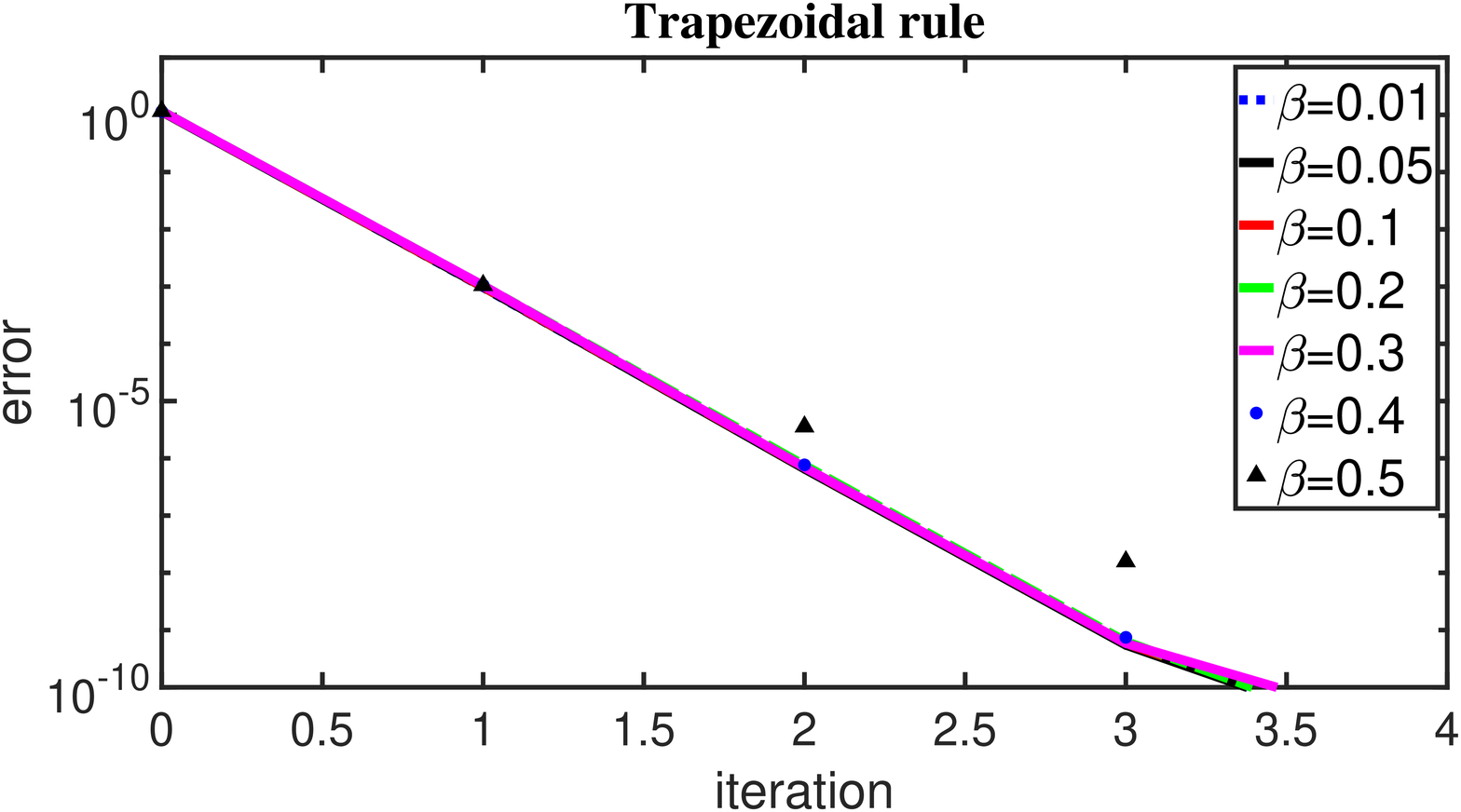} }}
     \subfloat{{\includegraphics[height=4cm,width=4cm]{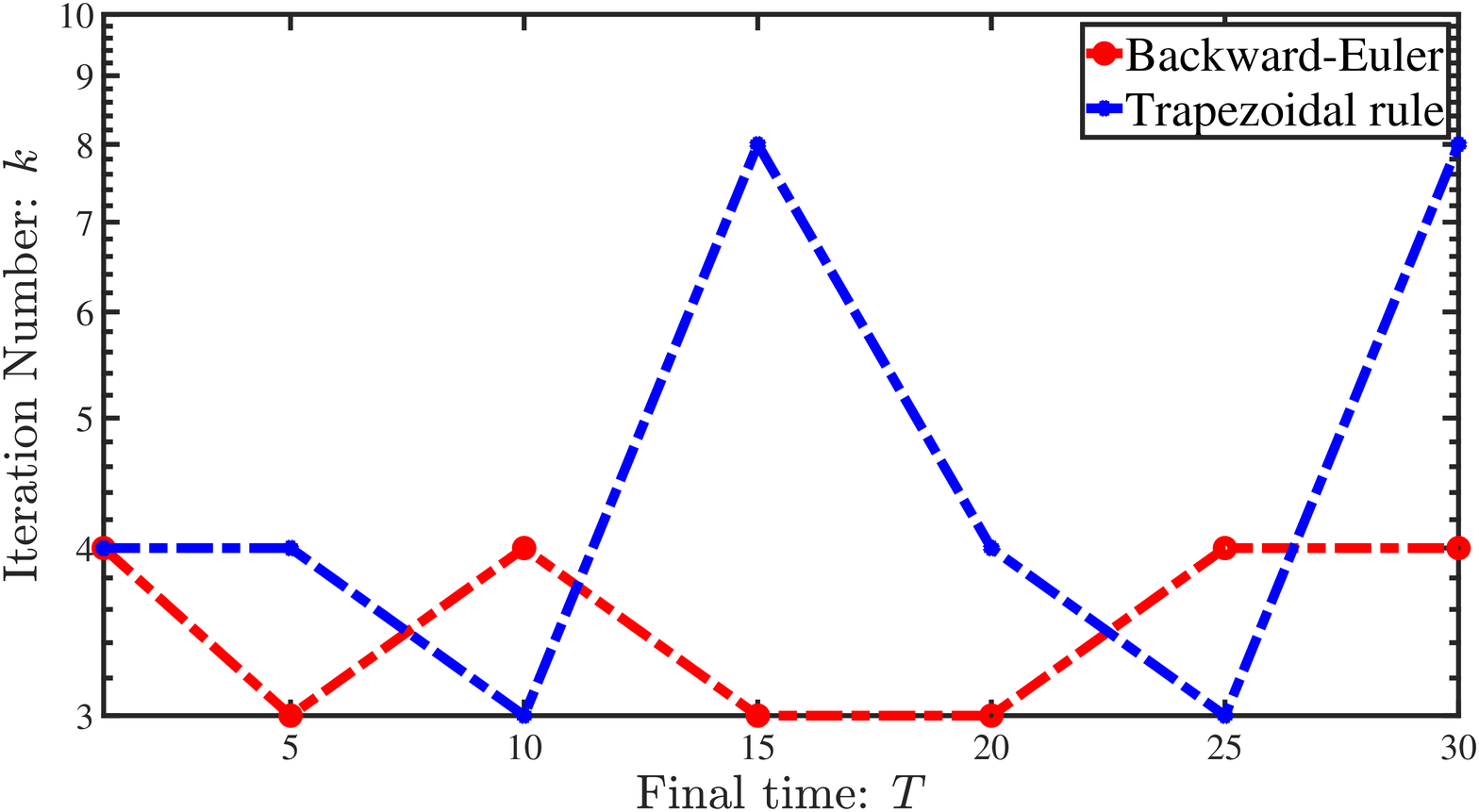} }}
    \caption{Convergence for various $\beta$ and $T$ with fixed $h=1/128, \Delta t=0.0001, \epsilon^2=0.01$. On the left and middle: for different $\beta$; On the right: comparison for different $T$ with $\beta=0.2$.}
    \label{linch_fig3}
\end{figure}

For the Linearized CH equation \eqref{modelproblem2a} in 2D we take the spatial domain $\Omega=(0, \pi)\times(0, \pi)$, having an equidistant mesh size $h$ on both the direction. Figure \ref{linch_fig33} shows the numerical error and the theoretical discrete error bound given in Theorem \ref{thm6} by fixing $T=1, \Delta t=0.0001$ and the PinT parameter $\alpha=0.001$. In Figure \ref{linch_fig4} we see the dependency on $\epsilon$ and $N_t$ with fixed parameters $\beta=0.2, \alpha=0.001, h=1/10, T=1$. As in the case of 1D, we observe similar convergence behaviour irrespective of problem parameters as well as mesh parameters in 2D. We skip the numerical experiments for the general fourth order PDE given in \eqref{general_4thorder} as we observe the similar convergence behaviour like the two already showed linear models. Next we show the convergence behaviour of the PinT algorithm applied to the nonlinear CH equation.
\begin{figure}[h!]
    \centering
    \subfloat{{\includegraphics[height=4cm,width=3.5cm]{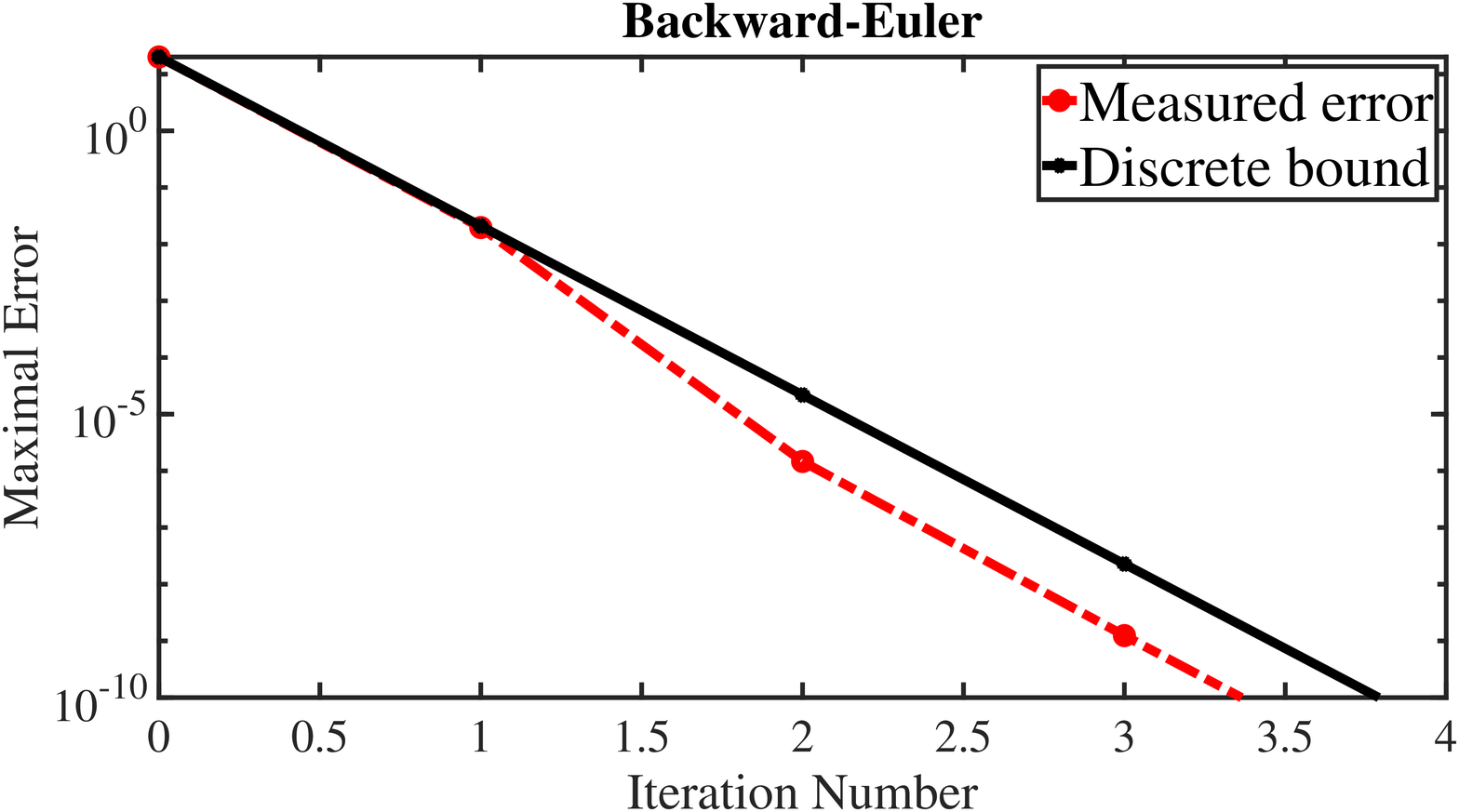} }}
     \subfloat{{\includegraphics[height=4cm,width=3.5cm]{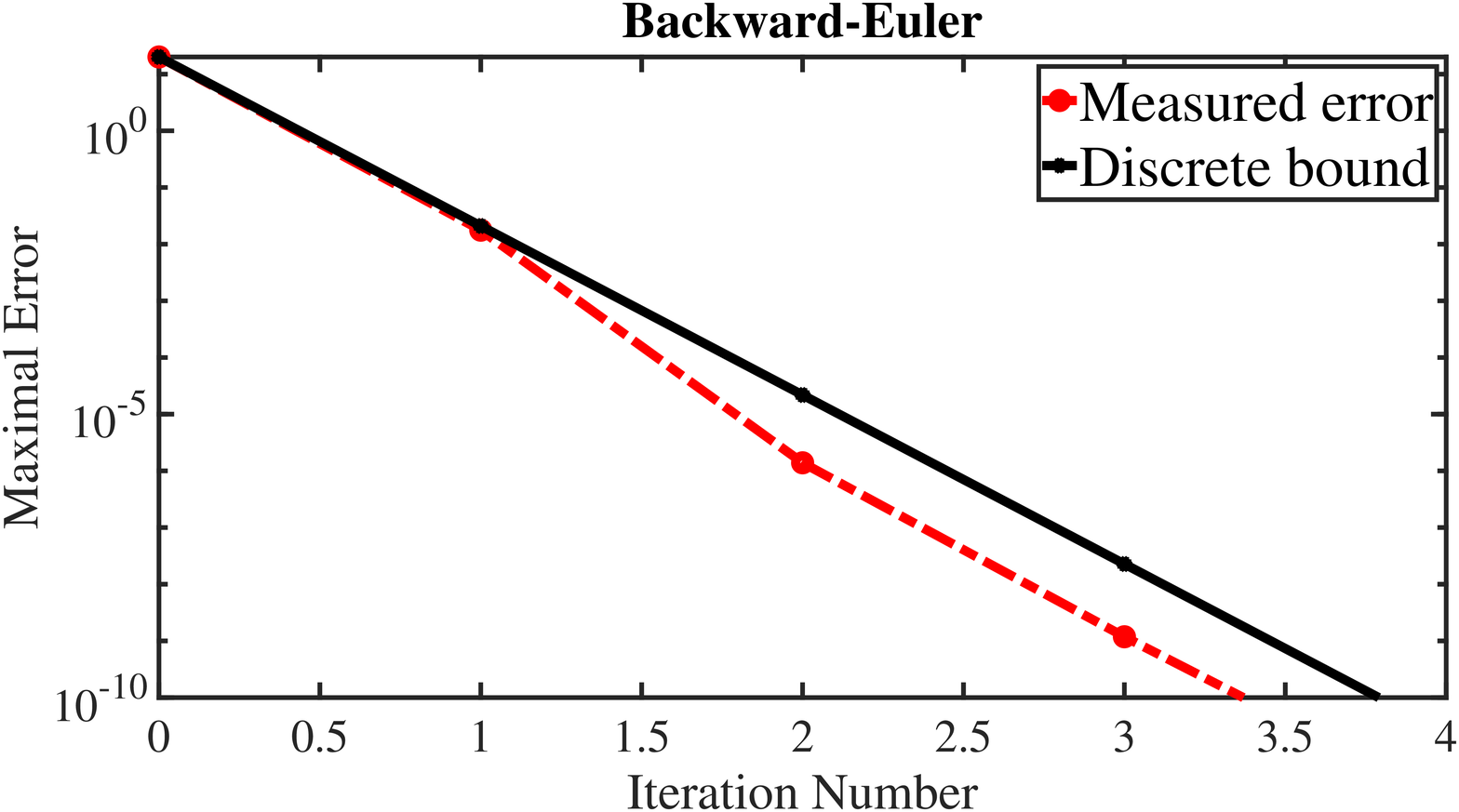} }}
     \subfloat{{\includegraphics[height=4cm,width=3.5cm]{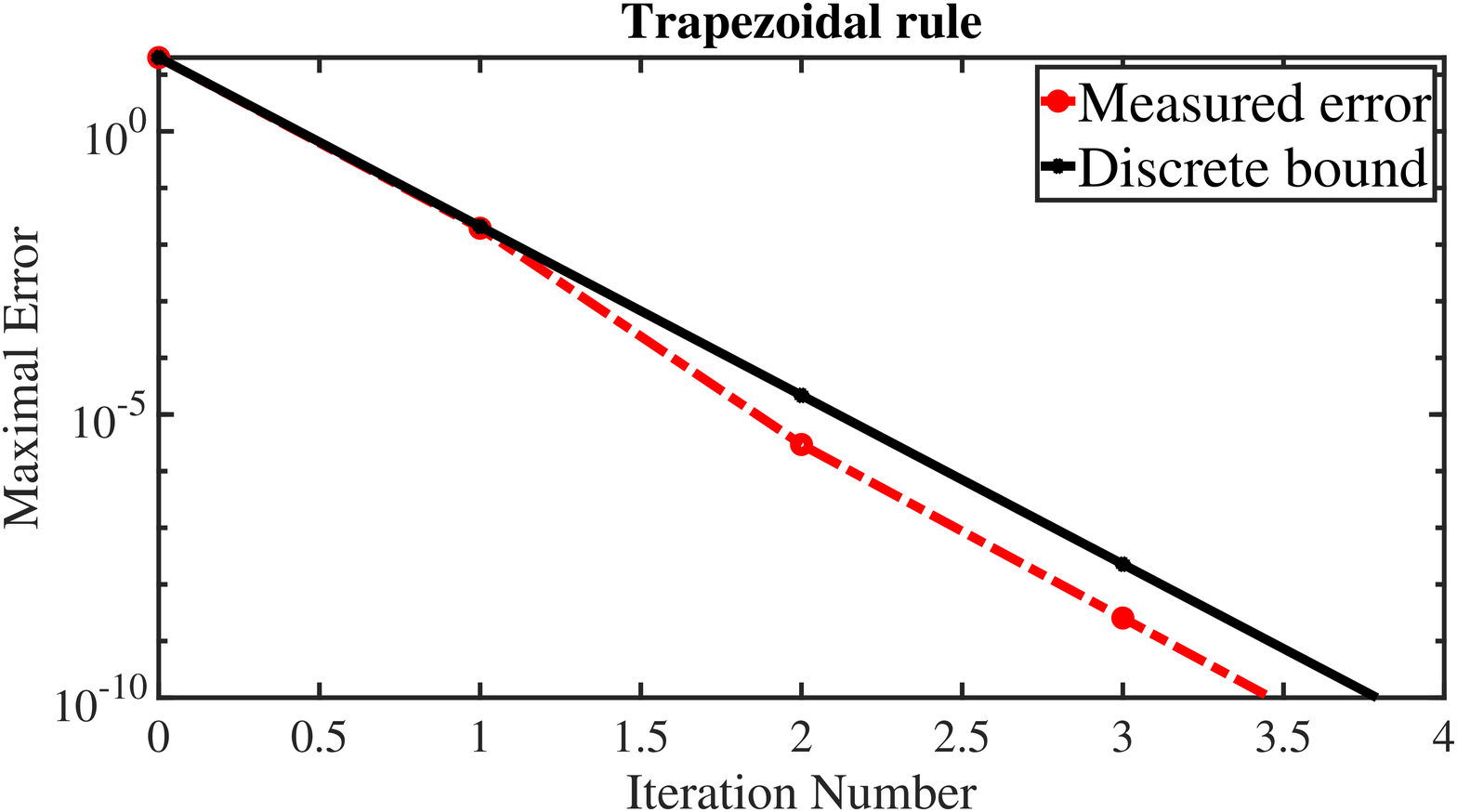} }}
     \subfloat{{\includegraphics[height=4cm,width=3.5cm]{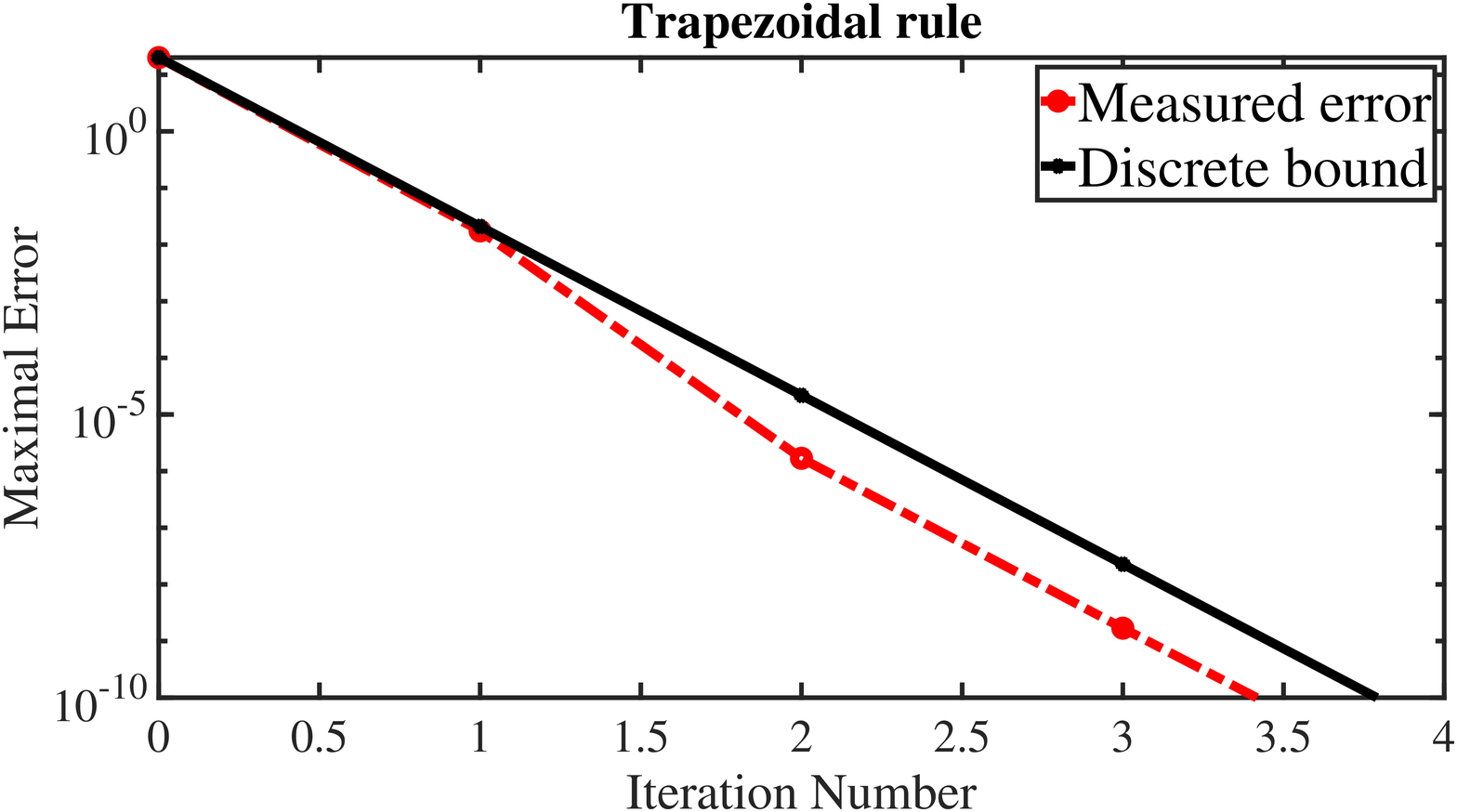} }}
    \caption{Comparison of theoretical estimate and numerical error with different mesh size $h$. First (with $h=1/10$) and second (with $h=1/20$) figure: Backward-Euler; Third (with $h=1/10$) and fourth (with $h=1/20$) figure: Trapezoidal rule.}
    \label{linch_fig33}
\end{figure}
\begin{figure}[h!]
    \centering
    \subfloat{{\includegraphics[height=4cm,width=3.5cm]{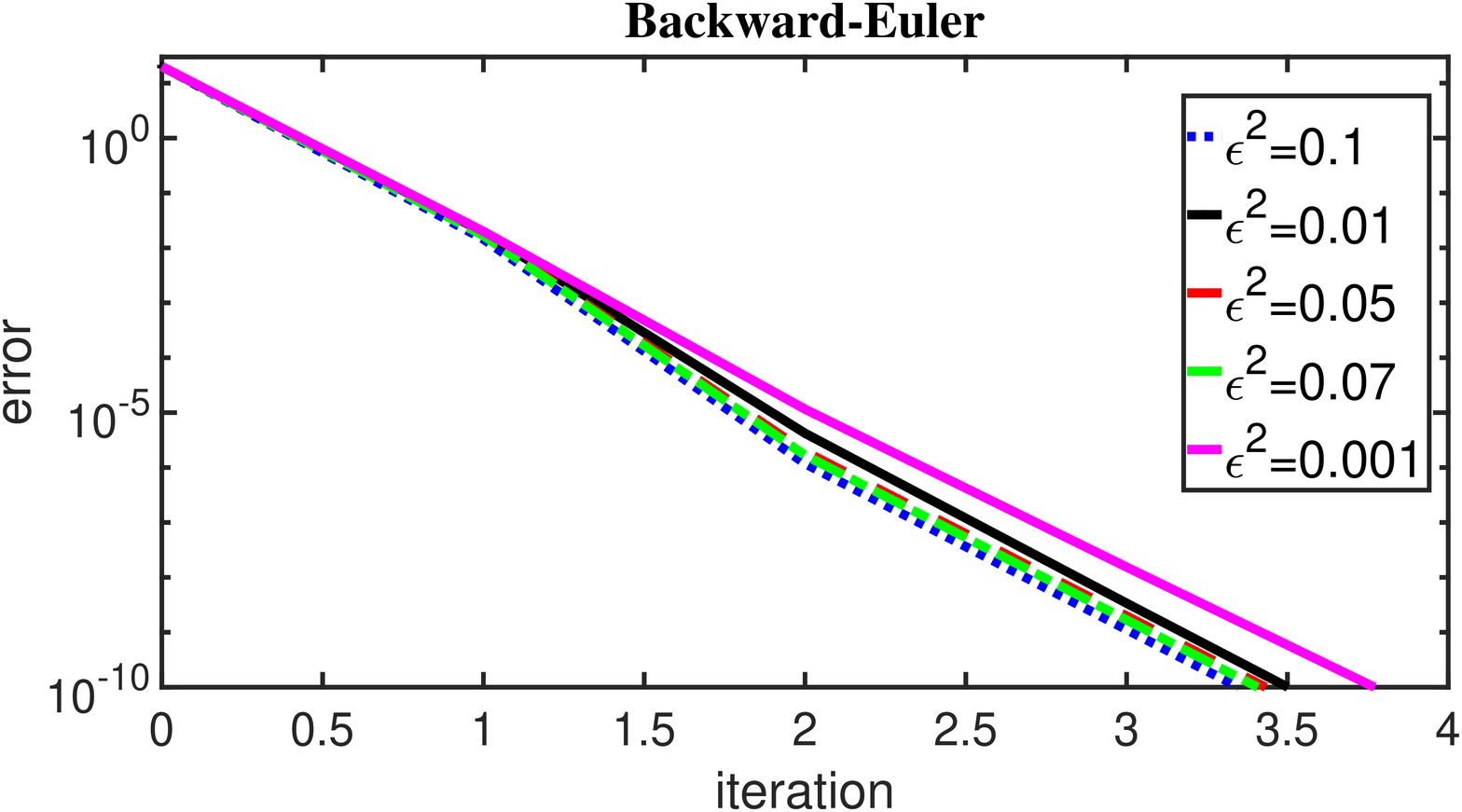} }}
     \subfloat{{\includegraphics[height=4cm,width=3.5cm]{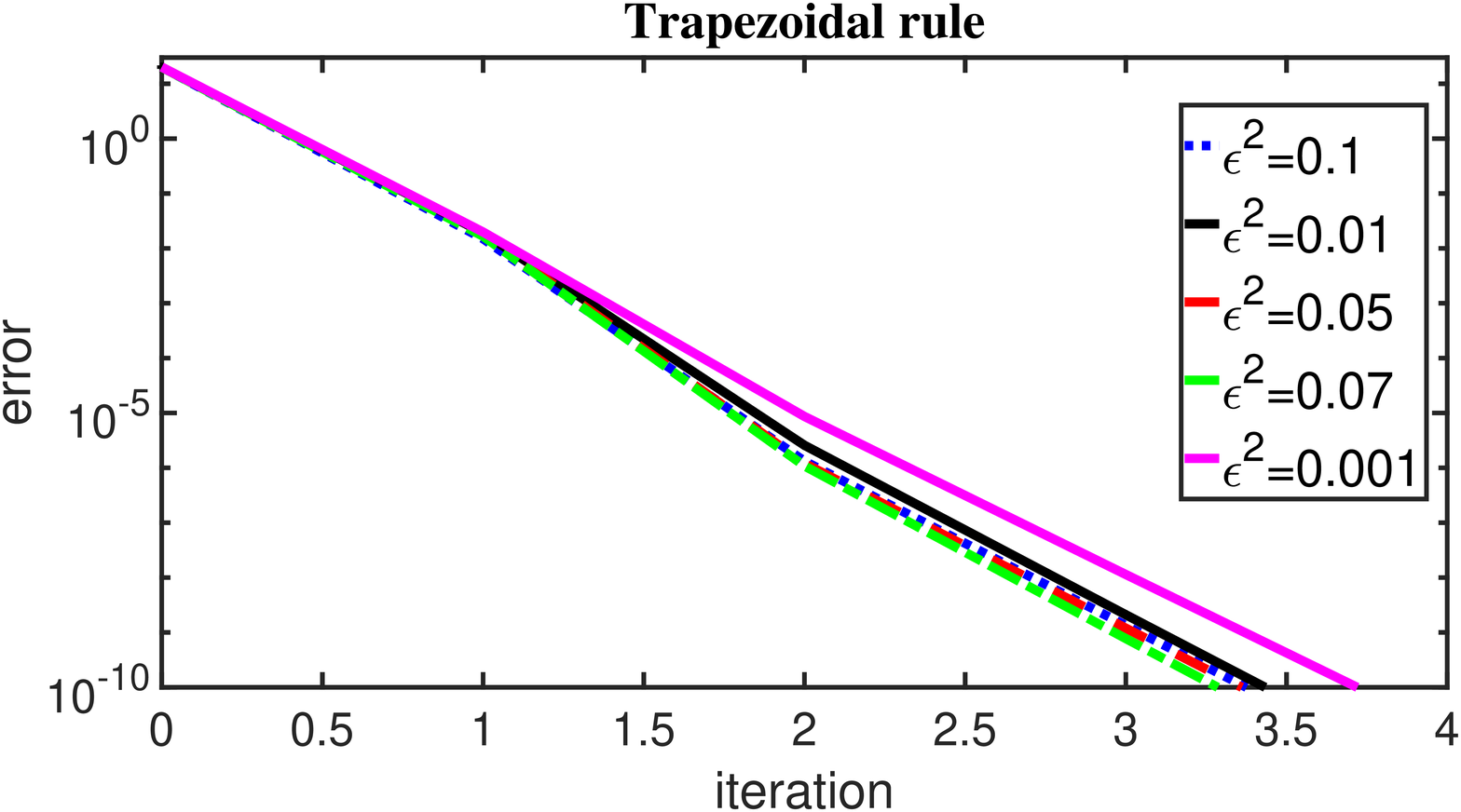} }}
     \subfloat{{\includegraphics[height=4cm,width=3.5cm]{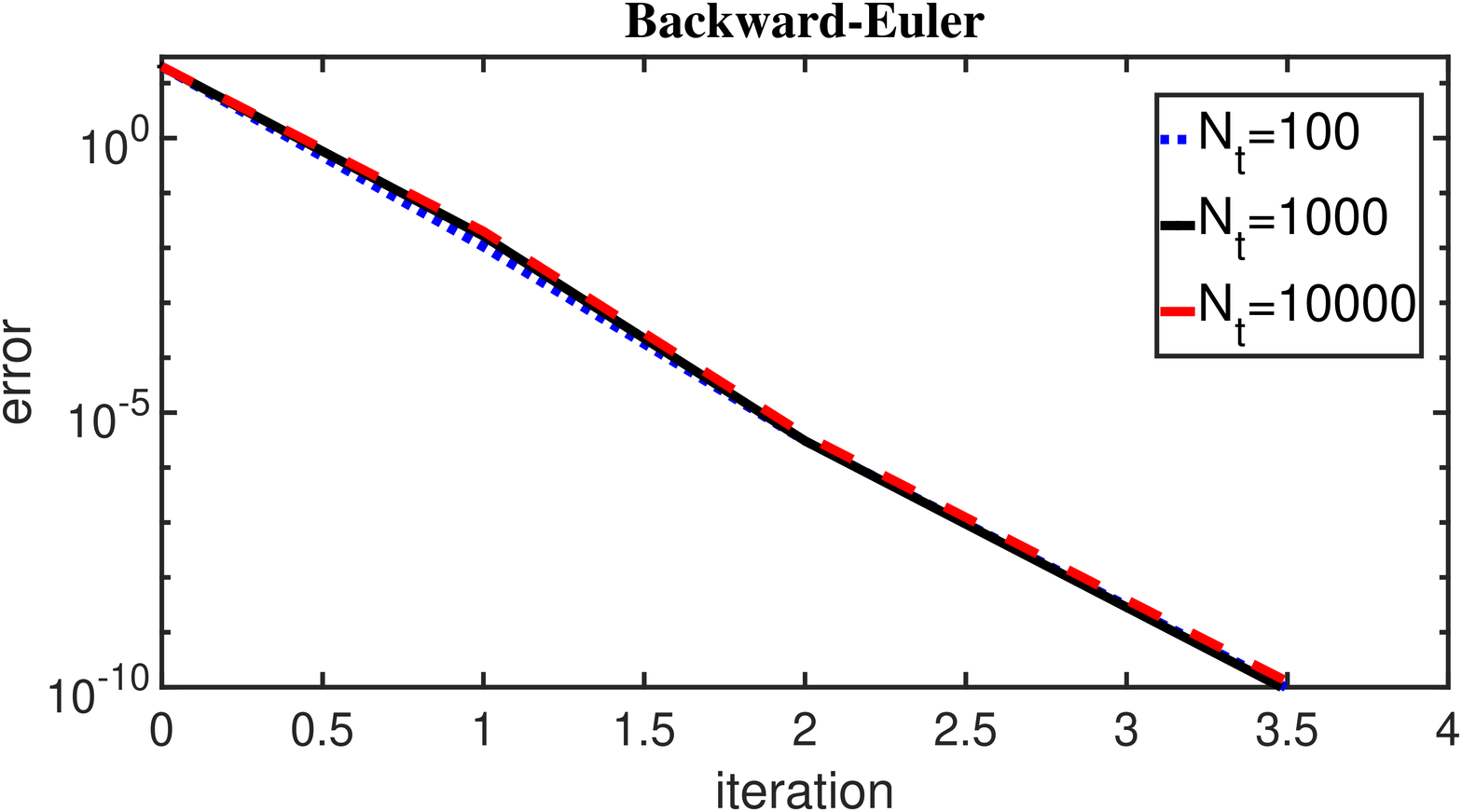} }}
     \subfloat{{\includegraphics[height=4cm,width=3.5cm]{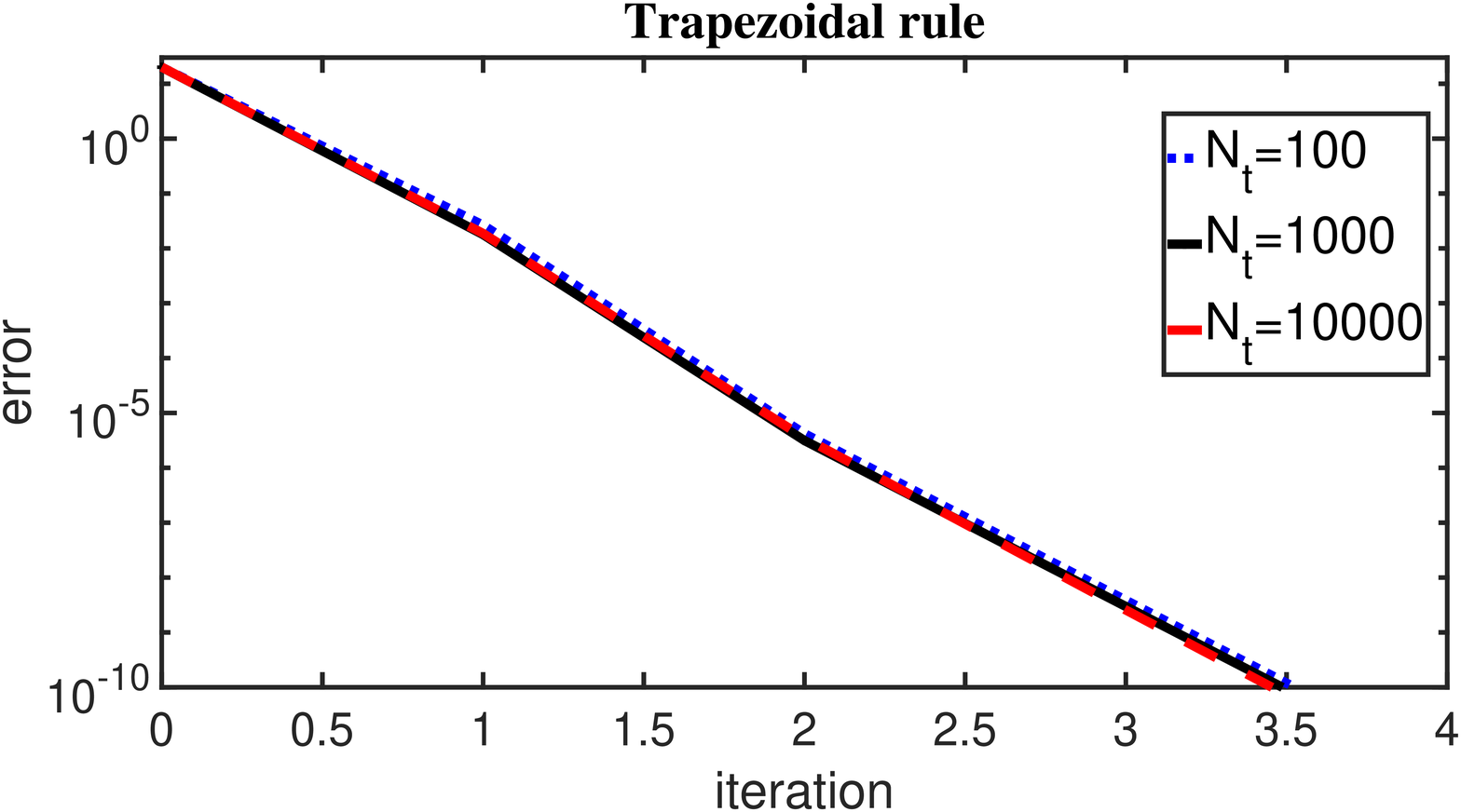} }}
    \caption{Convergence for various $\epsilon, N_t$ and different time integrators. First  and second figure: $\epsilon$ dependency with $\Delta t=10^{-4}$; Third and fourth figure: $N_t$ dependency with $\epsilon^2 = 0.01$.}
    \label{linch_fig4}
\end{figure}
\subsection{The nonlinear CH equation}
First we present the numerical results in 1D by considering the domain $\Omega=(0, 1)$. We plot the convergence curve in Figure \ref{nonlinch_fig1} of PinT-I and PinT-II to see the dependency on PinT parameter $\alpha$, and one can see that very small $\alpha$ does not provide good convergence. In Figure \ref{nonlinch_fig2}
we observe the mesh independence and problem parameter $\epsilon$ independence of the proposed methods.
\begin{figure}[h!]
    \centering
    \subfloat{{\includegraphics[height=5cm,width=6.5cm]{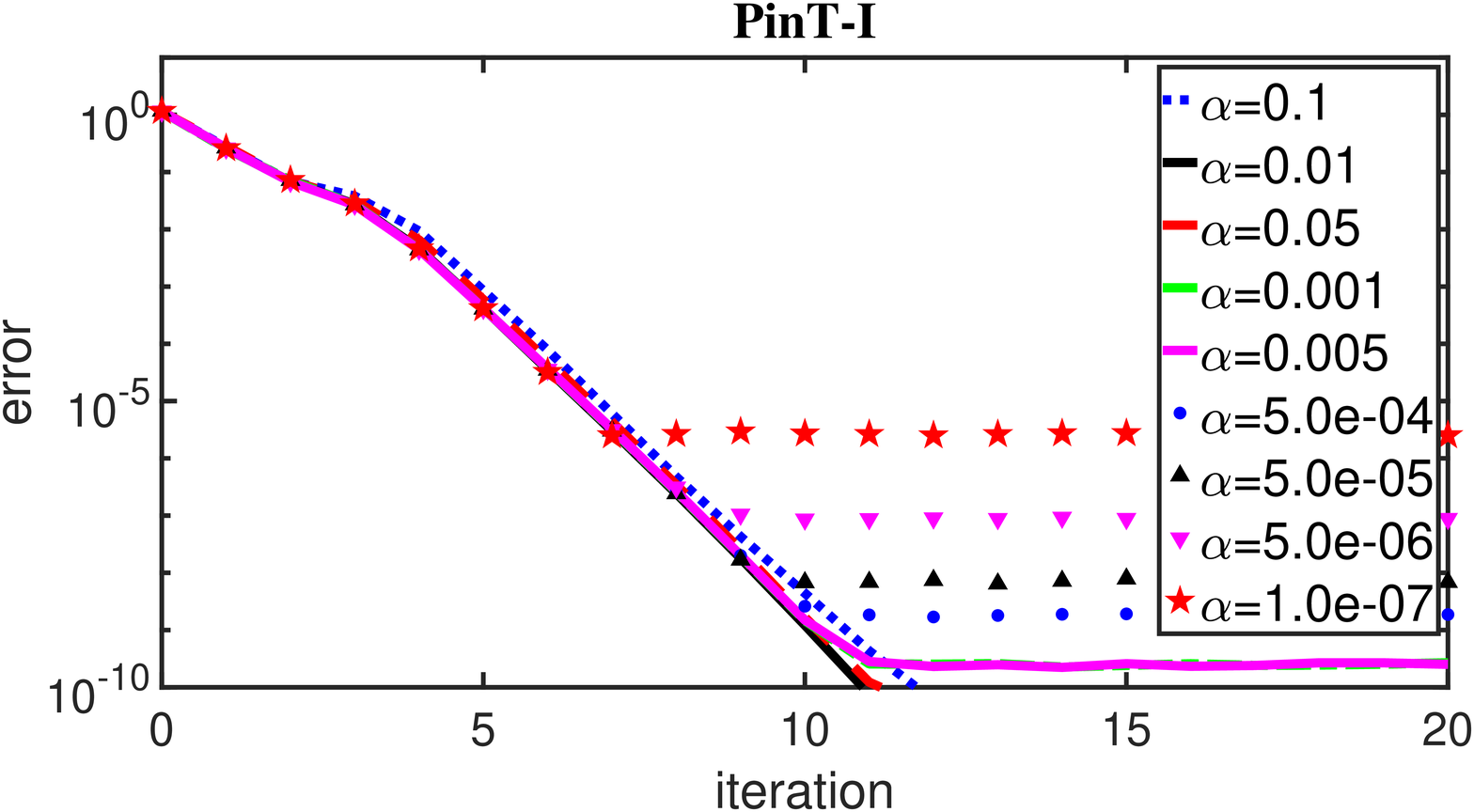} }}
     \subfloat{{\includegraphics[height=5cm,width=6.5cm]{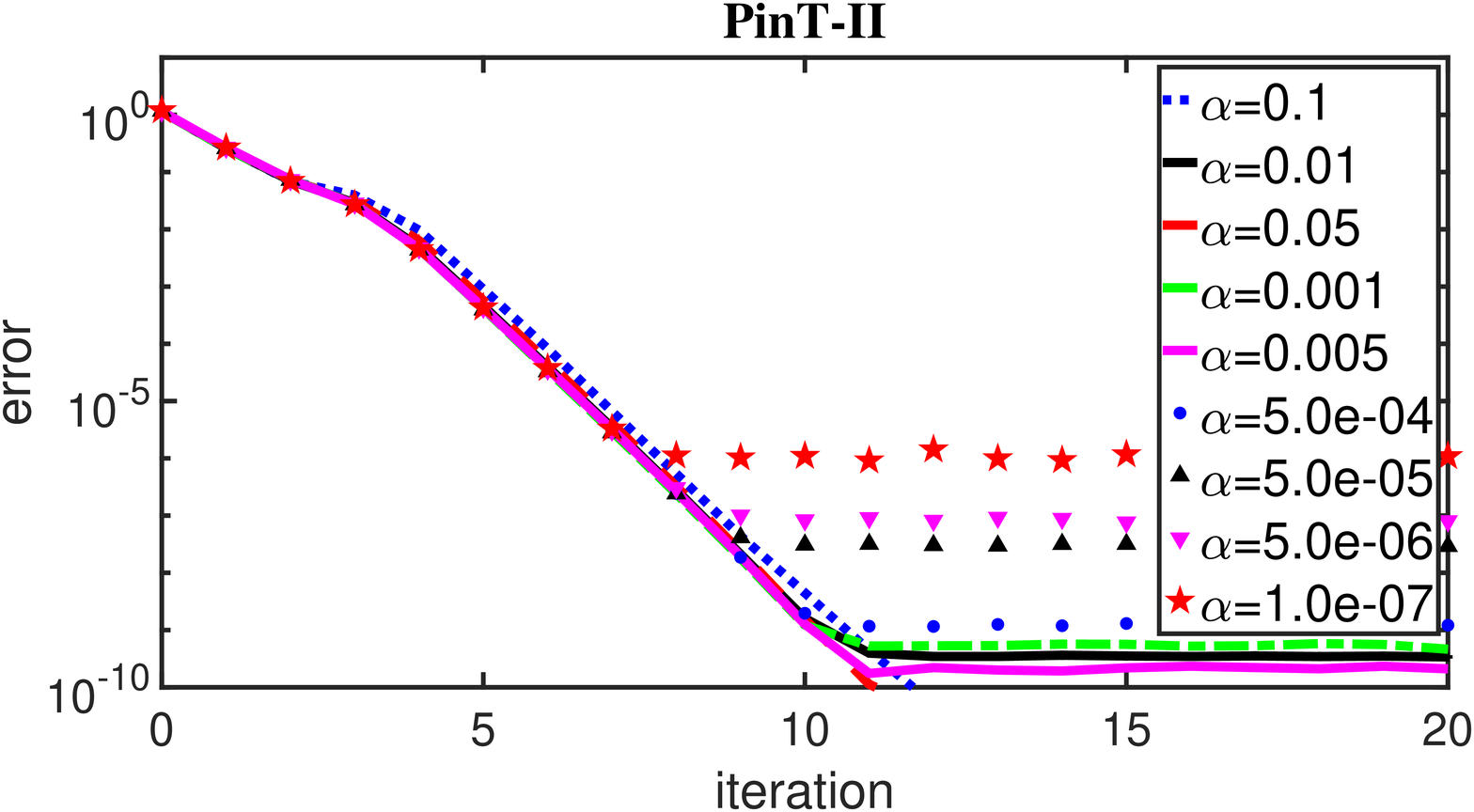} }}
    \caption{Convergence for different $\alpha$ with fixed $h=1/128, \Delta t=10^{-4}, T=0.1, \epsilon^2=0.05$. On the left: PinT-I; On the right: PinT-II.}
    \label{nonlinch_fig1}
\end{figure}
\begin{figure}[h!]
    \centering
    \subfloat{{\includegraphics[height=4cm,width=3.5cm]{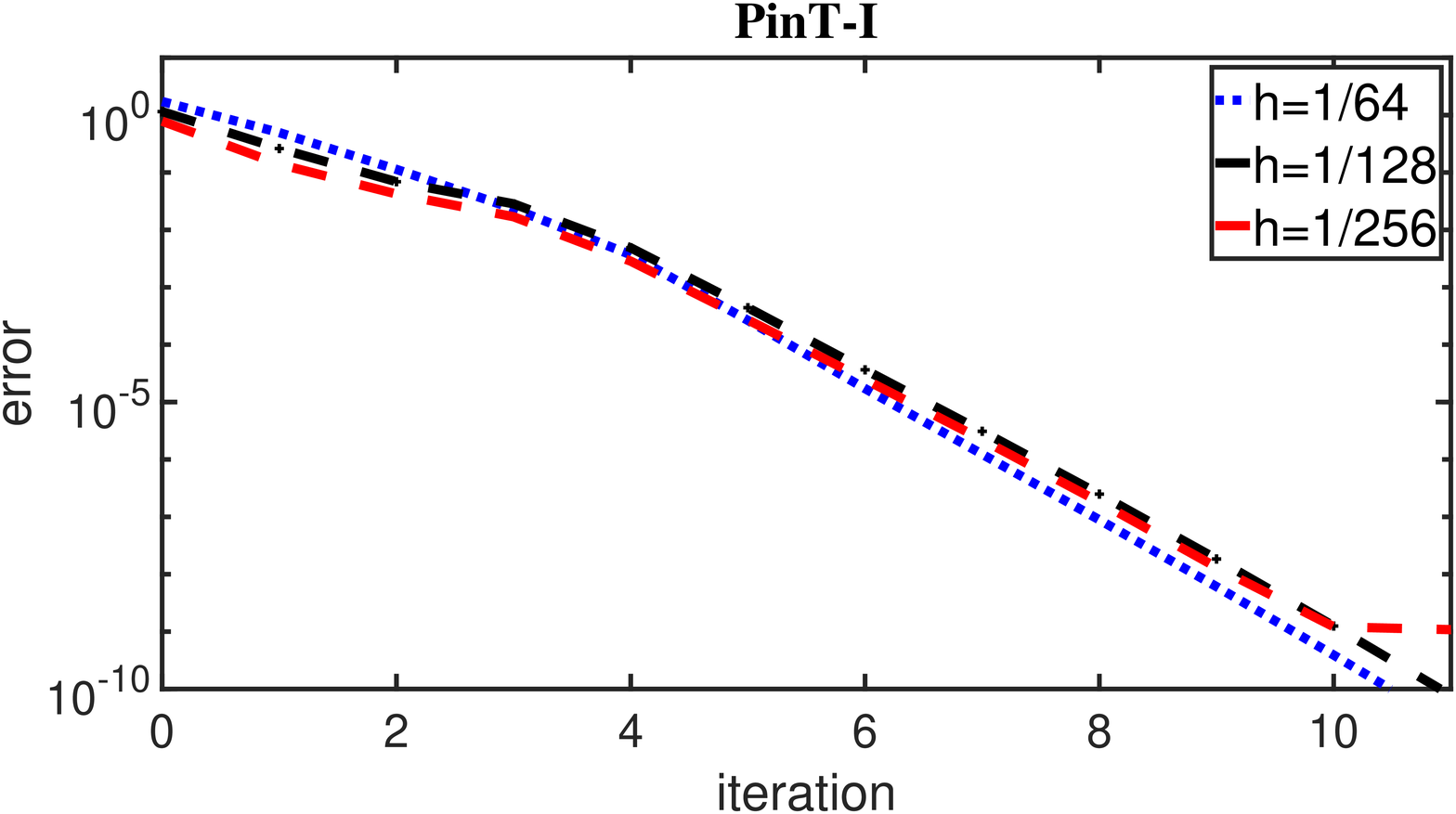} }}
     \subfloat{{\includegraphics[height=4cm,width=3.5cm]{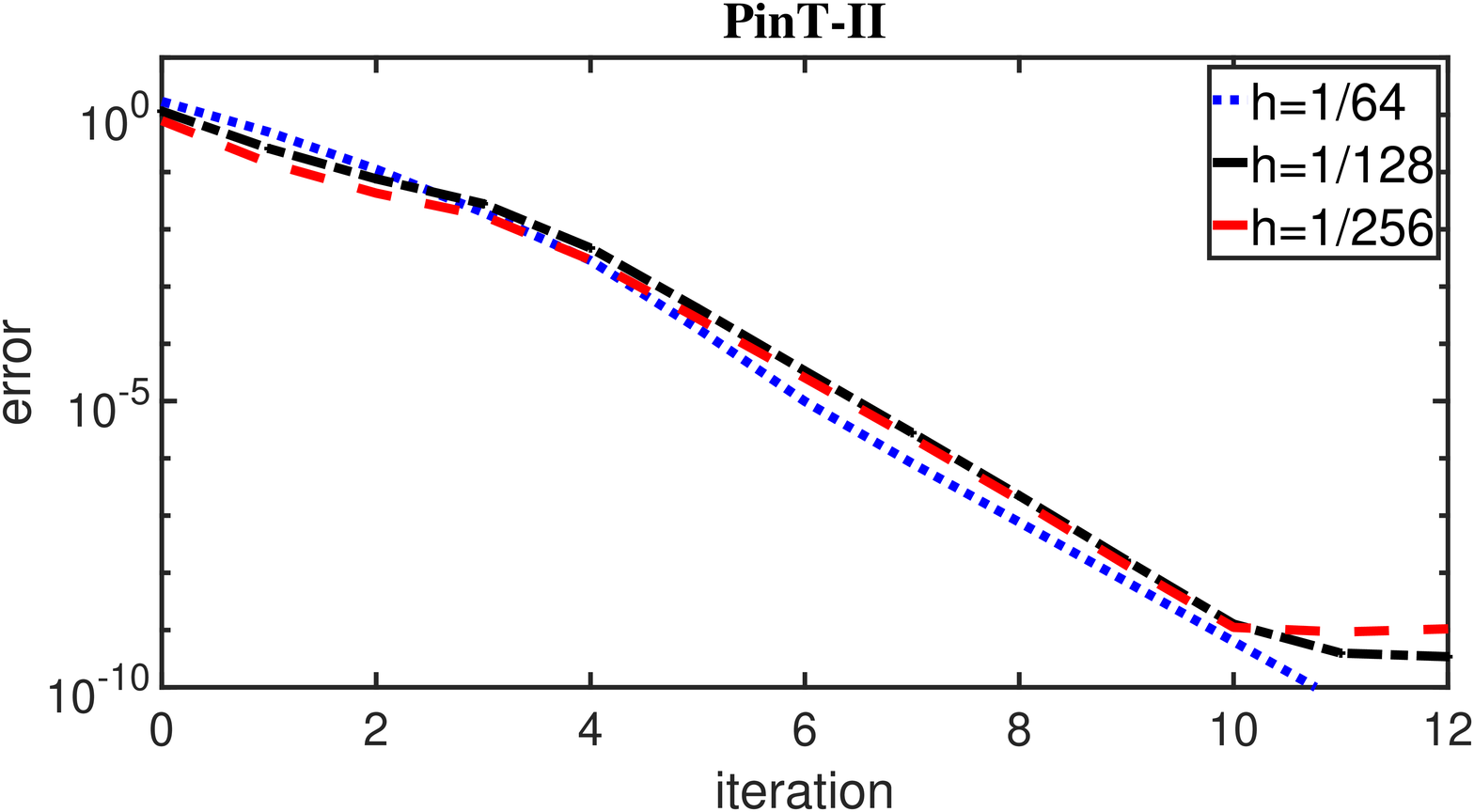} }}
     \subfloat{{\includegraphics[height=4cm,width=3.5cm]{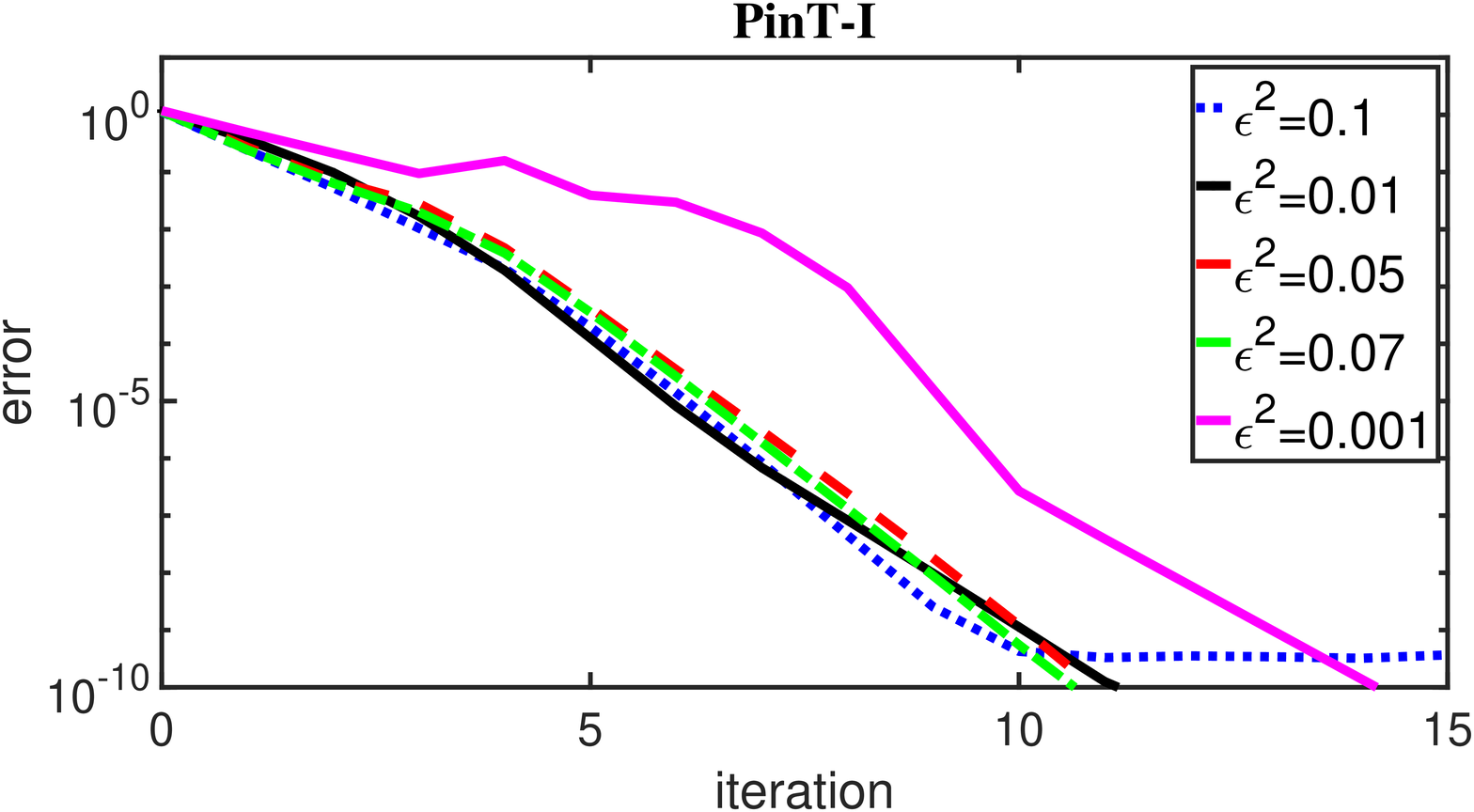} }}
     \subfloat{{\includegraphics[height=4cm,width=3.5cm]{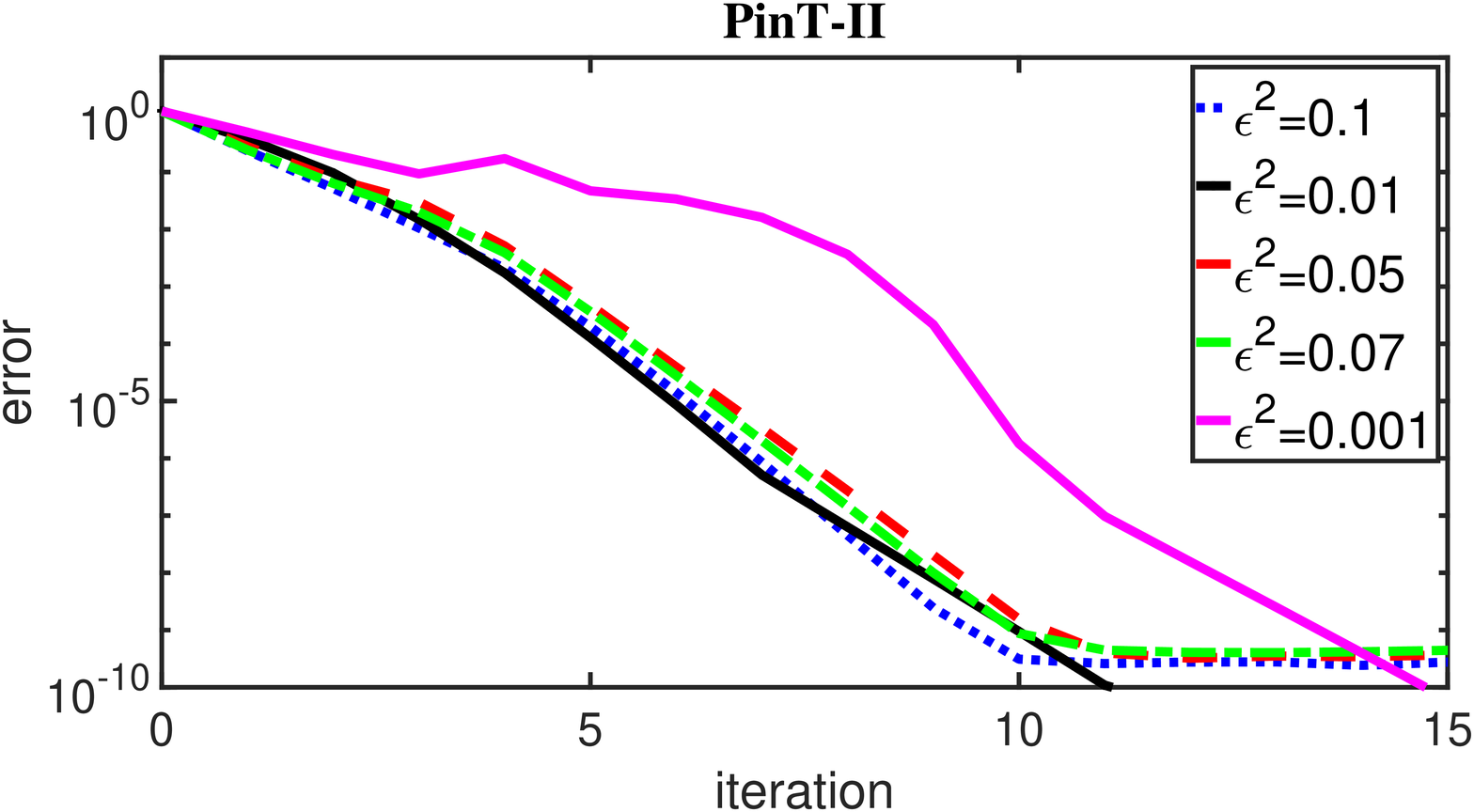} }}
    \caption{Convergence for various $h, \epsilon$ with different integrator. First and second figure: $h$ dependency with $\Delta t=10^{-4}, T=0.1, \epsilon^2=0.05,\alpha=0.01$; Third and fourth figure: $\epsilon$ dependency with $\Delta t=10^{-4}, T=0.1, h=1/128,\alpha=0.01$.}
    \label{nonlinch_fig2}
\end{figure}
\begin{figure}[h!]
    \centering
    \subfloat{{\includegraphics[height=4cm,width=3.5cm]{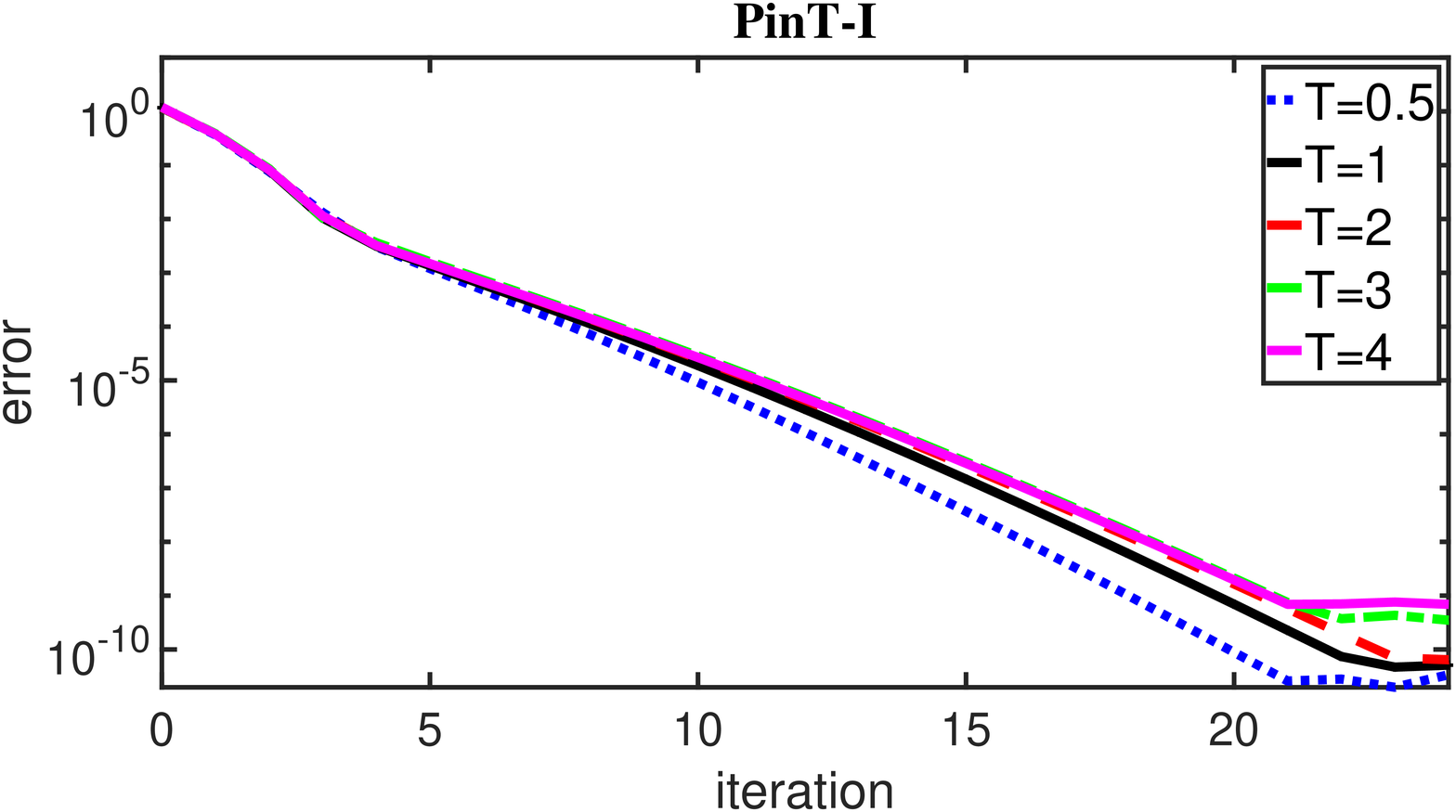} }}
     \subfloat{{\includegraphics[height=4cm,width=3.5cm]{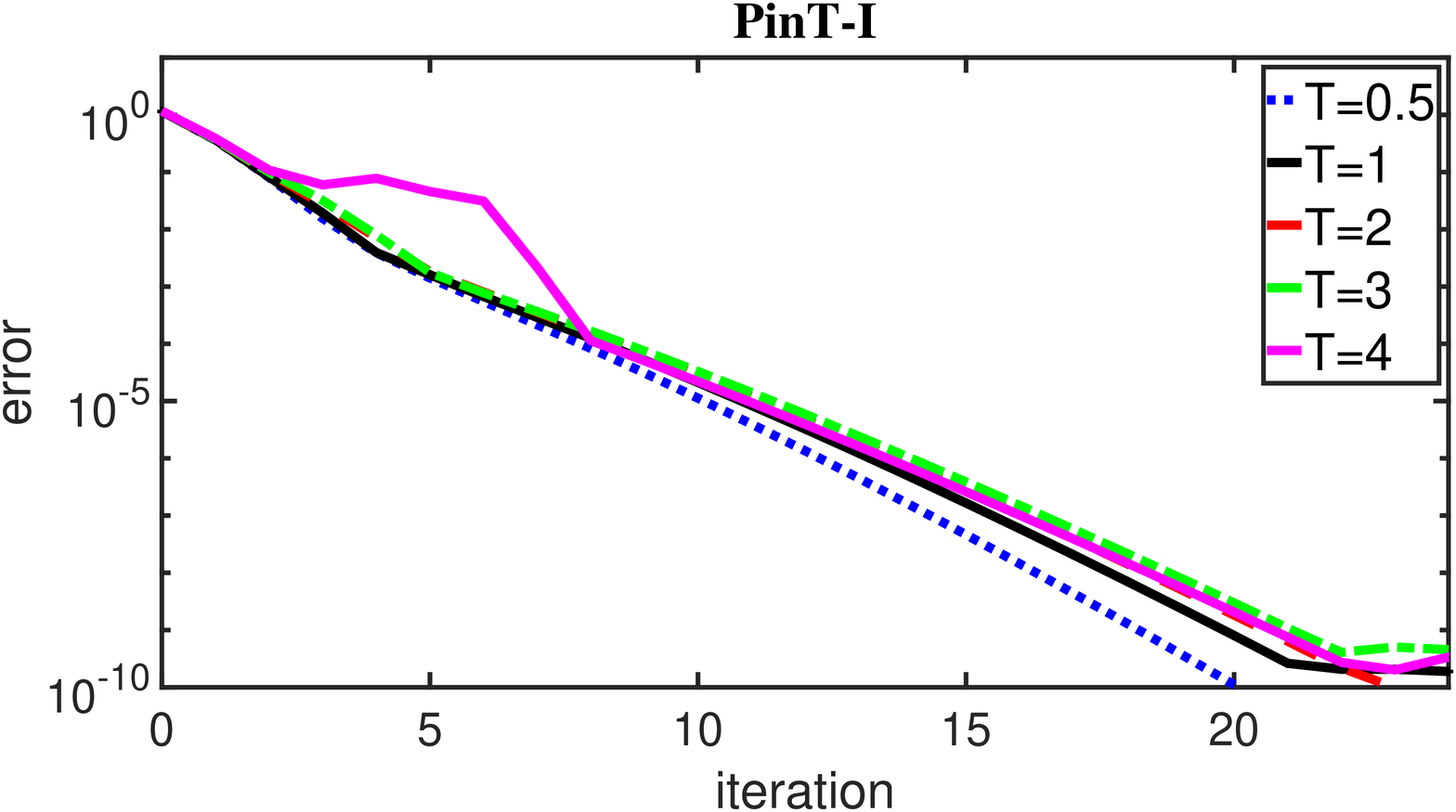} }}
     \subfloat{{\includegraphics[height=4cm,width=3.5cm]{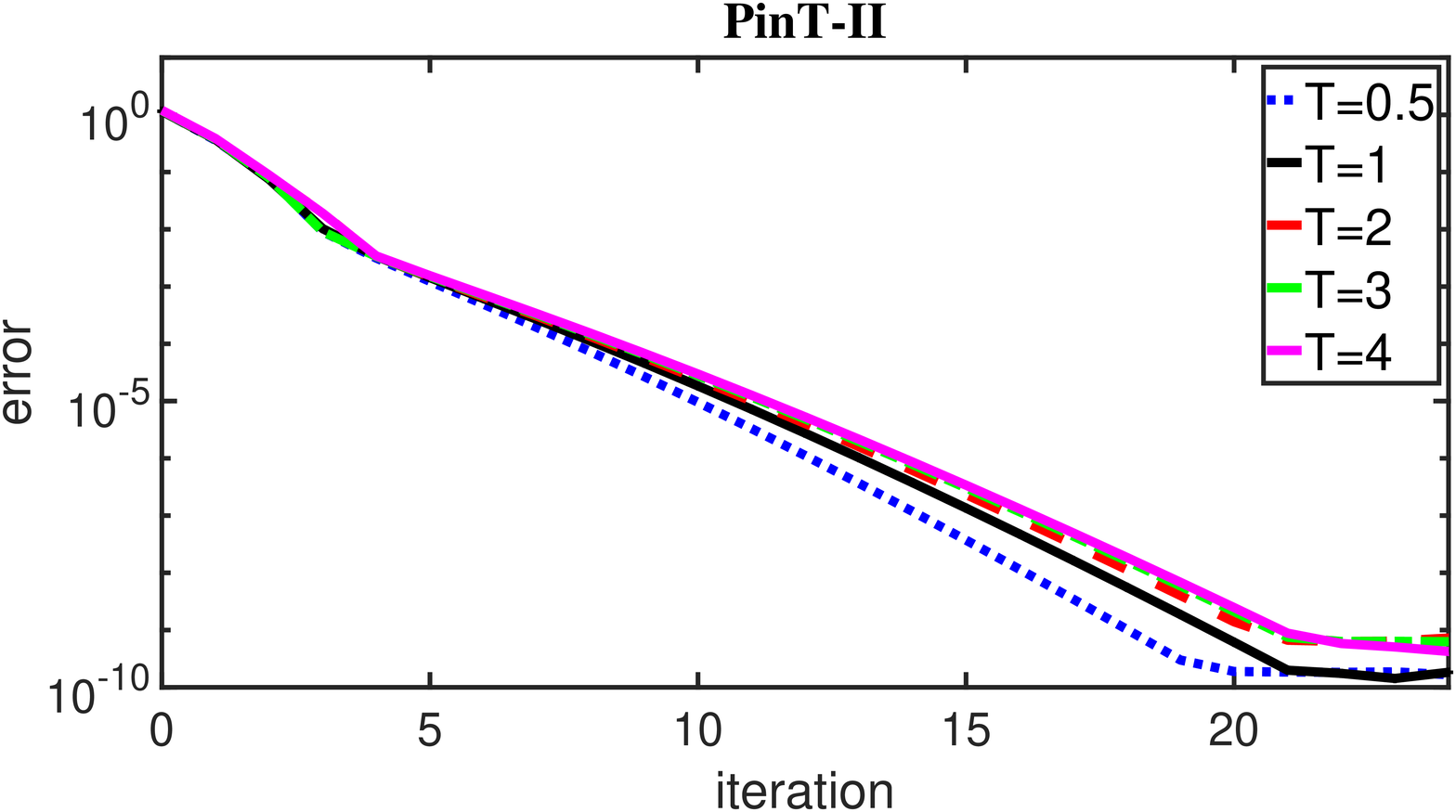} }}
     \subfloat{{\includegraphics[height=4cm,width=3.5cm]{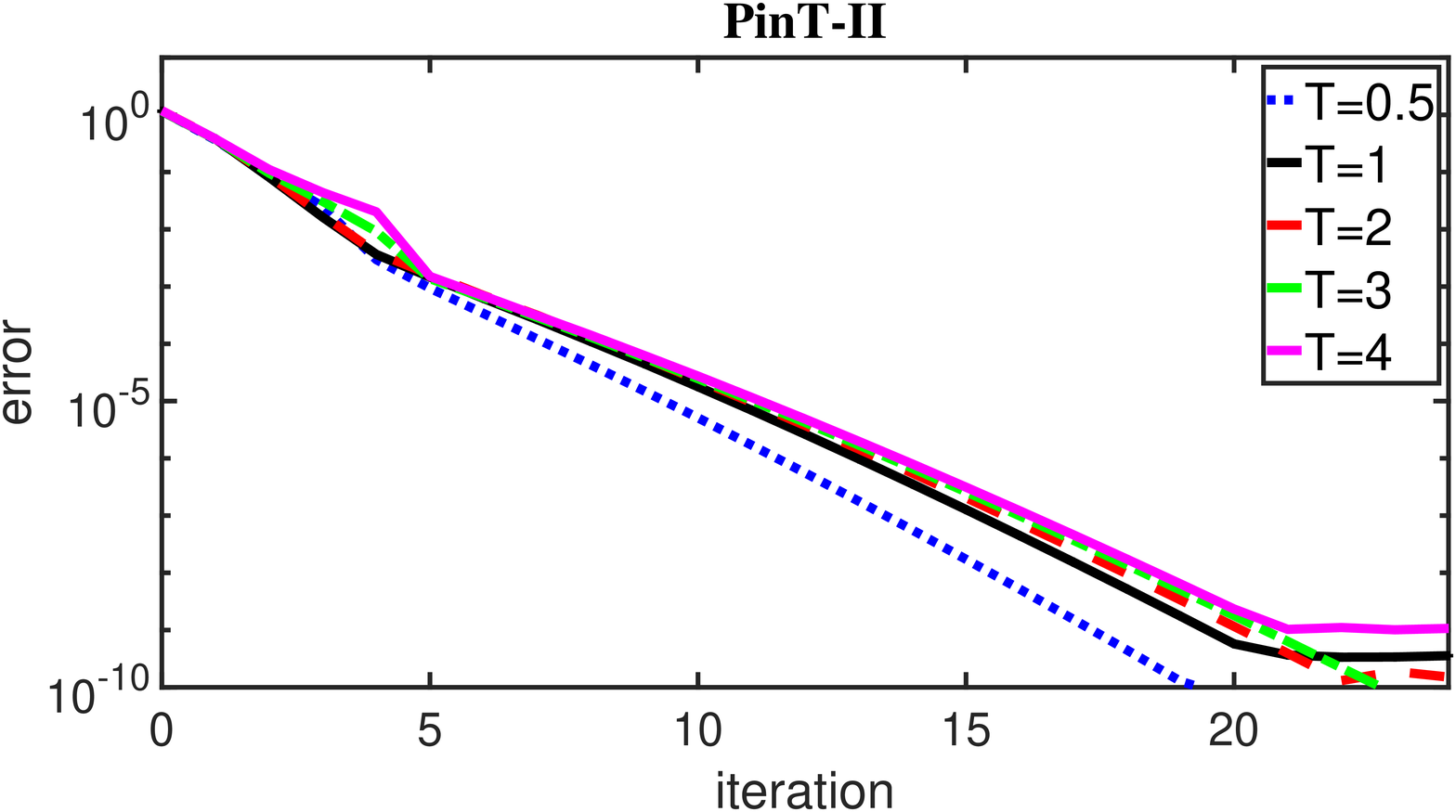} }}
    \caption{Convergence for different $T$ and different integrator with fixed $h=1/128, \epsilon^2=0.01$. First  and third figure: $\Delta t=10^{-4}, \alpha=0.05$; Second and fourth figure: $\Delta t=10^{-3}, \alpha=0.15$.}
    \label{nonlinch_fig3}
    \end{figure}

Both methods exhibit similar kind of convergence behaviour in terms of iteration count, and can be seen in Figure \ref{nonlinch_fig3}. Even though we get robust convergence for large $T$, the privileges of getting solution in a very long time window is not straight forward, because over a long time the Jacobian approximation may not be quite accurate. For example with $T=4, \Delta t=10^{-4}$ we have $N_t=40000$; so that one computes those many time steps in parallel per iteration and taking the Jacobian approximation by averaging in each iteration. As $N_t$ grows the approximation to Jacobian becomes poorer. One idea of getting solution over long time is to use windowing technique.
\begin{figure}[h!]
    \centering
    \subfloat{{\includegraphics[height=6cm,width=6.5cm]{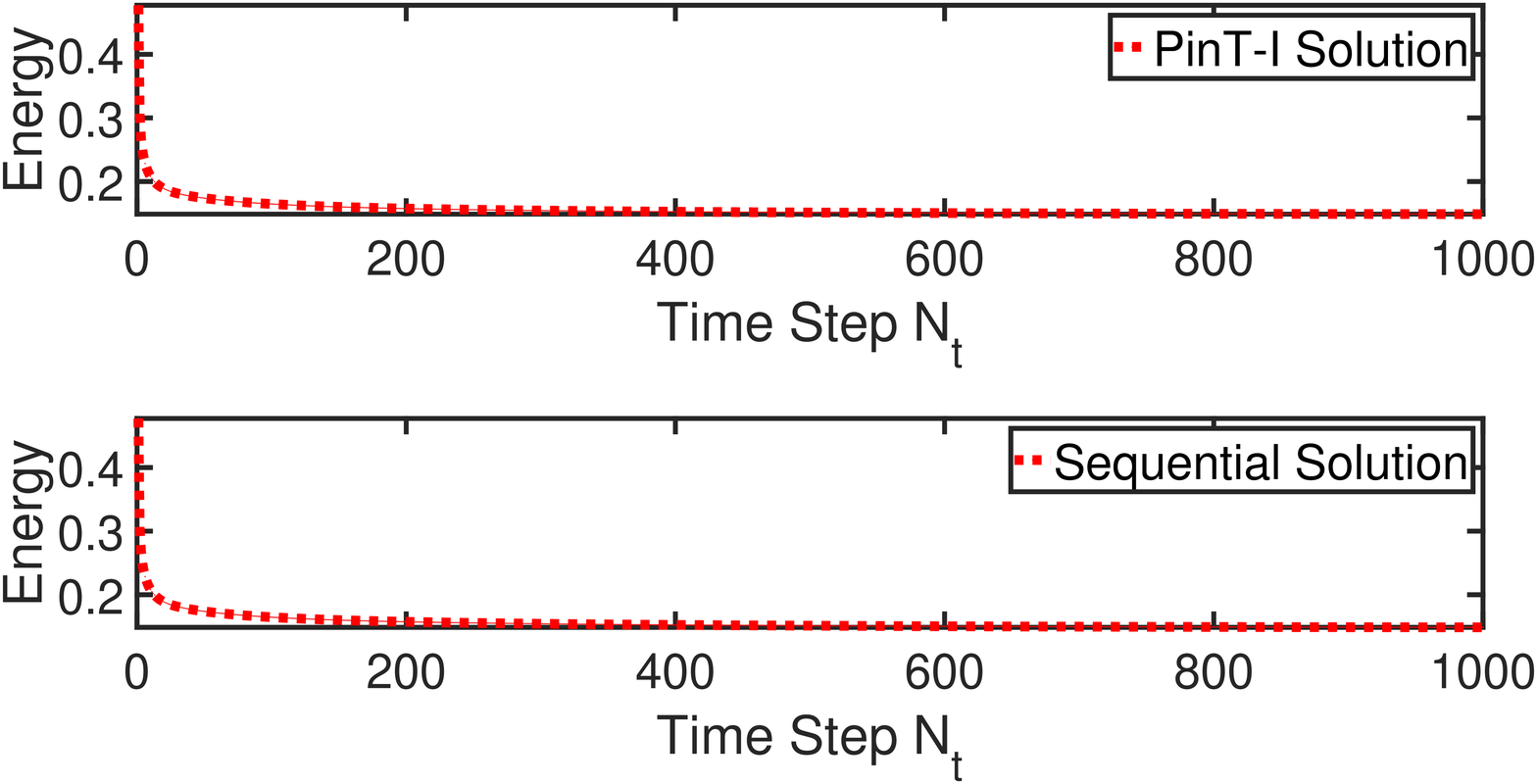} }}
     \subfloat{{\includegraphics[height=6cm,width=6.5cm]{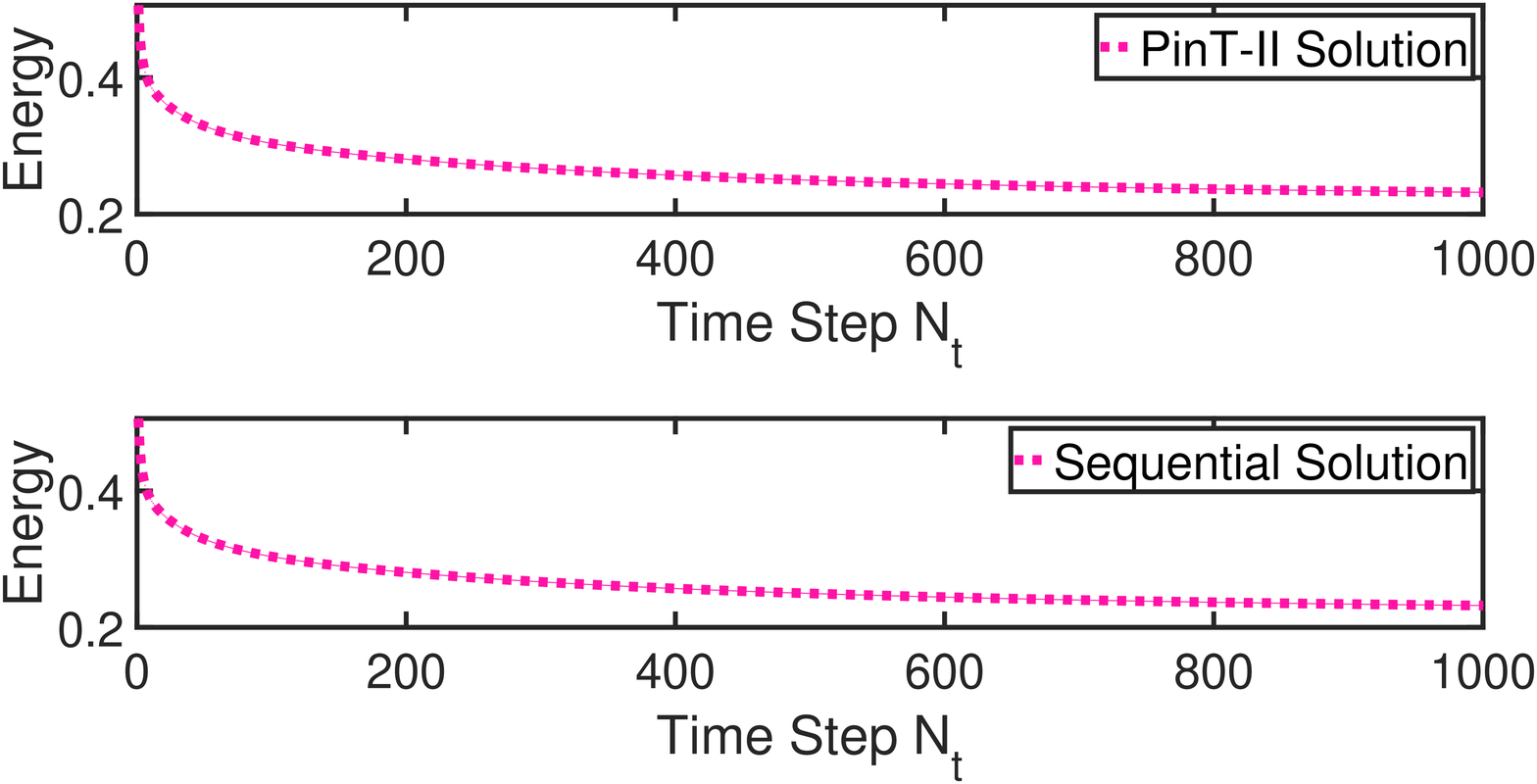} }}
    \caption{Discrete energy of solutions with fixed $h=1/128, \Delta t=10^{-4}, \alpha=0.05$. On the left: PinT-I with $\epsilon^2=0.025$; On the right: PinT-II with $\epsilon^2=0.075$.}
    \label{nonlinch_fig4}
\end{figure}
In Figure \ref{nonlinch_fig4} we show the non-increasing property of the discrete energy of both integrators with different choice of $\epsilon$.

To perform numerical experiments in 2D we consider the domain $\Omega=(0, 1)\times (0, 1)$. We observe the similar convergence behaviour of both PinT-I and PinT-II as in the case of 1D. Figure \ref{nonlinch_fig5} shows the convergence of PinT-II for long time window. The method shows very robust convergence behaviour in terms of iteration count over long time window with respect to different $\epsilon$. 
\begin{figure}[h!]
    \centering
    \subfloat{{\includegraphics[height=5cm,width=6.5cm]{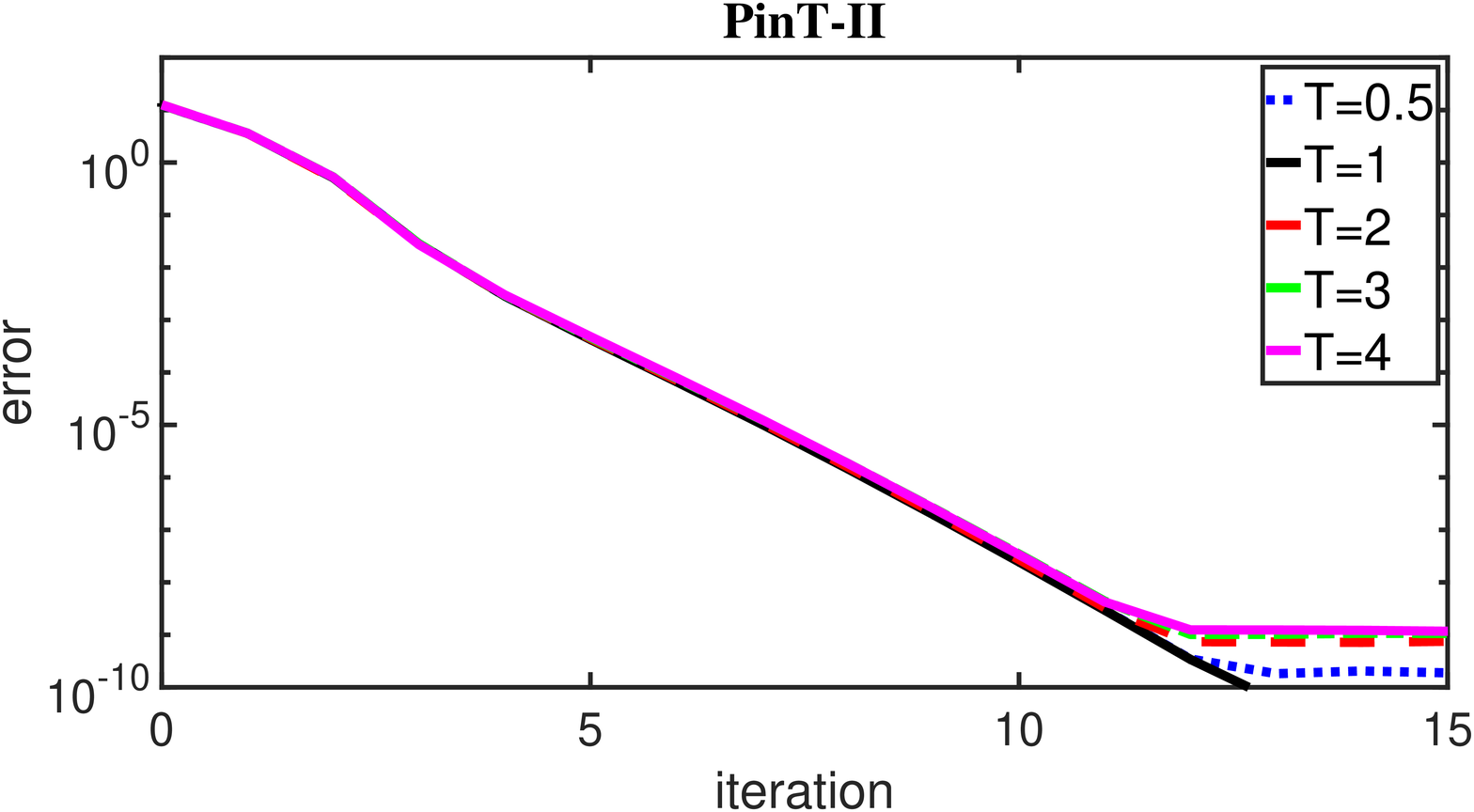} }}
     \subfloat{{\includegraphics[height=5cm,width=6.5cm]{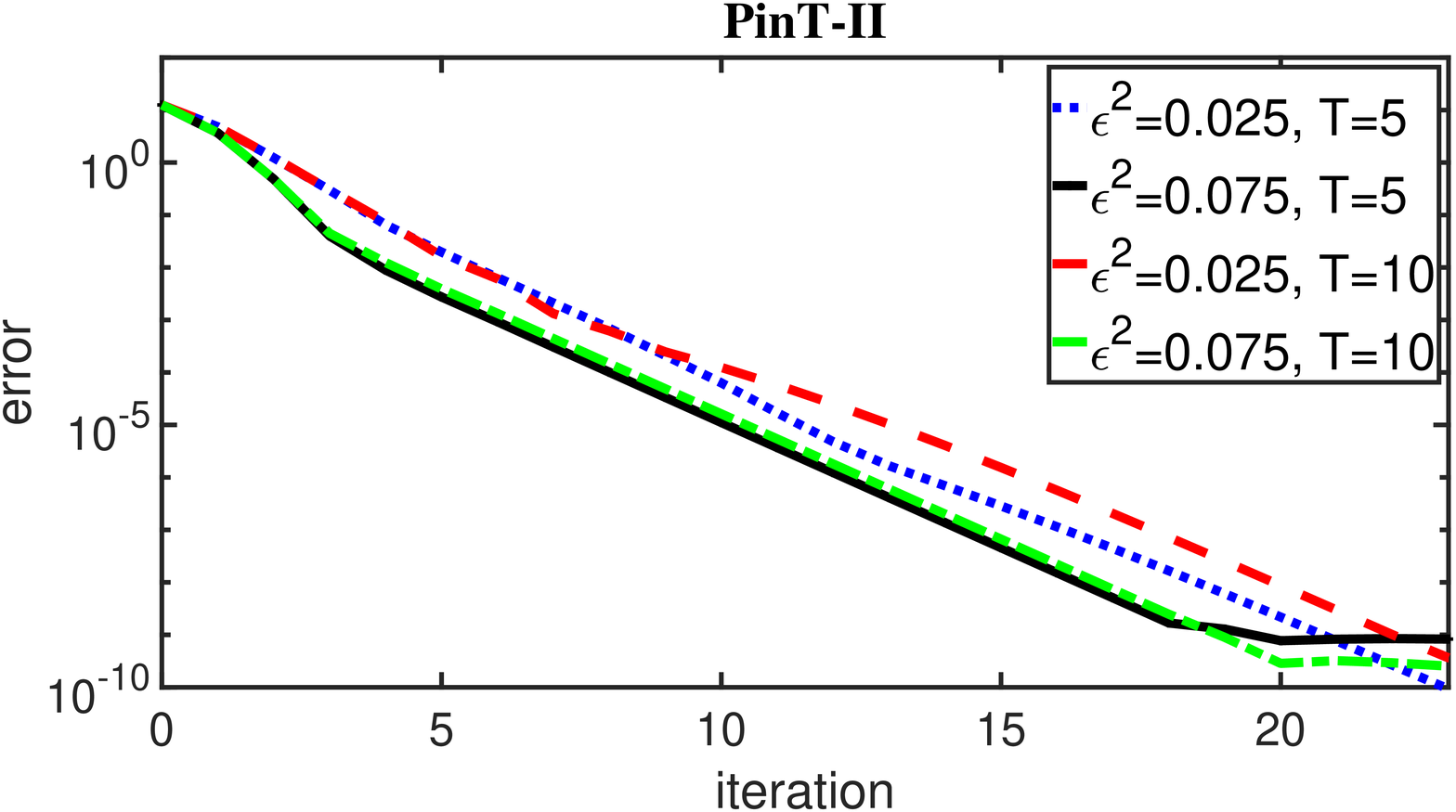} }}
    \caption{On the left: PinT-II convergence for different $T$ with fixed $h=1/32, \Delta t=10^{-4}, \alpha=0.05, \epsilon^2=0.075$. On the right: PinT-II convergence for different $T$ amd $ \epsilon$ with fixed $h=1/32, \Delta t=10^{-3}, \alpha=0.25$.}
    \label{nonlinch_fig5}
\end{figure}
\subsection{Verification of Solutions \& Physical Phenomenon}
In this subsection we present numerical solution of the nonlinear CH equation. First we present numerical solution of PinT-I in 1D. We take the computational domain $\Omega=(0, 1)$ with mesh size $h=1/128$. The initial solution of the CH equation is taken as $u_0(x, 0)=0.75\sin(2\pi x)+0.25\cos(4\pi x)$. The solution of the CH equation using PinT-I obtained over the time interval $(0, T=0.1)$ with fixed time step $\Delta t=10^{-4}, \epsilon=0.1$, and parameter $\alpha=0.005$. We start with random initial guess for PinT-I and then obtained follow up iterations, which can be seen from Figure \ref{pint1_earlystage}. After 2nd iteration still there are smaller frequency mode in the solution, and which are gone in 3rd iteration. We essentially get the solution at 4th iteration, and can be seen from middle plot of Figure \ref{pint1_earlystage2}. On the last plot of Figure \ref{pint1_earlystage2} we can see that energy is decaying over the time, which implies that the method PinT-I preserved the energy minimization property. 
\begin{figure}[h]
    \centering
    \subfloat{{\includegraphics[height=5.5cm,width=4.5cm]{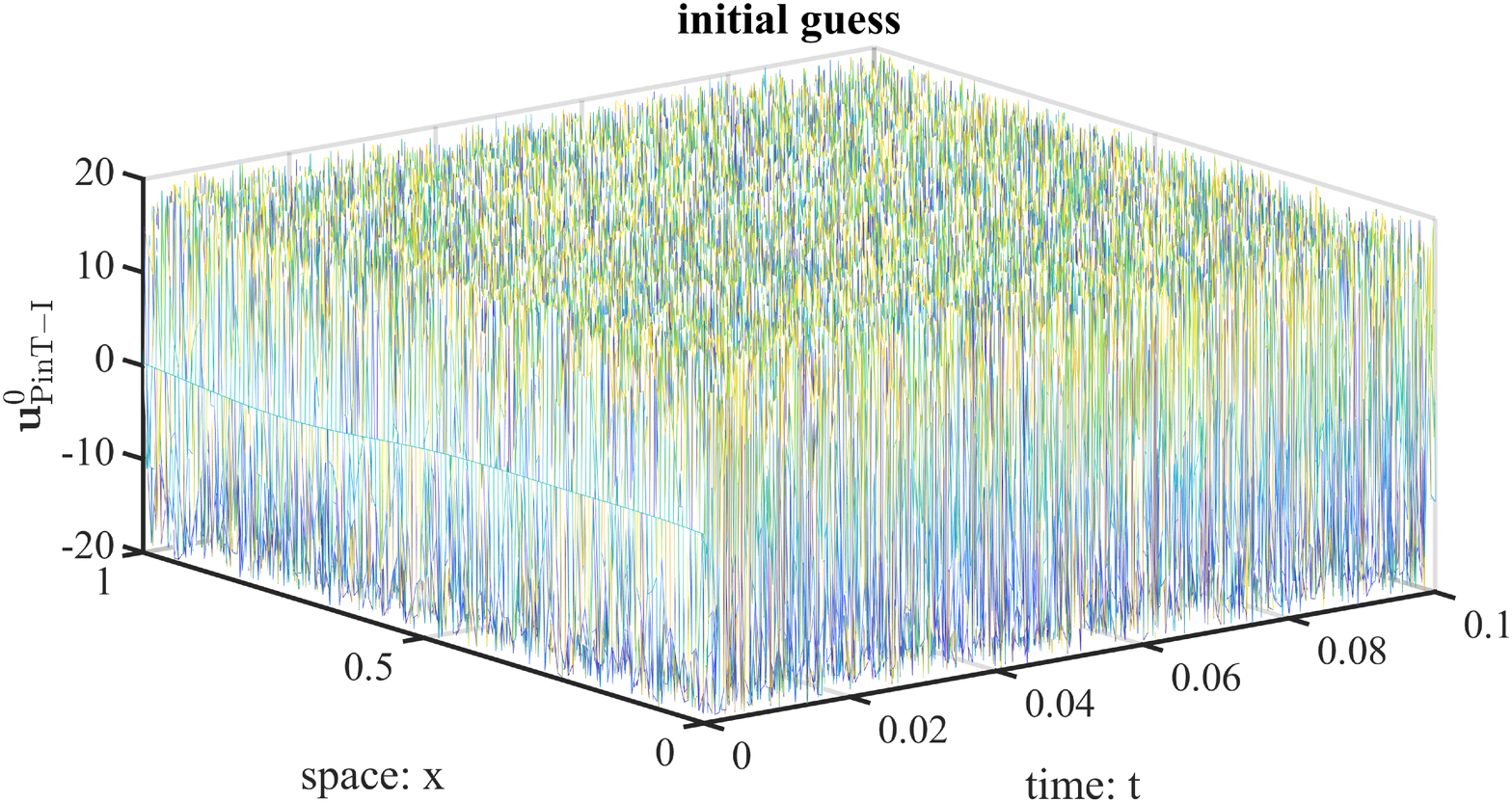} }}
     \subfloat{{\includegraphics[height=5.5cm,width=4.5cm]{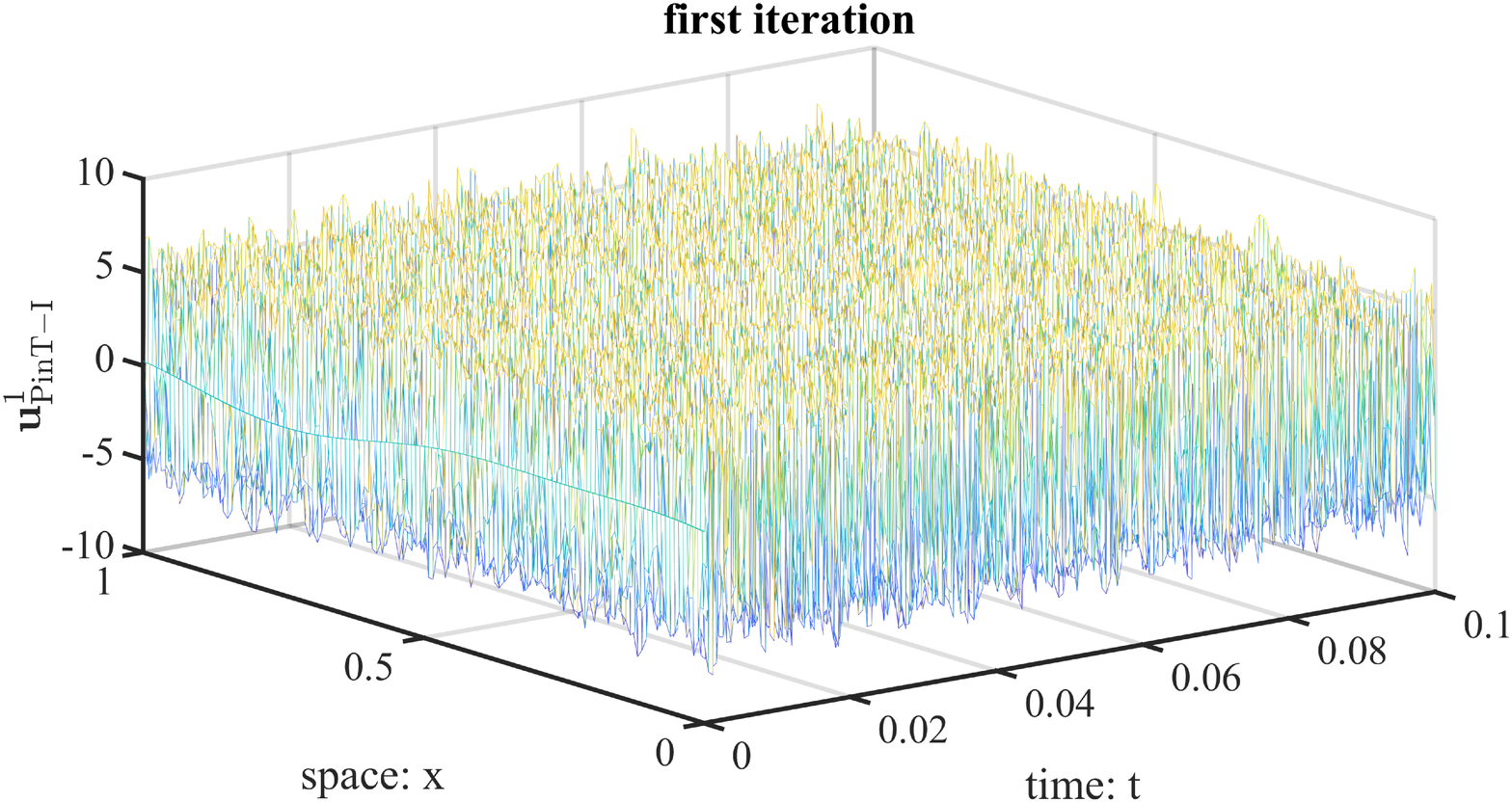} }}
     \subfloat{{\includegraphics[height=5.5cm,width=4.5cm]{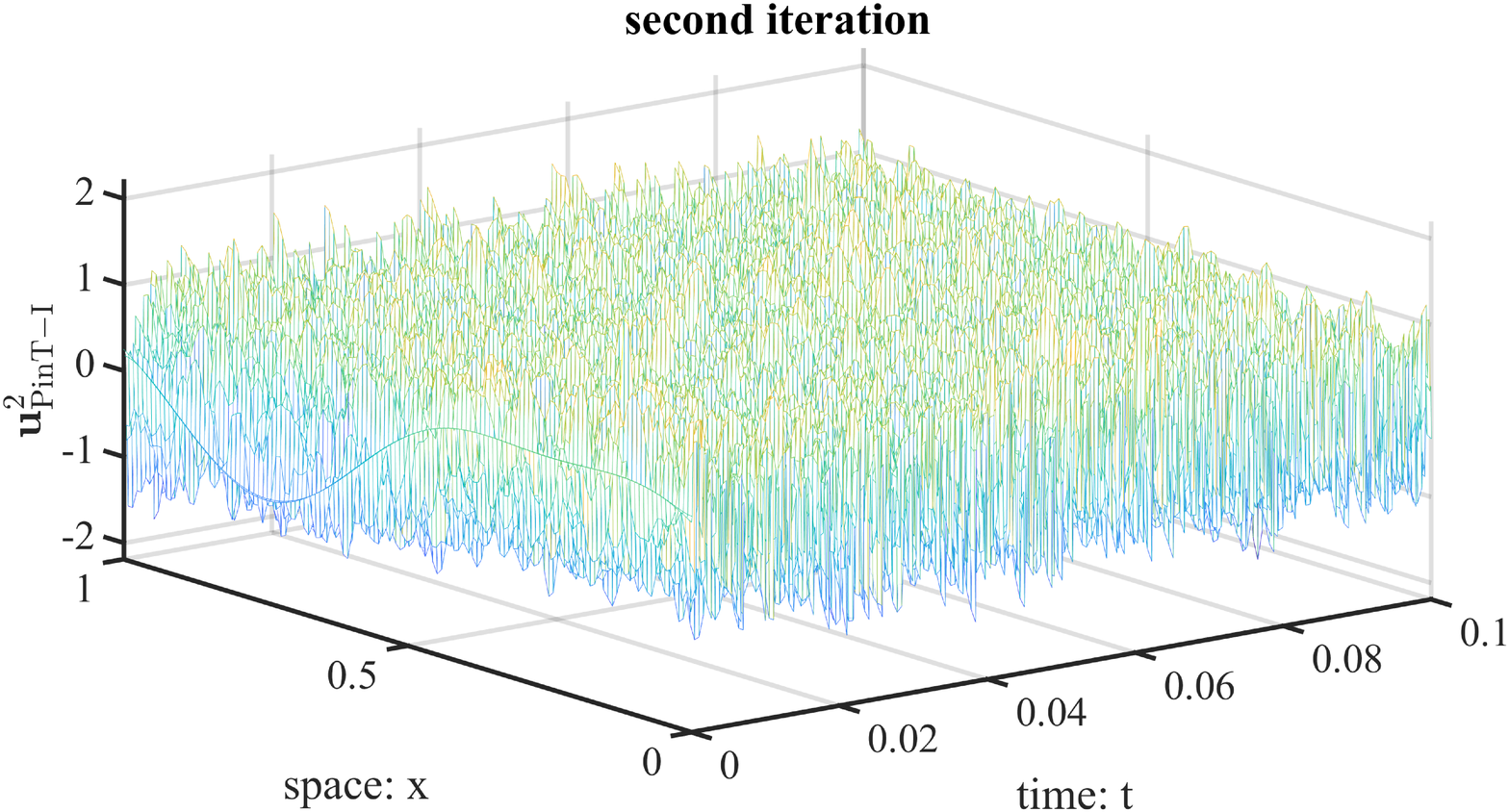} }}
    \caption{From left to right: solution of PinT-I, from initial guess to 2nd iteration}
    \label{pint1_earlystage}
\end{figure}
\begin{figure}[h]
    \centering
    \subfloat{{\includegraphics[height=5.5cm,width=4.5cm]{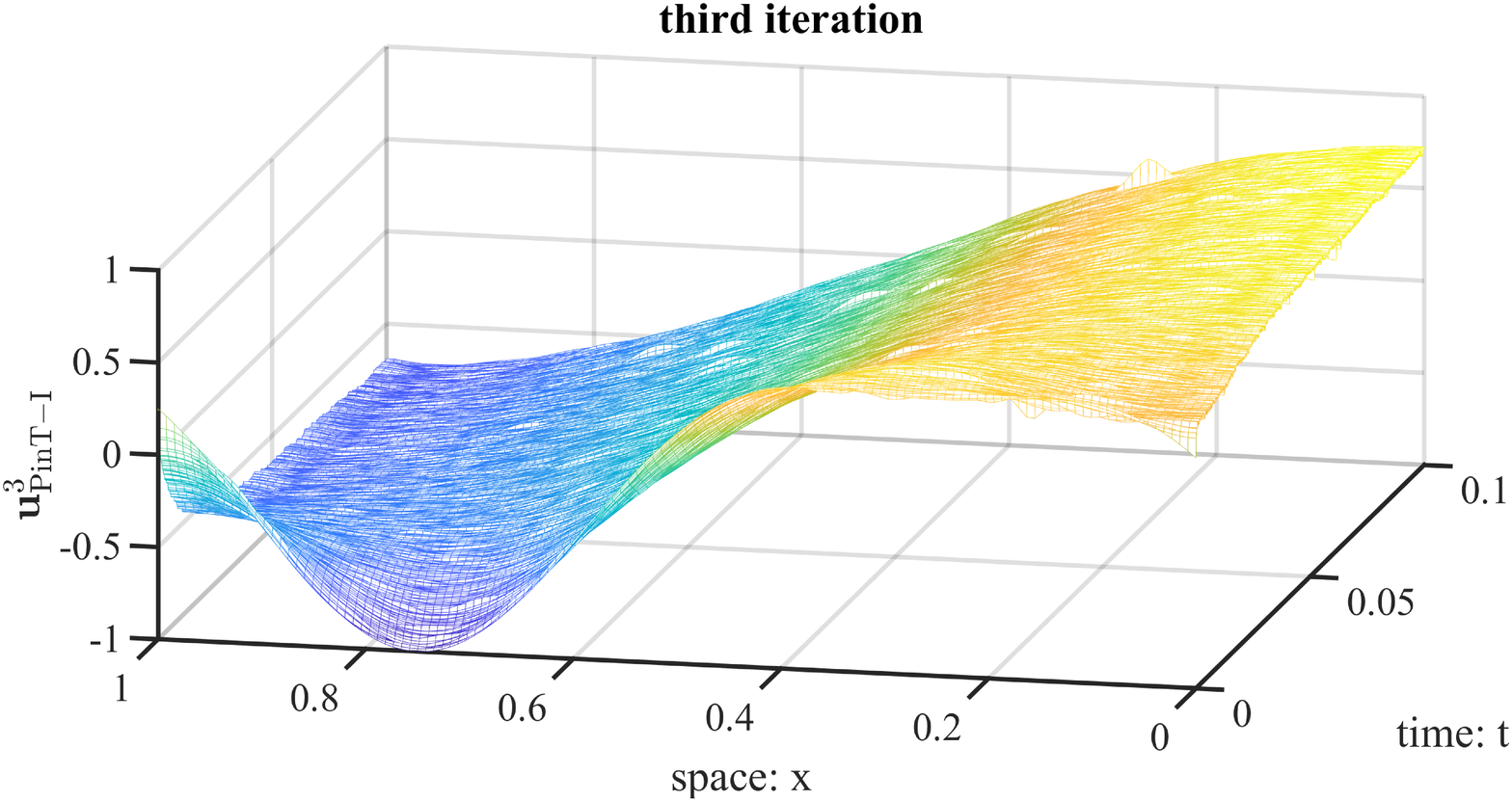} }}
     \subfloat{{\includegraphics[height=5.5cm,width=4.5cm]{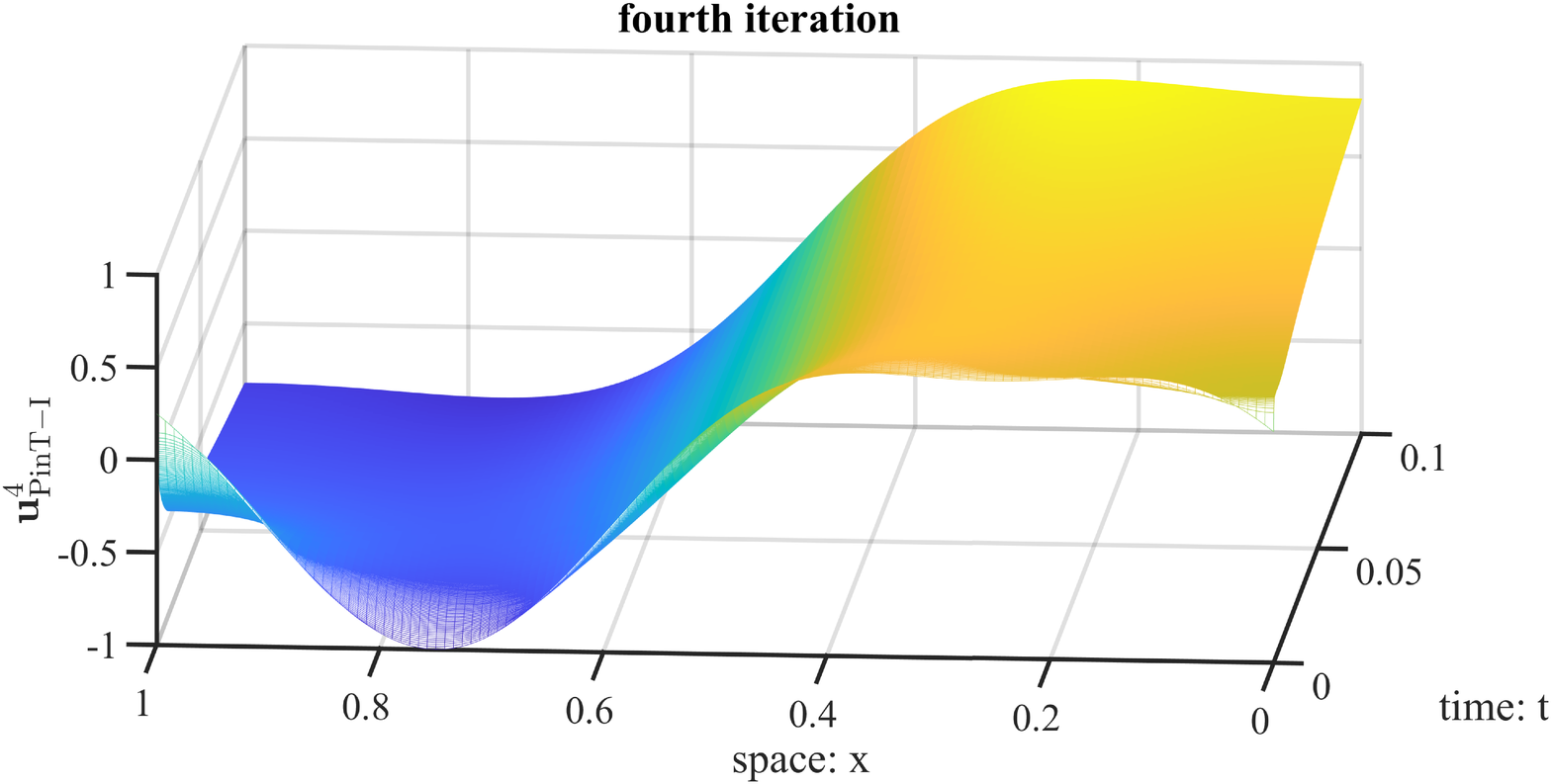} }}
     \subfloat{{\includegraphics[height=5.5cm,width=4.5cm]{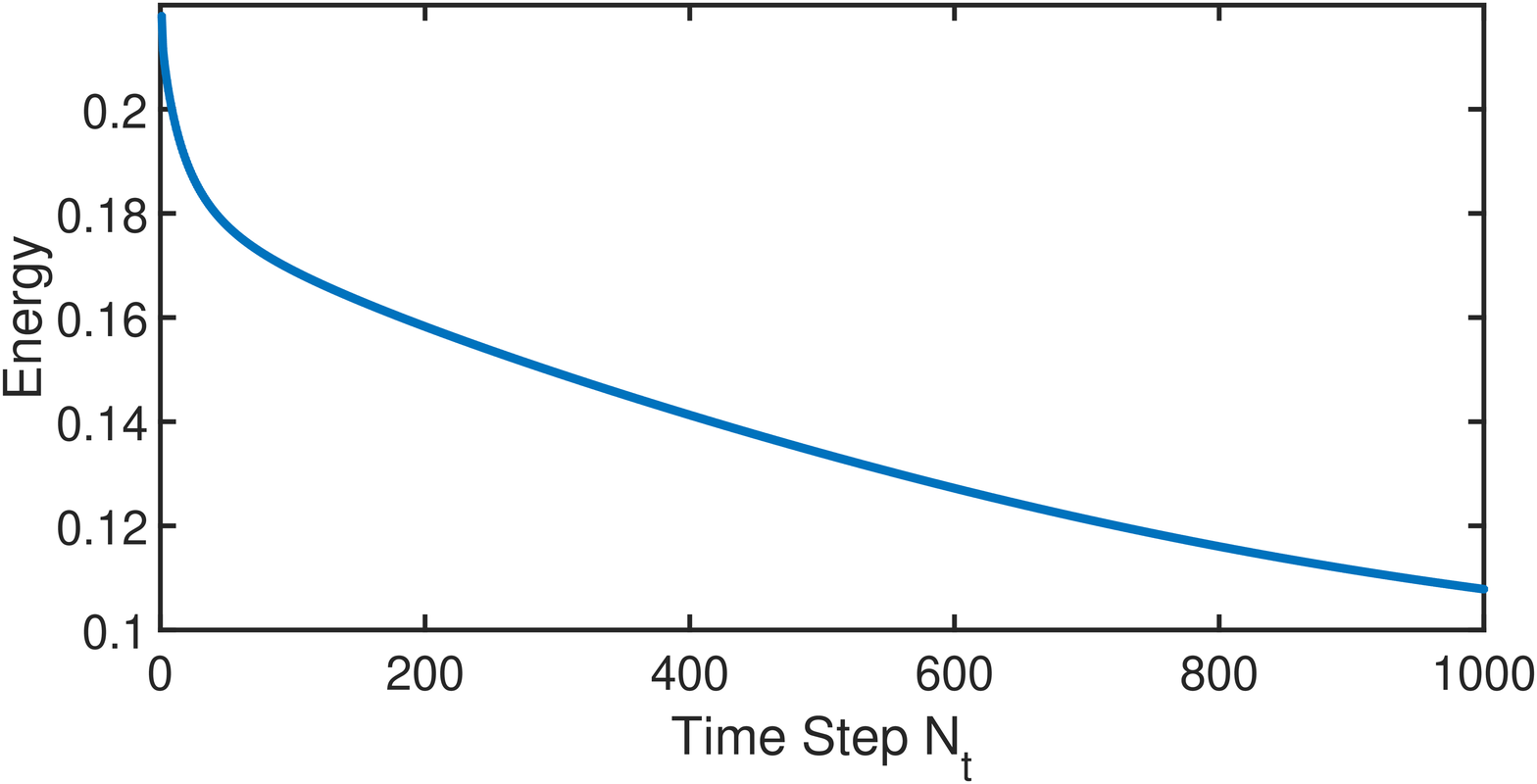} }}
    \caption{First two plot: 3rd and 4th iteration of PinT-I respectively; Last plot: discrete energy vs time }
    \label{pint1_earlystage2}
\end{figure}

To perform numerical experiments in 2D for PinT-II, we fixed the computational domain $\Omega=(0, 1)^2$ with the spatial discretization parameters $h_x=h_y=1/64$. We take the intimal solution $u_0(x, y, 0)=0.1\rand(x,y)$, where $\rand(x,y)$ is a random number in $[-1, 1]$. The solution of the CH equation using PinT-II method obtained over the time interval $(0, T=0.1)$ with fixed time step $\Delta t=10^{-5}$. The evolution of solution of the CH equation with $\epsilon=0.01, \alpha=0.05$ can be seen from Figure \ref{pint2_2d} for the phase separation and phase coarsening stage. We plot the discrete energy on the left of Figure \ref{pint2_energy}, and we find that energy is decaying over the time. To observe the conservation of mass over time interval see the right plot of Figure \ref{pint2_energy}. Hence we can see that the newly formulated PinT-II method respects all the physical phenomena, including the energy decay and mass conservation.

\begin{figure}[h]
    \centering %
\begin{subfigure}{0.2\textwidth}
  \includegraphics[height=4cm,width=4cm]{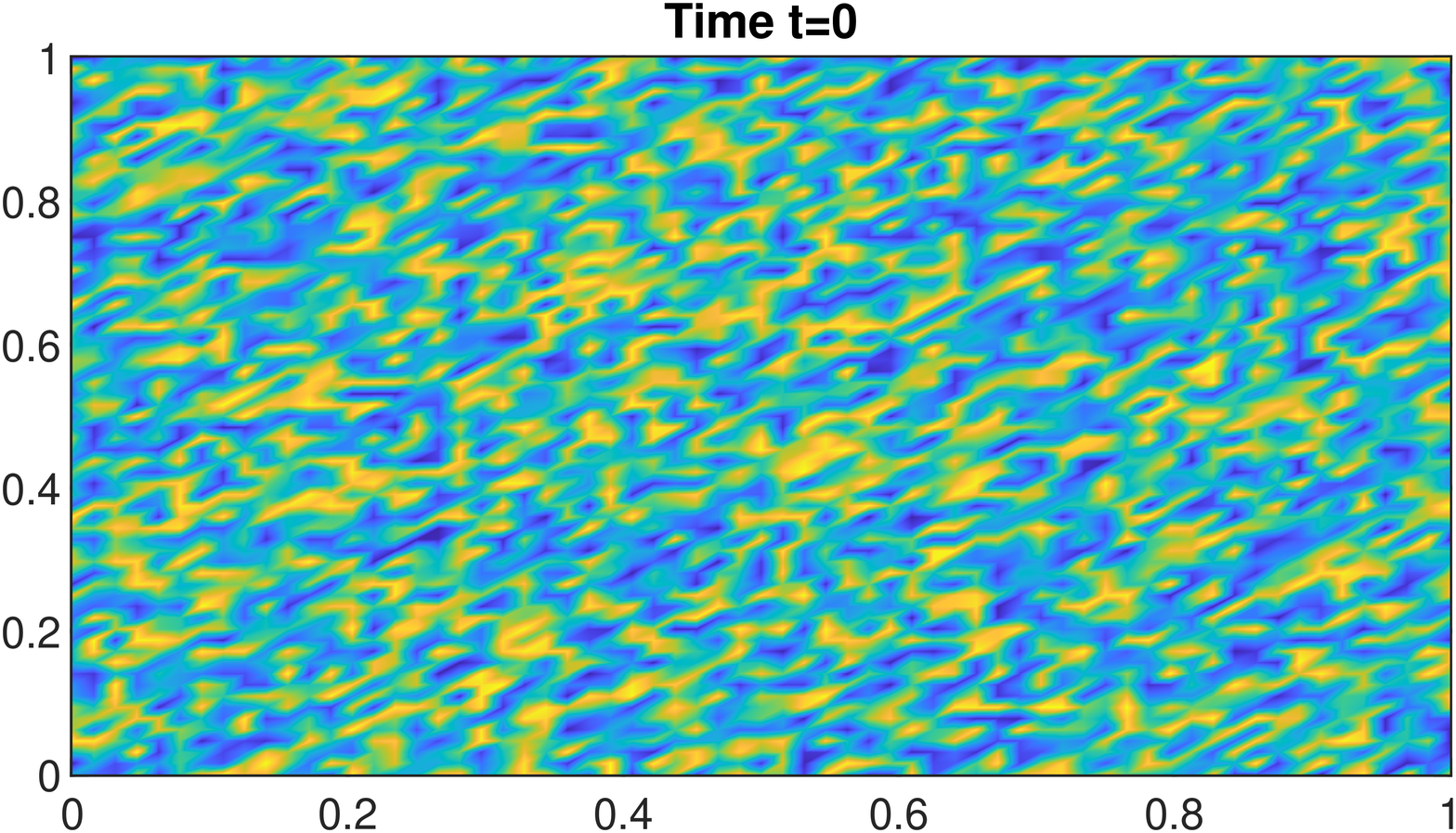}
\end{subfigure} 
\begin{subfigure}{0.2\textwidth}
  \includegraphics[height=4cm,width=4cm]{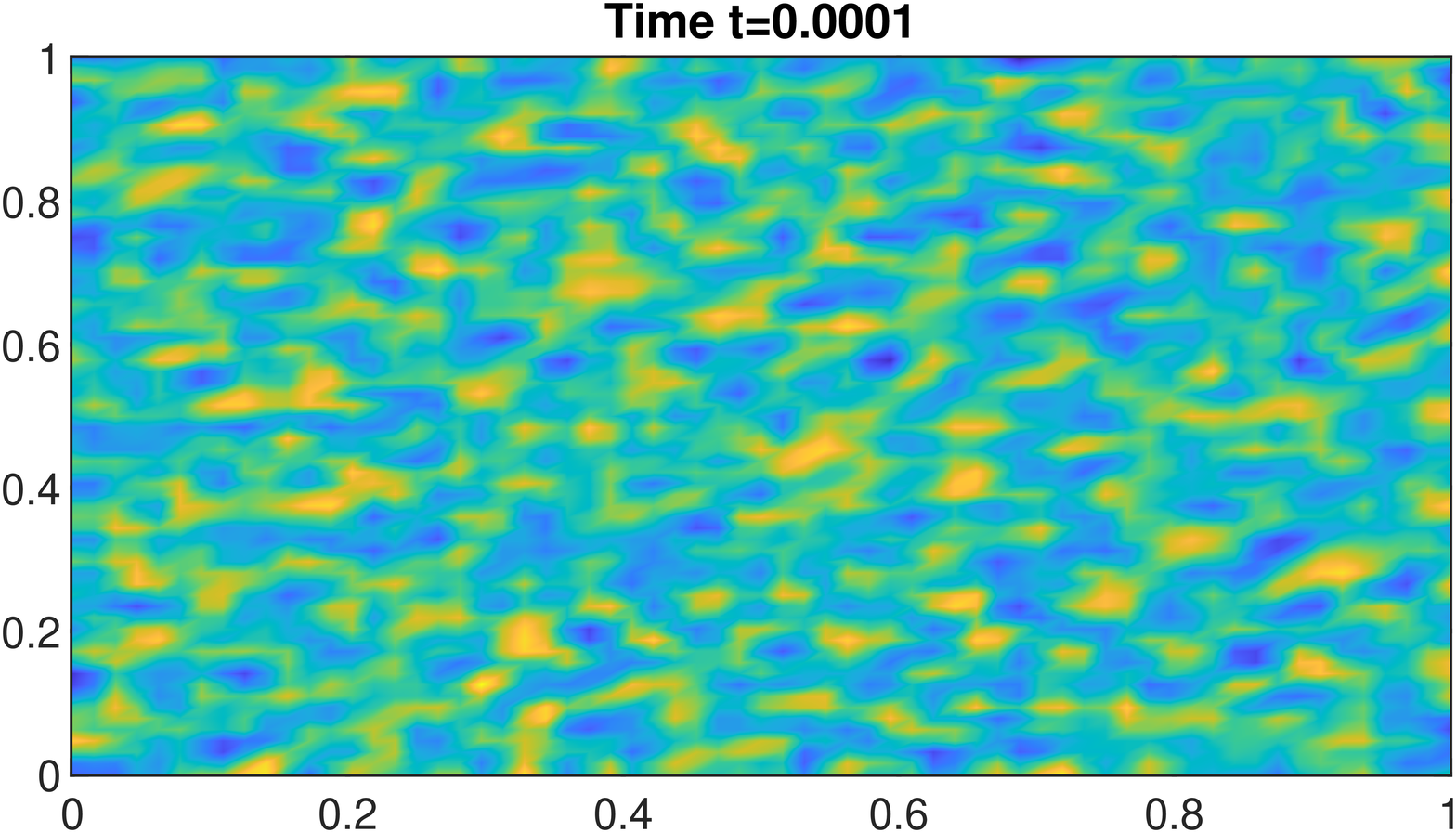} 
\end{subfigure}
\begin{subfigure}{0.2\textwidth}
 \includegraphics[height=4cm,width=4cm]{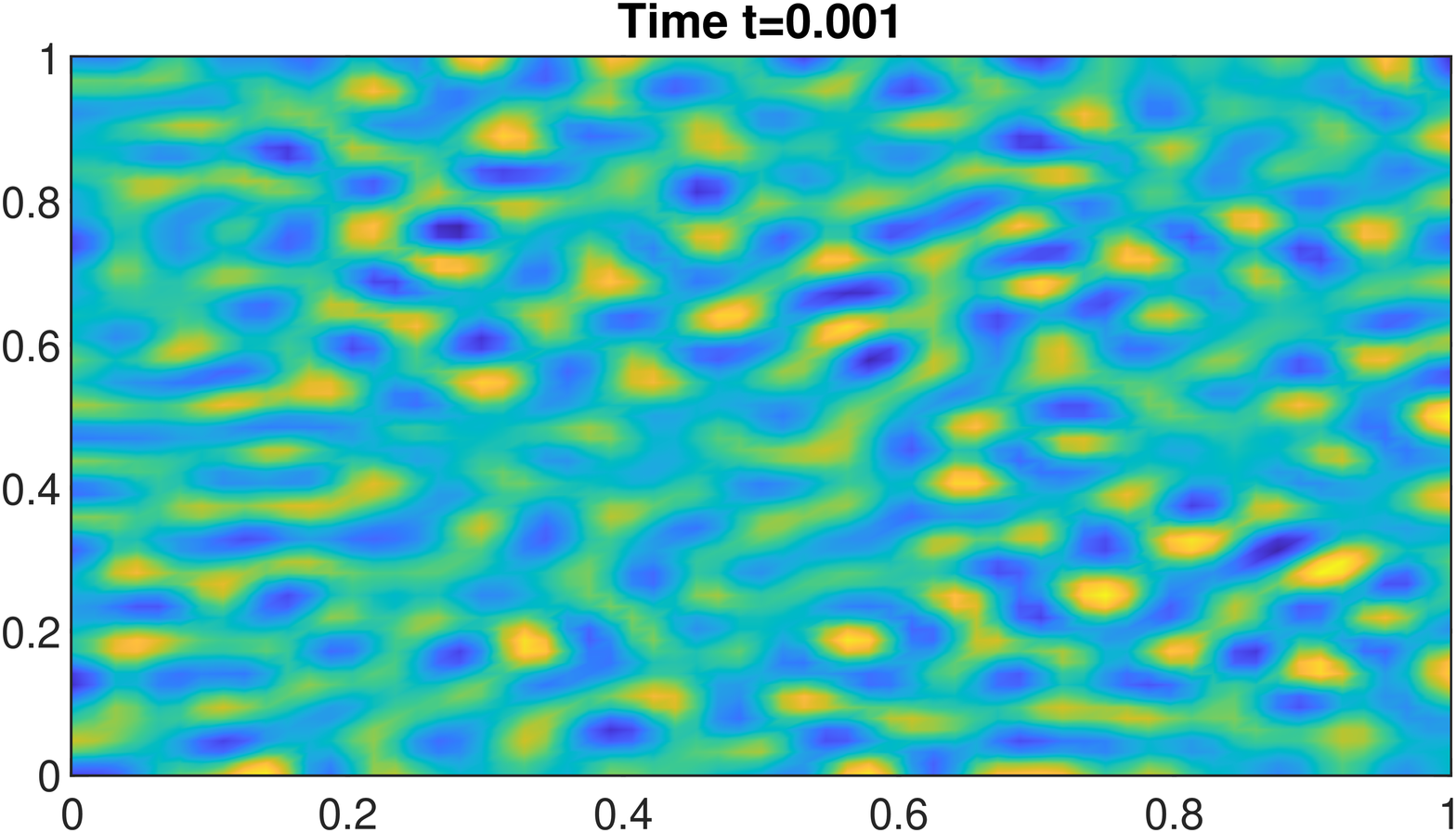}
\end{subfigure}
\begin{subfigure}{0.2\textwidth}
 \includegraphics[height=4cm,width=4cm]{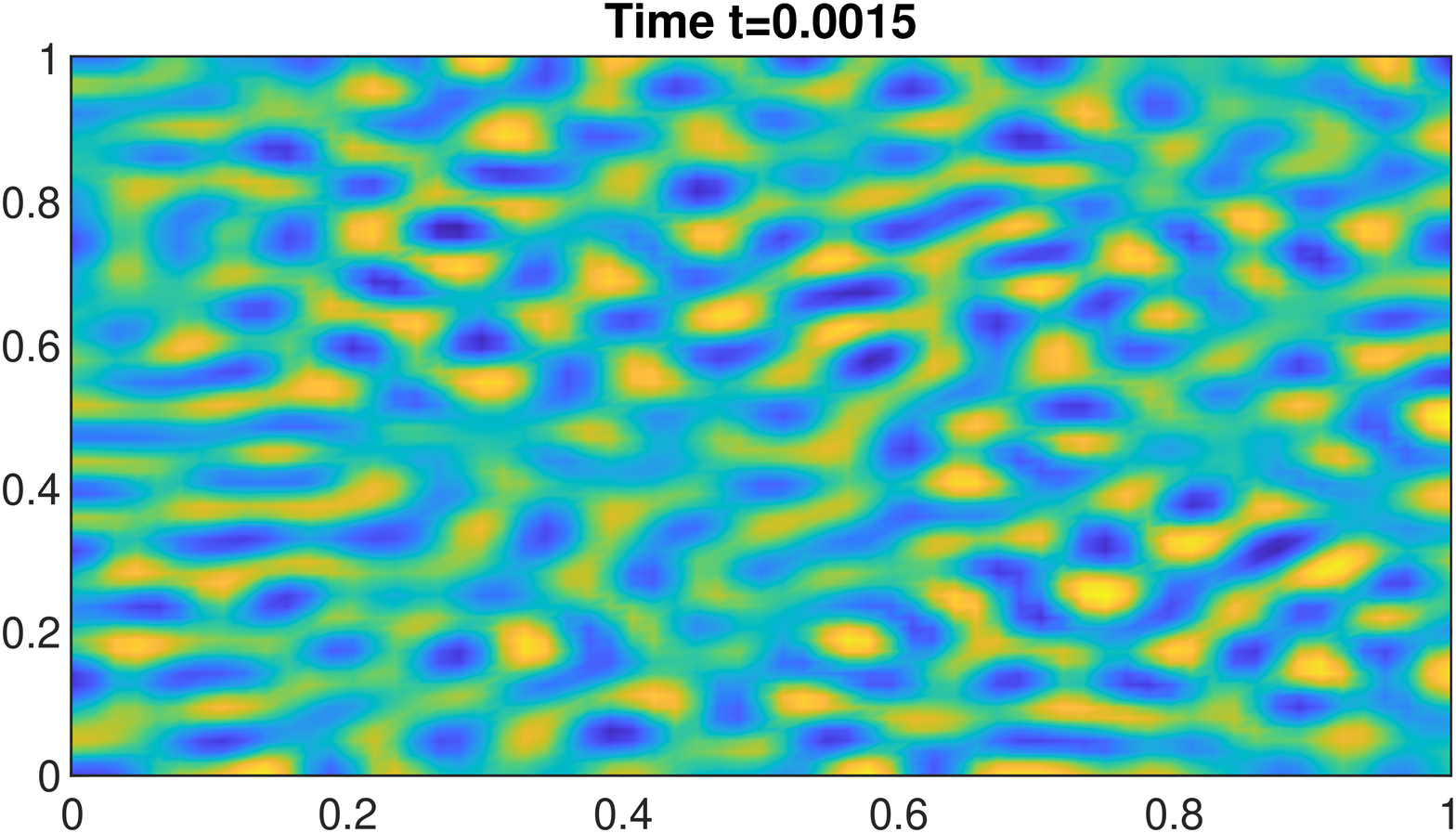}
\end{subfigure}

\medskip
\begin{subfigure}{0.2\textwidth}
  \includegraphics[height=4cm,width=4cm]{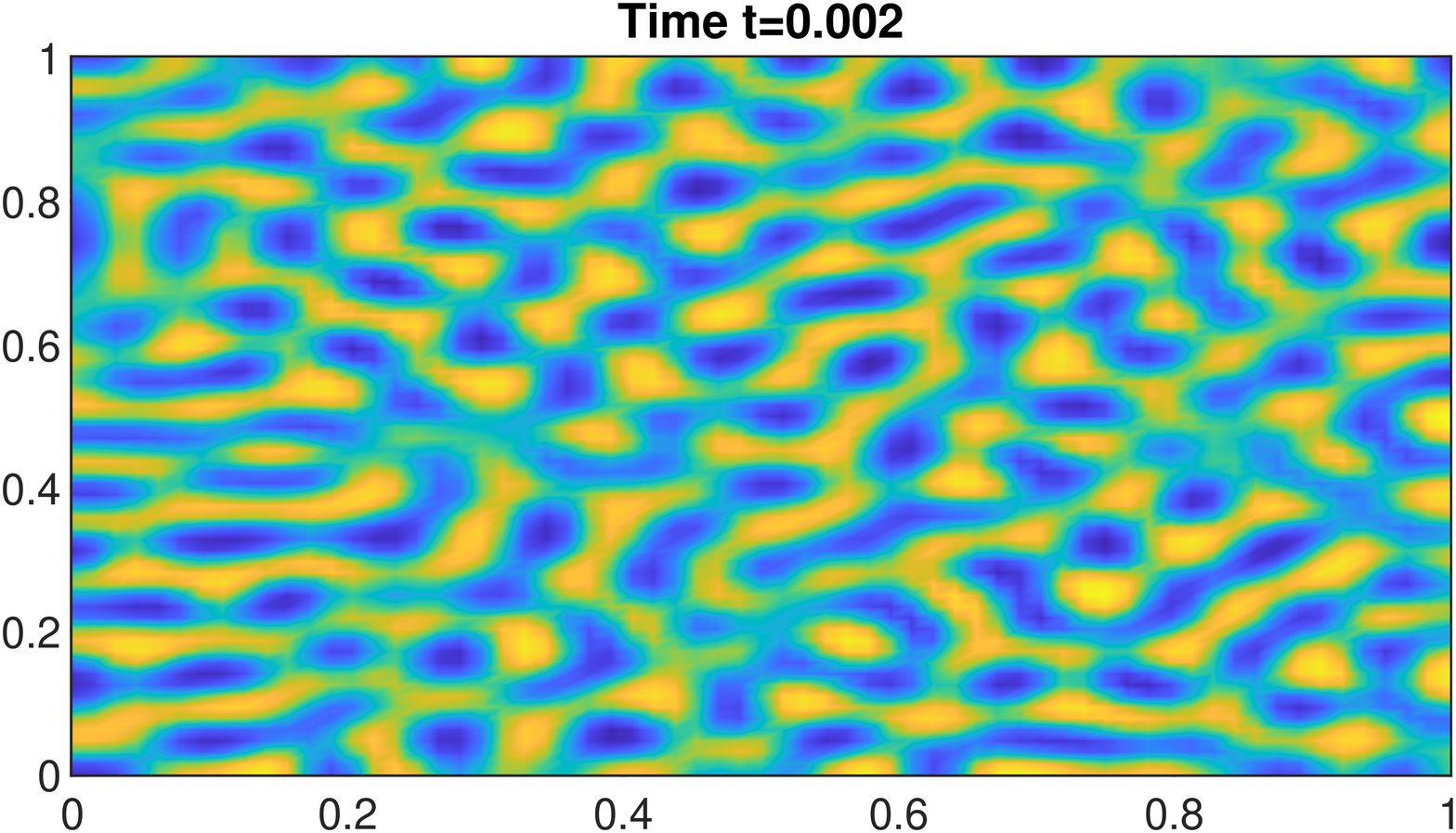}
\end{subfigure}
\begin{subfigure}{0.2\textwidth}
  \includegraphics[height=4cm,width=4cm]{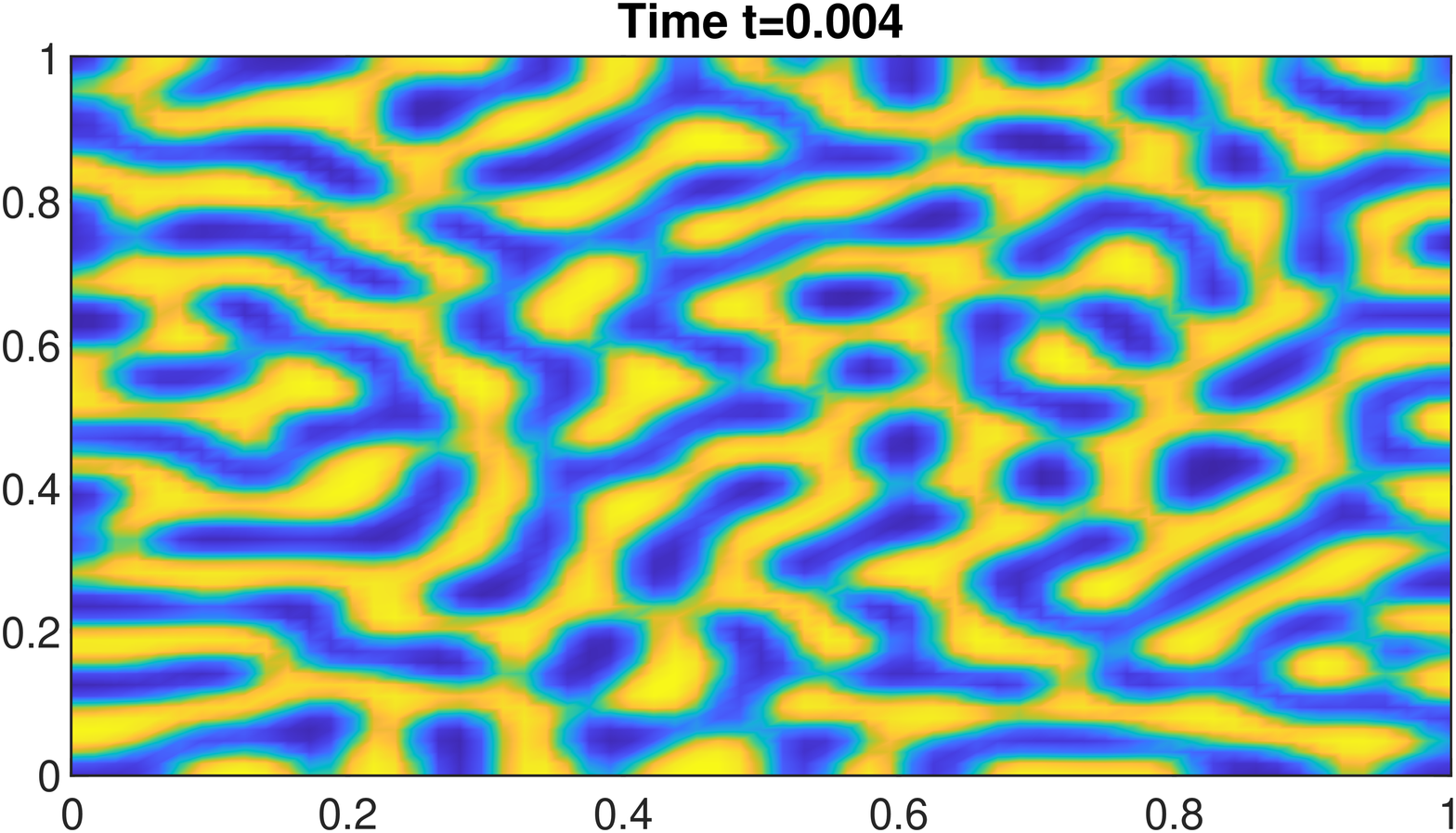}
\end{subfigure}
\begin{subfigure}{0.2\textwidth}
  \includegraphics[height=4cm,width=4cm]{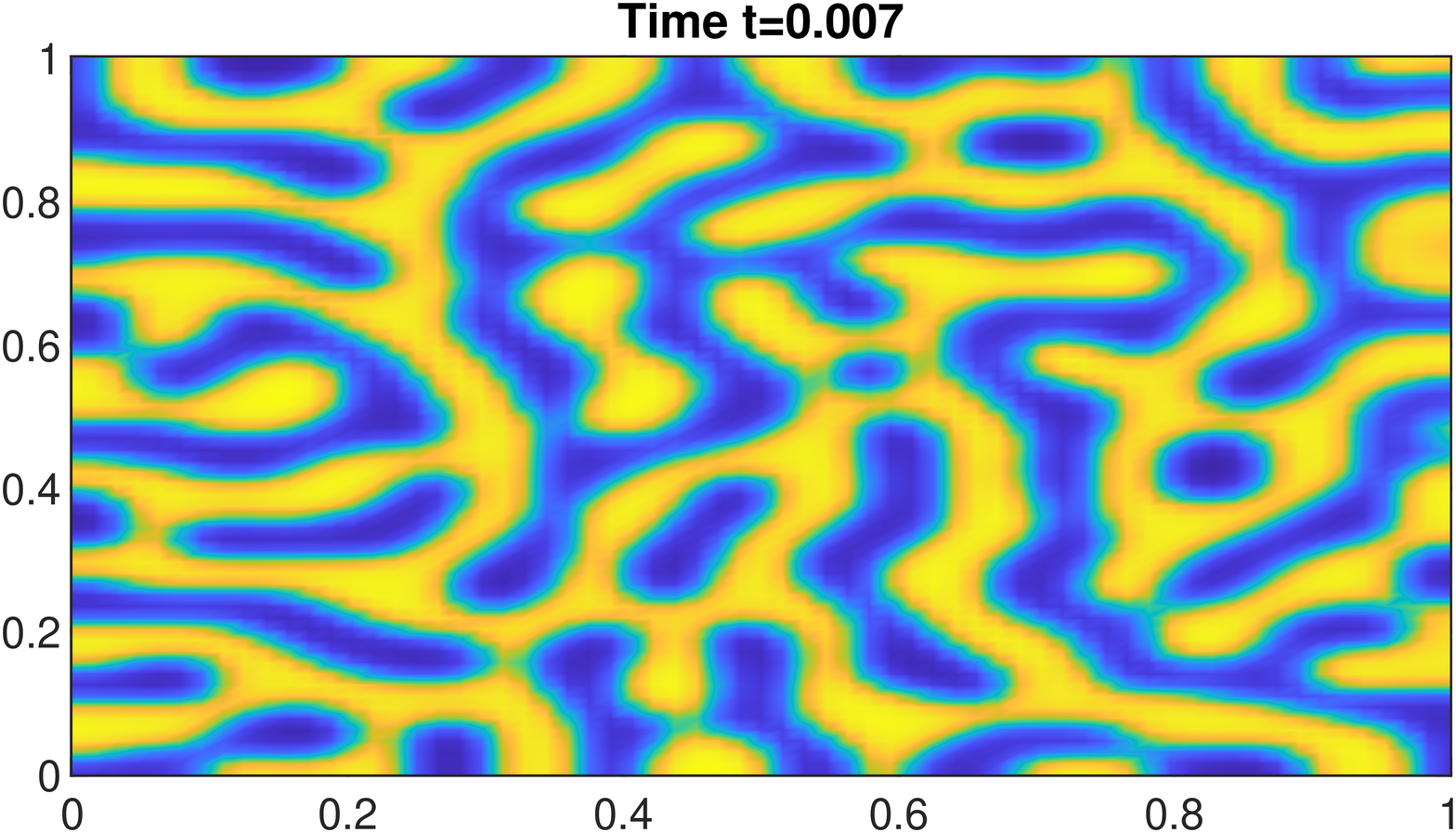}
\end{subfigure}
\begin{subfigure}{0.2\textwidth}
  \includegraphics[height=4cm,width=4cm]{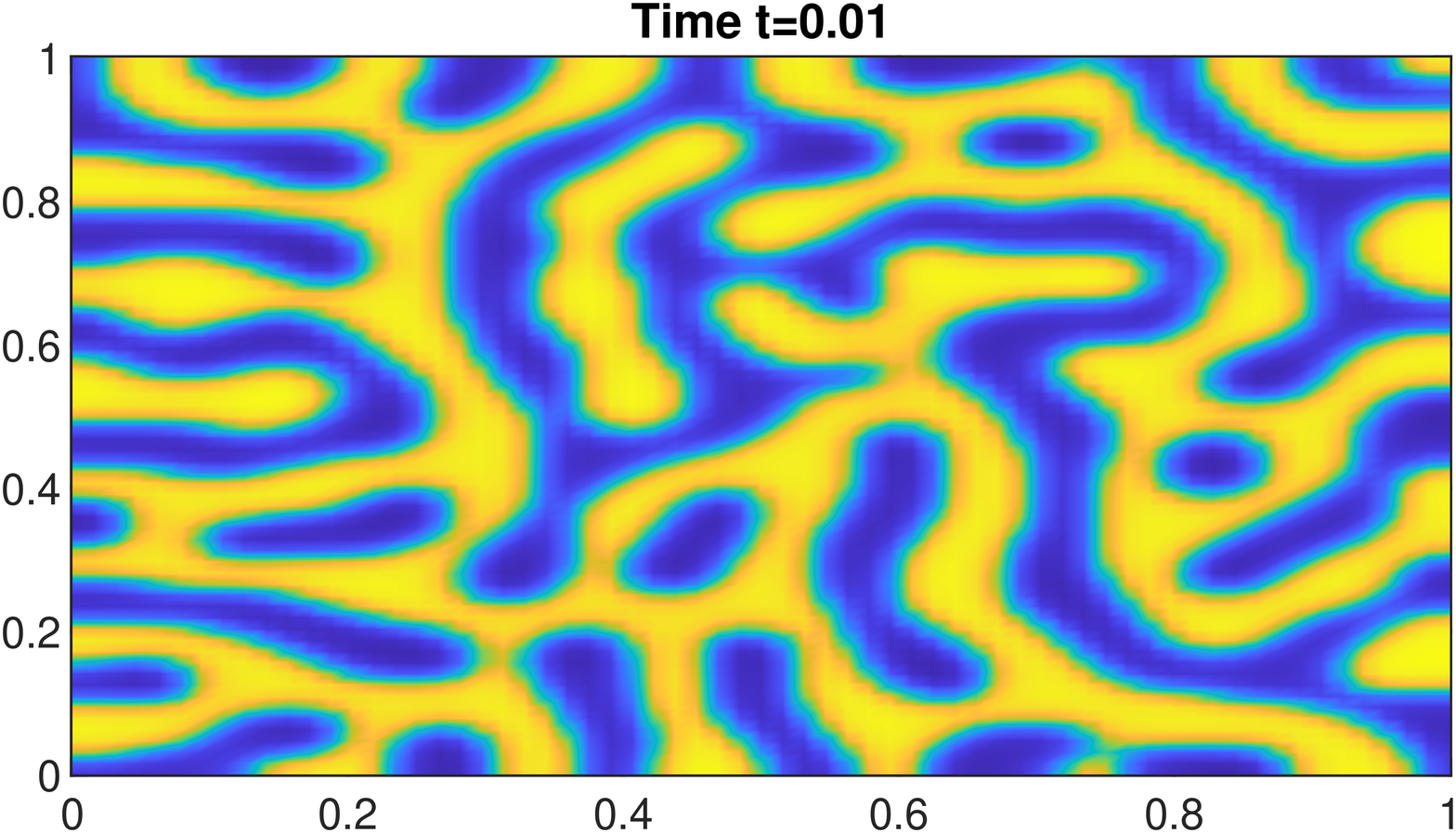}
\end{subfigure}

\medskip
\begin{subfigure}{0.2\textwidth}
  \includegraphics[height=4cm,width=4cm]{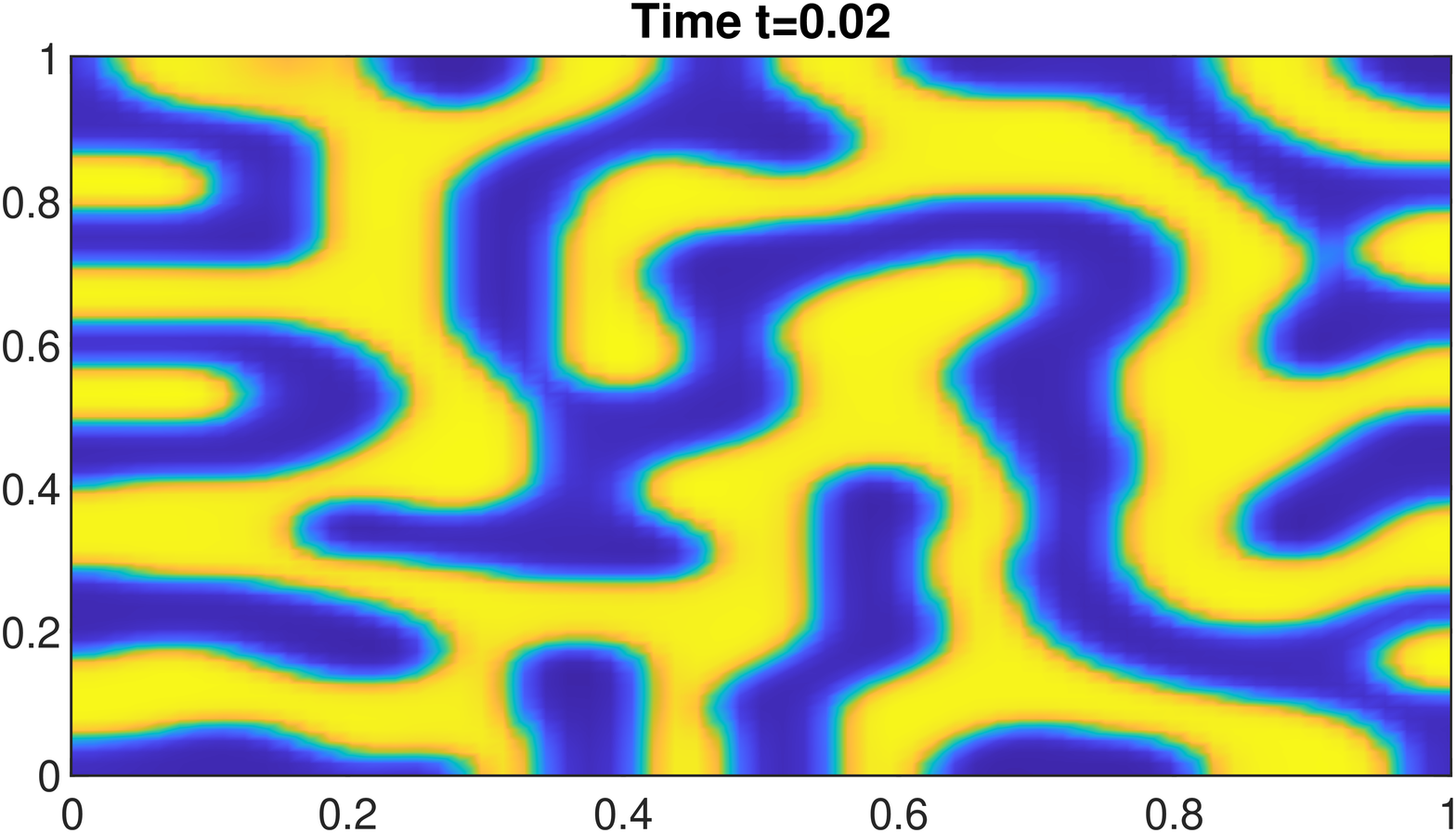}
\end{subfigure}
\begin{subfigure}{0.2\textwidth}
  \includegraphics[height=4cm,width=4cm]{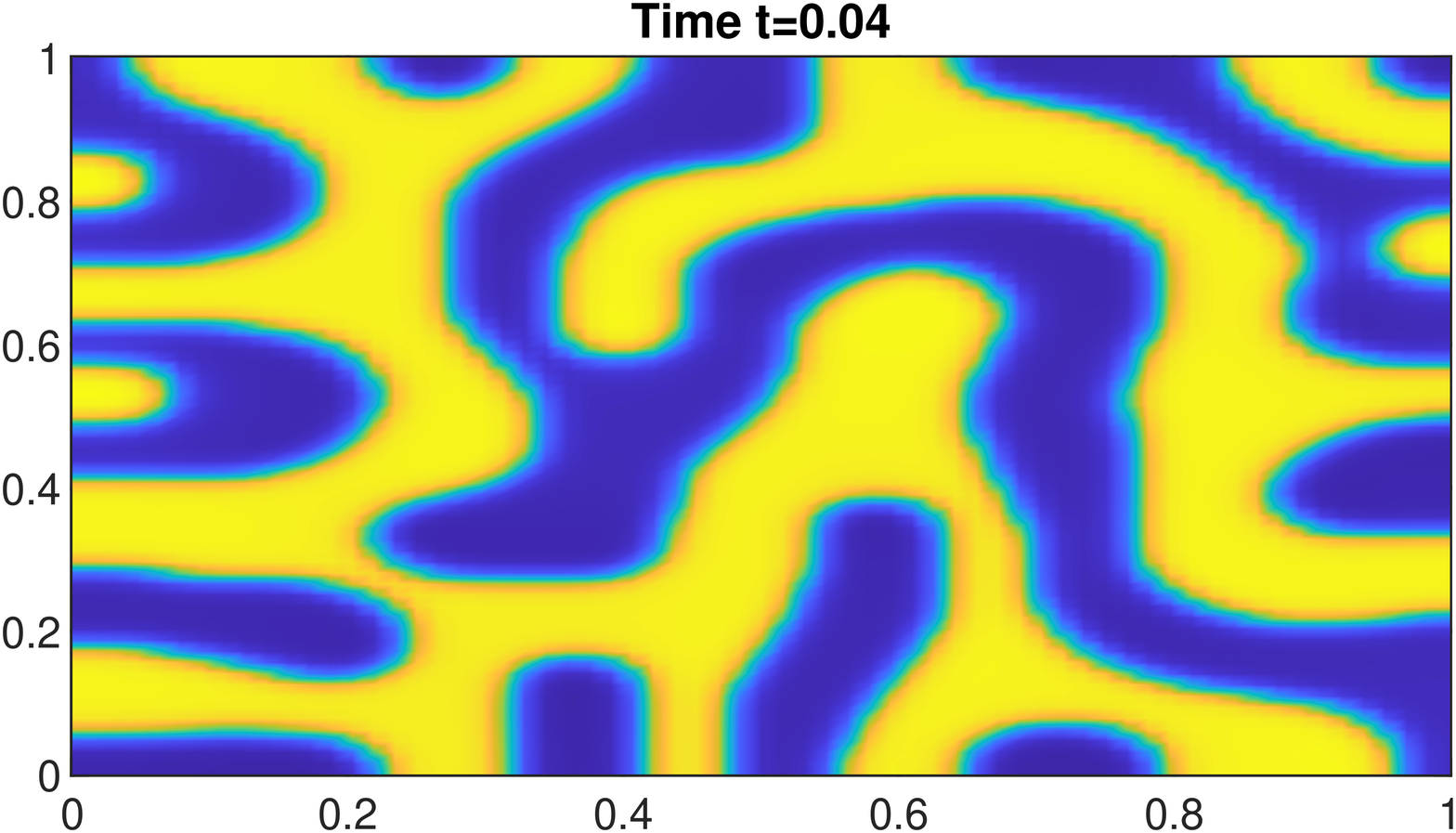}
\end{subfigure} 
\begin{subfigure}{0.2\textwidth}
  \includegraphics[height=4cm,width=4cm]{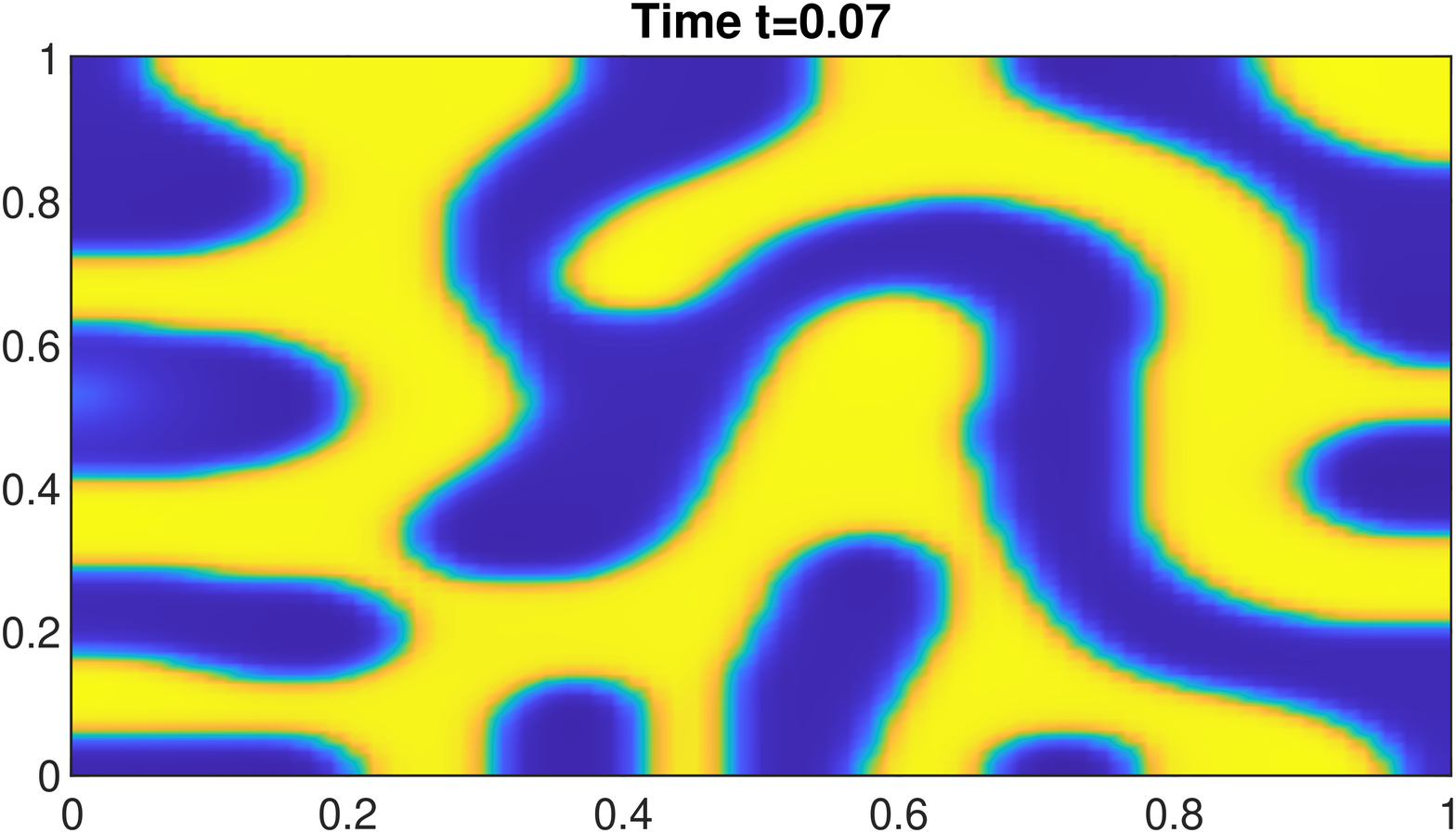}
\end{subfigure}
\begin{subfigure}{0.2\textwidth}
  \includegraphics[height=4cm,width=4cm]{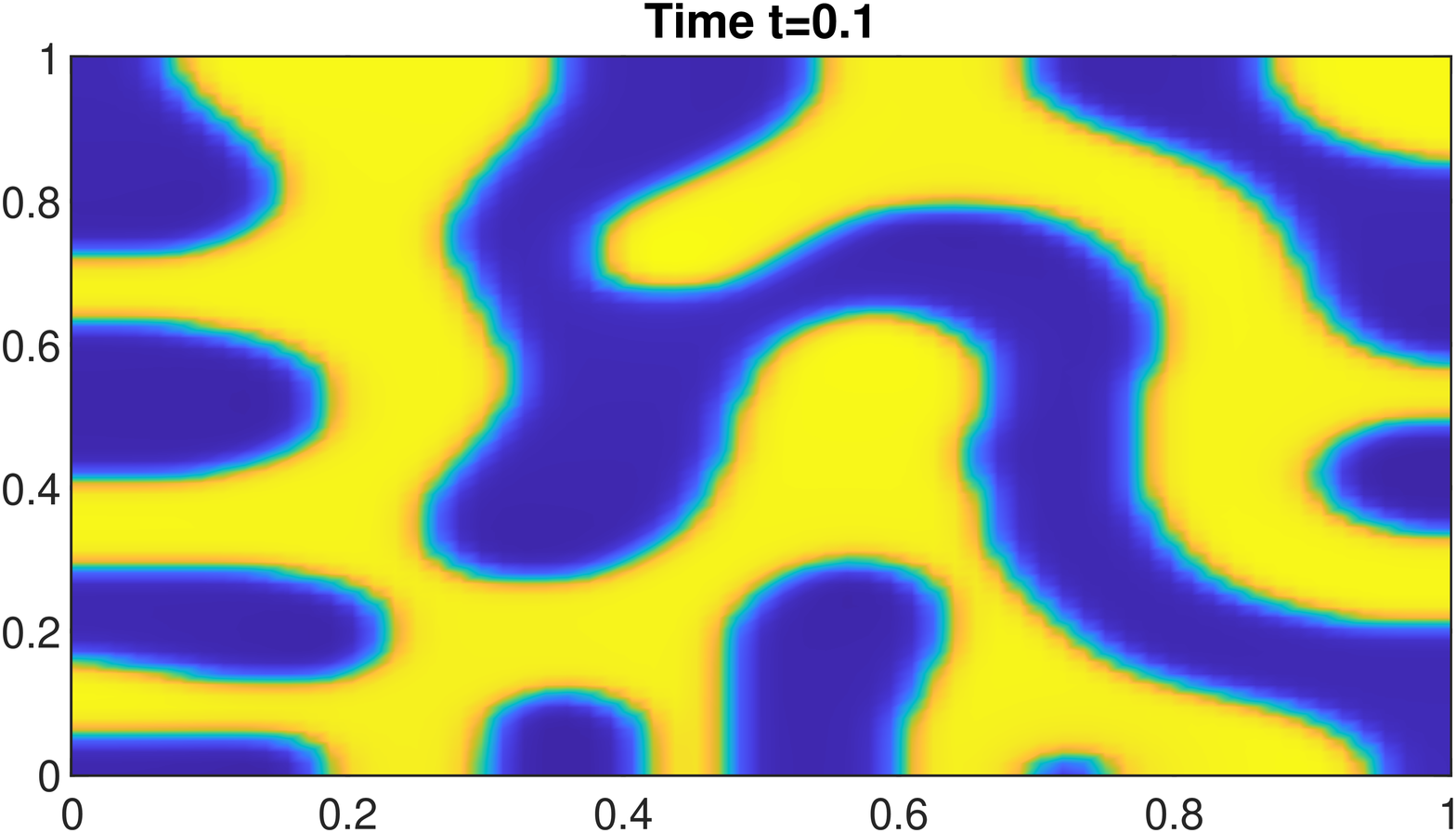}
\end{subfigure}
\caption{Evolution of solution of the CH equation at different time points using PinT-II. From phase separation to phase coarsening stage.}
\label{pint2_2d}
\end{figure}

%
%
%
\begin{figure}[h]
    \centering
    \subfloat{{\includegraphics[height=4cm,width=6cm]{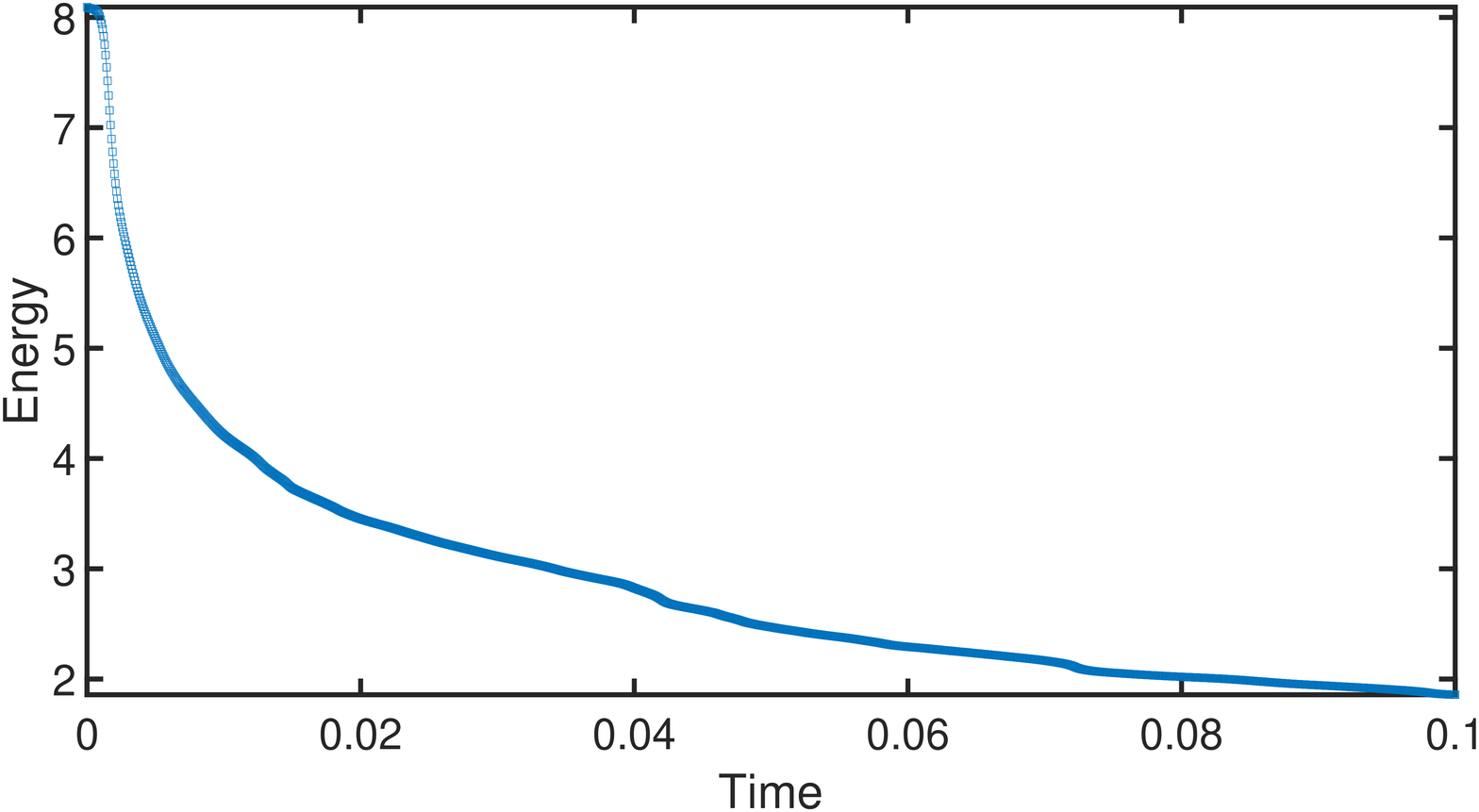} }}
    \subfloat{{\includegraphics[height=4cm,width=6cm]{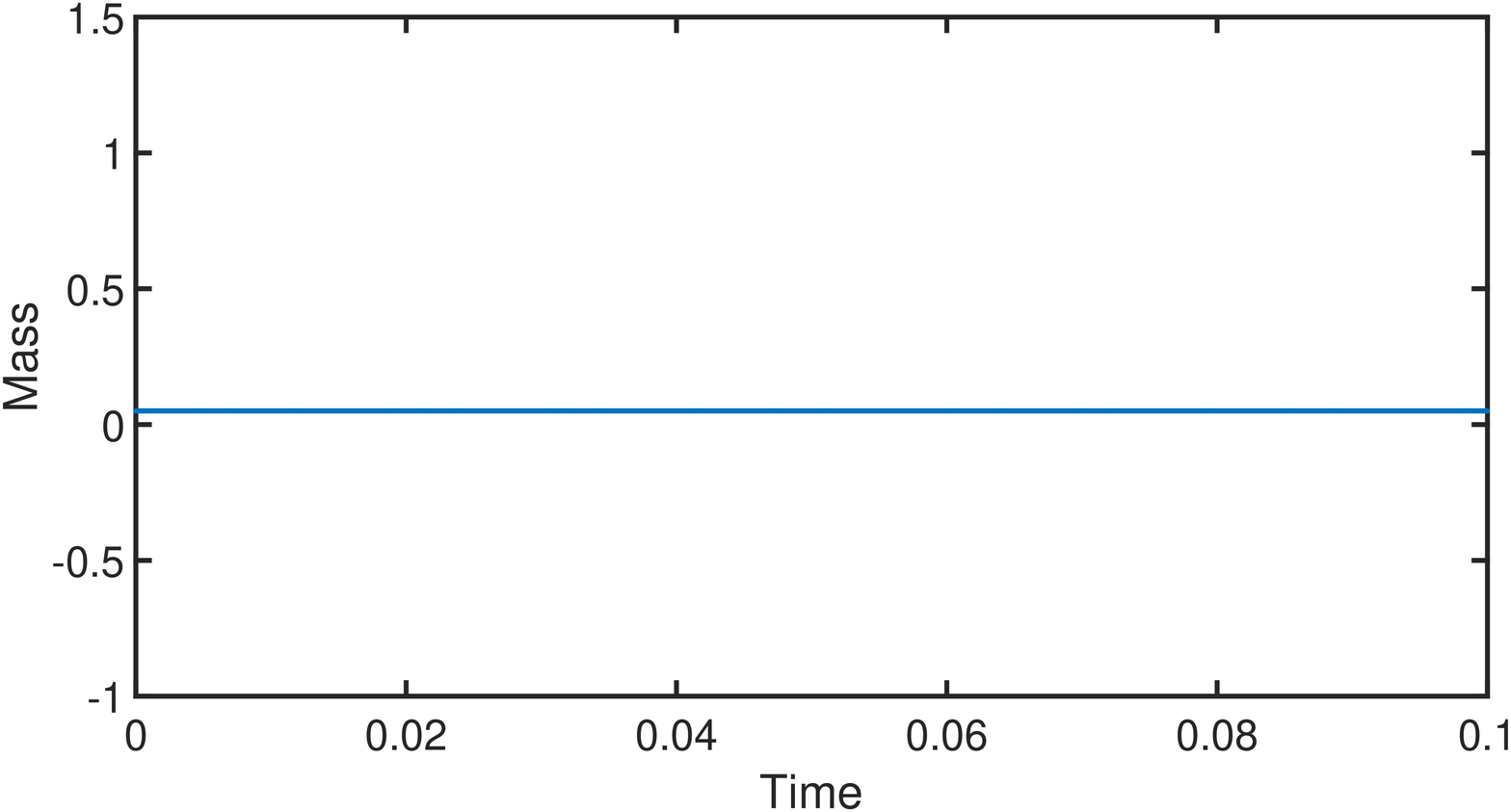} }}
    \caption{On the left: discrete energy vs time; On the right: discrete total mass vs time}
    \label{pint2_energy}
\end{figure}

\section{Conclusions}
We formulated and studied time parallel algorithms for a class of fourth order PDEs. We present rigorous convergence analysis for all the proposed PinT methods. We observe robust convergence behaviour for linear as well as nonlinear models.
\section*{Acknowledgement} The authors would like to thank the CSIR (File No:09/1059(0019)/2018-EMR-I) and DST-SERB (File No: SRG/2019/002164) for the research grant and IIT Bhubaneswar for providing excellent research environment. Authors also want to thank Prof. Yingxiang Xu for his valuable suggestions.

\bibliographystyle{siam}
\bibliography{pdbib}

\end{document}